\documentclass[fleqn]{article}

\usepackage{hyperref}
\usepackage{cleveref}
\hypersetup{hypertexnames = false, bookmarksdepth = 2, bookmarksopen = true, colorlinks, linkcolor = black, citecolor = black, urlcolor = black, pdfstartview={XYZ null null 1}}

\usepackage{amsfonts}
\usepackage{amssymb}
\usepackage{amsthm}
\usepackage{amsmath}
\usepackage{booktabs}
\usepackage{enumitem}
\usepackage{float}
\usepackage{mathtools}
\usepackage{multirow}
\usepackage{thmtools}
\usepackage{tikz}
\usepackage{tikz-cd}
\usepackage[textsize=small]{todonotes}
\usepackage{xparse}
\usepackage{xspace}
\usepackage{graphicx}

\declaretheoremstyle{article}
\theoremstyle{article}

\declaretheorem[name=Theorem]{theorem}

\newtheorem{conjecture}[theorem]{Conjecture}
\newtheorem{corollary}[theorem]{Corollary}
\newtheorem{definition}[theorem]{Definition}
\newtheorem{example}[theorem]{Example}

\newtheorem{lemma}[theorem]{Lemma}

\newtheorem{proposition}[theorem]{Proposition}
\newtheorem{remark}[theorem]{Remark}

\newcounter{todocounter}
\DeclareDocumentCommand\addreference{g}{\stepcounter{todocounter}\todo[color = blue!30]{\thetodocounter. Add reference\IfNoValueF{#1}{: #1}}\xspace}
\DeclareDocumentCommand\checkthis{g}{\stepcounter{todocounter}\todo[color = red!50]{\thetodocounter. Check this\IfNoValueF{#1}{: #1}}\xspace}
\DeclareDocumentCommand\fixthis{g}{\stepcounter{todocounter}\todo[color = orange!50]{\thetodocounter. Fix this\IfNoValueF{#1}{: #1}}\xspace}
\DeclareDocumentCommand\expand{g}{\stepcounter{todocounter}\todo[color = green!50]{\thetodocounter. Expand\IfNoValueF{#1}{: #1}}\xspace}

\mathchardef\mhyphen="2D
\newcommand\dash{\nobreakdash-\hspace{0pt}}

\newcommand\bounded{\ensuremath{\mathrm{b}}}

\crefname{conjecture}{conjecture}{conjectures}

\DeclareMathOperator\Aut{Aut}
\DeclareMathOperator\Bimod{BiMod} 
\DeclareMathOperator\BIMOD{BIMOD} 
\DeclareMathOperator\BiMod{Bimod} 
\DeclareMathOperator\bimod{bimod} 
\DeclareMathOperator\Bl{Bl}

\DeclareMathOperator\clifford{C\ell}

\DeclareMathOperator\coh{coh}
\DeclareMathOperator\coker{coker}

\DeclareMathOperator\derived{\mathbf{D}}

\DeclareMathOperator\Ext{Ext}

\DeclareMathOperator\Gm{\mathbb{G}_{\mathrm{m}}}
\DeclareMathOperator\Gr{Gr}
\DeclareMathOperator\gr{gr}
\DeclareMathOperator\HH{H}
\DeclareMathOperator\HHHH{HH}
\DeclareMathOperator\Hom{Hom}
\DeclareMathOperator\sheafHom{\mathcal{H}om}
\DeclareMathOperator\id{id}
\DeclareMathOperator\im{im}

\DeclareMathOperator\Mat{Mat}
\DeclareMathOperator\Mod{Mod}
\DeclareMathOperator\mult{mult}
\DeclareMathOperator\PGL{PGL}
\DeclareMathOperator\Pic{Pic}
\DeclareMathOperator\Proj{Proj}
\DeclareMathOperator\Qcoh{Qcoh}
\DeclareMathOperator\QGr{QGr}
\DeclareMathOperator\qgr{qgr}
\DeclareMathOperator\relSpec{\underline{Spec}}
\DeclareMathOperator\relProj{\underline{Proj}}

\DeclareMathOperator\Spec{Spec}
\DeclareMathOperator\Sym{Sym}
\DeclareMathOperator\rk{rk}
\DeclareMathOperator\tensor{T}
\DeclareMathOperator\trace{tr}
\DeclareMathOperator\Tors{Tors}

\DeclareFontFamily{U}{mathx}{\hyphenchar\font45}
\DeclareFontShape{U}{mathx}{m}{n}{
      <5> <6> <7> <8> <9> <10>
      <10.95> <12> <14.4> <17.28> <20.74> <24.88>
      mathx10
      }{}
\DeclareSymbolFont{mathx}{U}{mathx}{m}{n}
\DeclareFontSubstitution{U}{mathx}{m}{n}
\DeclareMathAccent{\widecheck}{0}{mathx}{"71}

\renewcommand\longrightarrow\rightarrow

\title{Comparison of two constructions of noncommutative surfaces with exceptional collections of length 4}
\author{Pieter Belmans \and Dennis Presotto \and Michel Van den Bergh}

\begin{document}

\maketitle

\begin{abstract}
  Recently the Euler forms on numerical Grothendieck groups of rank~4 whose properties mimick that of the Euler form of a smooth projective surface have been classified. This classification depends on a natural number~$m$, and suggests the existence of noncommutative surfaces which up to that point had not been considered for~$m\geq 2$. These have been constructed for~$m=2$ using noncommutative~$\mathbb{P}^1$\dash bundles, and for all~$m\geq 2$ by a different construction using maximal orders on~$\Bl_x\mathbb{P}^2$.

  In this article we compare the constructions for~$m=2$, i.e.~we compare the categories arising from half-ruled del Pezzo quaternion orders on~$\mathbb{F}_1$ with noncommutative~$\mathbb{P}^1$\dash bundles on~$\mathbb{P}^1$. This can be seen as a noncommutative instance of the classical isomorphism~$\mathbb{F}_1\cong\Bl_x\mathbb{P}^2$.
\end{abstract}

\tableofcontents

\section{Introduction}
Exceptional collections are a convenient way to study derived categories in (noncommutative) algebraic geometry, for (the noncommutative analogues of) smooth projective varieties. The existence of a full and strong exceptional collection is a very strong property for a triangulated category, but when it is satisfied the study of the derived category and its invariants essentially becomes the study of a (directed) finite-dimensional algebra.

One thing which sets apart the algebras obtained in this way from the class of all algebras is the behaviour of the Serre functor on the Grothendieck group: it can be shown that~$(-1)^{\dim X}\mathbb{S}$ is unipotent \cite[lemma~3.1]{MR1230966}. By \cite{MR2838062} the case of~$\dim X=1$ is trivial, the only finite-dimensional algebra that appears is the path algebra of the Kronecker quiver.

In \cite{1607.04246v1} the Serre functor in the case of a surface is studied, and some extra properties are determined: for a finitely generated free abelian group~$\Lambda$, a nondegenerate bilinear form~$\langle-,-\rangle\colon\Lambda\times\Lambda\to\mathbb{Z}$ and an automorphism~$s\in\Aut(\Lambda)$ we will ask that
\begin{description}[labelwidth=\widthof{\bfseries Serre automorphism\quad},align=right]
  \item[Serre automorphism]
    $\langle x,s(y)\rangle=\langle y,x \rangle$ for~$x,y \in \Lambda$;
  \item[unipotency]
    $s-\id_\Lambda$ is nilpotent;
  \item[rank]
    $\rk(s-\id_\Lambda)=2$;
\end{description}
Indeed, it can be shown that the action of the Serre functor on the Grothendieck group for all smooth projective surfaces (besides being unipotent) has the extra property that the rank of~$s-\id_\Lambda$ is precisely~2. Independently, in \cite{MR3691718} a closely related notion is introduced.

One can then try to classify the possible Serre functors on the numerical level, and indeed this is done for~$\Lambda$ of rank~3 and~4\ in \cite{1607.04246v1}. For rank~3 the only case is the Grothendieck group arising from~$\mathbb{P}^2$ as the Beilinson quiver (and its mutations). For rank~4 the classification is more interesting, and it is given in~\cite[theorem~A]{1607.04246v1}. The following matrices describe all possible bilinear forms (up to the action of the signed braid group corresponding to mutation of the exceptional collections):
\begin{equation}
  \tag{A}
  \label{equation:type-A}
  \begin{pmatrix}
    1 & 2 & 2 & 4 \\
    0 & 1 & 0 & 2 \\
    0 & 0 & 1 & 2 \\
    0 & 0 & 0 & 1
  \end{pmatrix}
\end{equation}
and
\begin{equation}
  \tag{$\mathrm{B}_m$}
  \label{equation:type-Bm}
  \begin{pmatrix}
    1 & m & 2m & m \\
    0 & 1 & 3 & 3 \\
    0 & 0 & 1 & 3 \\
    0 & 0 & 0 & 1
  \end{pmatrix}
\end{equation}
where~$m\in\mathbb{N}$. One observes that \eqref{equation:type-A} corresponds to~$\mathbb{P}^1\times\mathbb{P}^1$, $(\mathrm{B}_0)$ to~$\mathbb{P}^2\sqcup\mathrm{pt}$ and~$(\mathrm{B}_1)$ to~$\Bl_x\mathbb{P}^2$. Moreover in all of these cases it is possible to change the relations of the quiver, which corresponds to considering the noncommutative analogues of these varieties.

For~\eqref{equation:type-Bm} there are no smooth projective surfaces with this Euler form. But in \cite{1503.03992v4} a suitable adaptation of the theory of noncommutative~$\mathbb{P}^1$\dash bundles \cite{MR2958936} was used to construct a noncommutative surface of type~$(\mathrm{B}_2)$. Using completely different techniques the first two authors constructed noncommutative surfaces of type~\eqref{equation:type-Bm} for all~$m\geq 2$ in \cite{MR3847233}. This construction takes a maximal order on~$\mathbb{P}^2$ and pulls it back along the blowup of a point outside the ramification locus. We recall some of the details of these constructions in \cref{subsection:construction-as-bundle,subsection:construction-as-blowup}.

Hence there are two geometric constructions of an abelian category which behaves like the category of coherent sheaves on a smooth projective surface: it is of global dimension~2, and its derived category has a Serre functor which is compatible with the standard t-structure, just like in algebraic geometry.

The goal of the current paper is to compare these two constructions, i.e.~compare the abelian category arising as a noncommutative~$\mathbb{P}^1$\dash bundle on~$\mathbb{P}^1$ to the category of coherent sheaves on a maximal order on~$\Bl_x\mathbb{P}^2$. It turns out that they can be explicitly compared using the geometry of the linear systems used in their construction. The idea is to write both categories in terms of (relative) Clifford algebras, which in turn reduces the problem to comparing the data defining these algebras. This allows us to explicitly relate the input for the noncommutative~$\mathbb{P}^1$\dash bundle to the input for the maximal order construction.

In \cref{section:clifford-algebras} we recall various notions of Clifford algebra which are needed for the results in this paper. Most of these are well-known, but we also introduce a comparison result which might be of independent interest, relating the Clifford algebra with values in a relatively ample line bundles to the total Clifford algebra. In \cref{section:nc-P1-as-clifford} we consider the noncommutative~$\mathbb{P}^1$\dash bundles constructed in \cite{1503.03992v4}, and explain how we can use the formalism of generalised preprojective algebras to describe these abelian categories using Clifford algebras. In \cref{section:blowup-as-clifford} we obtain a similar description of the blowups of maximal orders on~$\mathbb{P}^2$ constructed in \cite{MR3847233} for the special case where the the fat point modules have multiplicity 2. In op.~cit.~it is also shown that these are del Pezzo orders in the sense of \cite{MR1954458}, and moreover that they are half-ruled in the sense of \cite{MR3343212}.

We compare these constructions in \cref{section:comparison}. The essential ingredient in the proof is a comparison between basepoint-free pencils of binary quartics and basepoint-free nets of conics together with the choice of a smooth conic in the net. The correspondence on the geometric level is not complete at the moment, but it is sufficient to give a comparison of the two constructions.

The comparison of the two constructions can be seen as a suitable noncommutative analogue of the isomorphism of the first Hirzebruch surface with the blowup of~$\mathbb{P}^2$ in a point. The formalism of Clifford algebras will allow us to compare Clifford algebras relative to the two projective morphisms in
\begin{equation}
  \label{equation:F1-BlxP2}
  \begin{tikzcd}
    & \mathbb{F}_1\cong\Bl_x\mathbb{P}^2 \arrow[ld, swap, "p"] \arrow[rd, "\pi"] \\
    \mathbb{P}^2 & & \mathbb{P}^1.
  \end{tikzcd}
\end{equation}
There is also a different (conjectural) noncommutative analogue of the isomorphism~$\mathbb{F}_1\cong\Bl_x\mathbb{P}^2$, which is actually a deformation of the commutative isomorphism. The two constructions in this case are the noncommutative~$\mathbb{P}^1$\dash bundles from \cite{MR2958936} versus the notion of blowup from \cite{MR1846352}. The main difference between the notions of bundles is that the one considered in this paper is of rank~$(1,4)$, where originally they were of rank~$(2,2)$. The main difference between the notions of blowing up is that we consider a point outside the point scheme (or rather the ramification), whereas the original blowup considers a point on the point scheme. We come back to the comparison of these constructions in \cref{subsection:final-remarks}.

\paragraph{Conventions}
The more general parts of this paper work over any scheme where~2~is invertible. For the actual comparison we will work with varieties over an algebraically closed field~$k$ not of characteristic~2 or~3.

\section{Clifford algebras}
\label{section:clifford-algebras}
The main tool in comparing the constructions from~\cite{1503.03992v4} and~\cite{MR3847233} is the formalism of Clifford algebras. It turns that it is possible to write both abelian categories as a category associated to a certain Clifford algebra. This is done in \cref{section:nc-P1-as-clifford} and \cref{section:blowup-as-clifford}. It then becomes possible to compare the linear algebra data describing the quadratic forms. In this way we can set up an explicit correspondence between the two models, and explicitly relate the input data for the~$\mathbb{P}^1$\dash bundle model to the input data for the blowup model. This is done in \cref{section:comparison}.

In \cref{subsection:clifford-algebras} we recall the notion of the Clifford algebra associated to a quadratic form on a scheme. This is a coherent~$\mathbb{Z}/2\mathbb{Z}$\dash graded sheaf of algebras. In \cref{subsection:clifford-algebras-with-values} we recall the more general version where the quadratic form is allowed to take values in a line bundle \cite{MR1260714,MR1220425}. In this case we cannot obtain a~$\mathbb{Z}/2\mathbb{Z}$\dash algebra, but the even part of the Clifford algebra is nevertheless well-defined.

In \cref{subsection:graded-clifford-algebras} we recall the notion of the graded Clifford algebra associated to a linear system of quadrics \cite{MR0814174,MR1356364}. Provided the linear system is basepoint-free we get an Artin--Schelter regular algebra which is finite over its center. 

Finally, we can generalise the notion of a graded Clifford algebra to the relative setting for projective morphisms. This is done in \cref{subsection:clifford-algebras-with-ample-values}, by combining \cref{subsection:clifford-algebras-with-values} and \cref{subsection:graded-clifford-algebras}. The classical case of a graded Clifford algebra then corresponds to the morphism~$\Proj\Sym_k(E)\to\Spec k$, where~$E$ is the vector space spanned by the elements of degree~1.

\subsection{Clifford algebras}
\label{subsection:clifford-algebras}
Classically the Clifford algebra is a finite-dimensional~$k$\dash algebra associated to a quadratic form~$q\colon E\to k$ on a vector space~$E$. It is defined as the quotient of the tensor algebra~$\tensor(E)$ by the two-sided ideal generated by
\begin{equation}
  v\otimes v-q(v)
\end{equation}
for~$v\in E$. We will use the notation~$\clifford_k(E,q)$ for the Clifford algebra. Under the standing assumption that the characteristic of~$k$ is not two we can also consider the symmetric bilinear form~$b_q$ associated to~$q$, and rewrite the relation as
\begin{equation}
  u\otimes v+v\otimes u-2b_q(u,v).
\end{equation}
The correspondence between symmetric bilinear forms and quadratic forms is an instance of the isomorphism
\begin{equation} \label{equation:comparison-of-sym}
  \Sym^2E\cong\Sym_2E
\end{equation}
Here~$\Sym_2E \subset E \otimes E$ is the subspace of all symmetric tensors, i.e. as a vectorspace it is generated by elements of the form~$u \otimes v + v \otimes u$. Conversely~$\Sym^2E$ is the quotient of~$E \otimes E$ by all relations of the form~$u \otimes v - v \otimes u$. The isomorphism in \eqref{equation:comparison-of-sym} is then given by
\[ \overline{u \otimes v} \mapsto \frac{u \otimes v + v \otimes u}{2} \]
This isomorphism will be used in a more general form on schemes later.

One can similarly define a Clifford algebras over arbitrary commutative rings
\begin{definition}
  \label{definition:clifford-algebra}
  Let~$A$ be a commutative algebra and~$E$ a finitely generated projective~$A$\dash module. Moreover let~$q\colon E\rightarrow A$ be a quadratic form, i.e.~(up to the identification of~$q$ and~$b_q$)~$q$ is a symmetric~$A$\dash linear bilinear form
  \begin{equation}
     q\colon\Sym_A^2(E)\longrightarrow A
  \end{equation}
  Then the Clifford algebra~$\clifford_A(E,q)$ is the quotient of~$\tensor_A(E)$ by the relations
  \begin{equation}
     v \otimes w + w \otimes v - q(vw)
  \end{equation}
\end{definition}

Observe that the obvious~$\mathbb{Z}$\dash grading on~$\tensor(V)$ equips~$\clifford(V,q)$ with a~$\mathbb{Z}/2\mathbb{Z}$\dash grading. The even degree part forms a subalgebra which we will denote~$\clifford(V,q)_0$, and the odd degree part forms the~$\clifford(V,q)_0$\dash bimodule~$\clifford(V,q)_1$.

We want to globalise these constructions to quadratic forms over schemes. First we have to say what we mean by a quadratic form on a scheme~$X$.
\begin{definition}
  Let~$X$ be a scheme such that~$2$ is invertible on~$X$. A \emph{quadratic form on~$X$} is a pair~$(\mathcal{E},q)$ where~$\mathcal{E}$ is a locally free sheaf and~$q\colon\Sym_2(\mathcal{E})\to\mathcal{O}_X$ is a morphism of~$\mathcal{O}_X$\dash modules, and~$\Sym_2\mathcal{E}$ denotes the submodule of symmetric tensors inside~$\tensor^2(\mathcal{E})$.
\end{definition}
Let us denote the symmetric square as~$\Sym^2\mathcal{E}$, i.e.~this is the quotient of~$\tensor^2\mathcal{E}$ by the relation~$v\otimes w-w\otimes v$. As in \eqref{equation:comparison-of-sym}, using the assumption that~$2$ is invertible on~$X$, there exists an isomorphism~$\Sym_2\mathcal{E}\cong\Sym^2\mathcal{E}$, relating quadratic forms to symmetric bilinear forms. We will from now on identify both sheaves and consider quadratic forms as morphisms from~$\Sym^2\mathcal{E}$.

We can then define what we mean by a Clifford algebra associated to such a quadratic form.
\begin{definition}
  \label{definition:clifford}
  Let~$X$ be a scheme. Let~$(\mathcal{E},q)$ be a quadratic form over~$X$. The \emph{Clifford algebra}~$\clifford_X(\mathcal{E},q)$ is the sheafification of
  \begin{equation}
    U\mapsto\tensor_{\mathcal{O}_X(U)}(\mathcal{E}(U))/(v\otimes v-q(v)1\mid v\in\HH^0(U,\mathcal{E})).
  \end{equation}
\end{definition}
From the grading on the tensor algebra we see that~$\clifford_X(\mathcal{E},q)$ is a~$\mathbb{Z}/2\mathbb{Z}$\dash graded algebra, as was the case for~$\clifford(V,q)$. In order to be compatible with the notation from \cref{subsection:clifford-algebras-with-values} we will sometimes denote this Clifford algebra as~$\clifford_X(\mathcal{E},q,\mathcal{O}_X)$.

\begin{remark}
  Now in the case that~$q$ is nondegenerate we get that~$\clifford_X(\mathcal{E},q)$ is \emph{strongly}~$\mathbb{Z}/2\mathbb{Z}$\dash graded, and~$\clifford_X(\mathcal{E},q)_1$ is an invertible sheaf of~$\clifford_X(\mathcal{E},q)_0$\dash bimodules. See also \cite[example~1.1.24]{MR3523984}.
\end{remark}

Similar to \cref{definition:clifford-algebra} we can extend this definition by replacing~$\mathcal{O}_X$ by a sheaf of commutative algebras~$\mathcal{A}$ as follows.
\begin{definition}
  \label{definition:clifford-algebra-sheaf}
  Let~$X$ be a scheme and~$\mathcal{A}$ a sheaf of commutative algebras on~$X$. Let~$\mathcal{E}$ be a locally free~$\mathcal{A}$-module and let~$q\colon\Sym_{\mathcal{A}}^2(\mathcal{E})\to\mathcal{A}$ be a morphism of~$\mathcal{A}$\dash modules. The \emph{Clifford algebra}~$\clifford_\mathcal{A}(\mathcal{E},q)$ is the sheafification of
  \begin{equation}
    U\mapsto\tensor_{\mathcal{A}(U)}(\mathcal{E}(U))/(v\otimes_{\mathcal{A}(U)} v-q(v)1\mid v\in\HH^0(U,\mathcal{E})).
  \end{equation}
\end{definition}
We will use this construction in \cref{remark:graded-clifford-is-clifford,subsection:clifford-algebras-with-ample-values}, where~$\mathcal{A}$ will be a sheaf of \emph{graded} algebras concentrated in even degree. In this case we give~$\mathcal{E}$ degree~1, and then~$\clifford_{\mathcal{A}}(\mathcal{E},q)$ will be a~$\mathbb{Z}$\dash graded algebra.

\subsection{Clifford algebras with values in line bundles}
\label{subsection:clifford-algebras-with-values}
We would also like to consider morphisms of the form~$\Sym^2(\mathcal{E})\to\mathcal{L}$, where~$\mathcal{L}\not\cong\mathcal{O}_X$ is a line bundle on~$X$. In this case we cannot mimick the construction of \cref{subsection:clifford-algebras} to produce a~$\mathbb{Z}/2\mathbb{Z}$\dash graded Clifford algebra for this more general situation. But an analogue of the even part of the Clifford algebra \emph{can} be defined, together with a bimodule over this algebra \cite{MR1260714,MR1220425} which generalises the odd part of the usual Clifford algebra.

The generalisation of quadratic forms taking values in a line bundle is the following.
\begin{definition}
  Let~$X$ be a scheme such that~$2$ is invertible on~$X$. Let~$\mathcal{L}$ be an invertible sheaf. An \emph{$\mathcal{L}$\dash valued quadratic form on~$X$} is a triple~$(\mathcal{E},q,\mathcal{L})$ where~$\mathcal{E}$ is a locally free sheaf and~$q\colon\Sym^2(\mathcal{E})\to\mathcal{L}$ is a morphism of~$\mathcal{O}_X$\dash modules.
\end{definition}

Let~$X$ be a scheme. Let~$(\mathcal{E},q,\mathcal{L})$ be an~$\mathcal{L}$\dash valued quadratic form on~$X$. We can consider the total space of~$\mathcal{L}$ as
\begin{equation}
  p\colon Y\coloneqq\relSpec_X\left( \bigoplus_{n\in\mathbb{N}}\mathcal{L}^{\otimes n} \right)\to X.
\end{equation}
Because~$p^*\mathcal{L}\cong\mathcal{O}_Y$ the quadratic form~$(p^*\mathcal{E},p^*q)$ takes values in~$\mathcal{O}_Y$, so we are in the situation of the previous section.

\begin{definition}
  The sheaf of algebras~$p_*(\clifford_Y(p^*\mathcal{E},p^*q))$ on~$X$ is called the \emph{generalised Clifford algebra} or \emph{total Clifford algebra}. As~$p$ is an affine morphism,~$p_*$ is just the forgetful functor, and the structure as~$\mathcal{O}_Y$\dash module induces a~$\mathbb{Z}$\dash grading on the total Clifford algebra.

  The \emph{even Clifford algebra}~$\clifford_X(\mathcal{E},q,\mathcal{L})_0$ is the degree~0~subalgebra of~$p_*(\clifford_Y(p^*\mathcal{E},p^*q))$.

  The \emph{Clifford module}~$\clifford_X(\mathcal{E},q,\mathcal{L})_1$ is the degree~1~submodule of the total Clifford algebra considered as a bimodule over~$\clifford_X(\mathcal{E},q,\mathcal{L})_0$.
\end{definition}
It is not possible to combine these two pieces into a~$\mathbb{Z}/2\mathbb{Z}$\dash graded algebra. Rather we get that the bimodule structure gives rise to a multiplication map
\begin{equation}
  \clifford_X(\mathcal{E},q,\mathcal{L})_1\otimes_{\clifford_X(\mathcal{E},q,\mathcal{L})_0}\clifford_X(\mathcal{E},q,\mathcal{L})_1\to\clifford(\mathcal{E},q,\mathcal{L})_0\otimes_{\mathcal{O}_X}\mathcal{L}.
\end{equation}
These construction satisfy the same pleasant properties as the Clifford algebra from \cref{subsection:clifford-algebras}. In particular we will use that it is compatible with base change \cite[lemma~3.4]{MR1260714}.

\subsection{Graded Clifford algebras}
\label{subsection:graded-clifford-algebras}
In \cite{MR1260714,MR1220425} there is a construction of connected graded algebras based on linear algebra input. For sufficiently general choices it gives rise to an interesting class of Artin--Schelter regular algebras of arbitrary dimension which are finite over their center.

\begin{definition}
  \label{definition:graded-clifford-algebra}
  Let~$M_1,\dotsc,M_n$ be symmetric matrices in~$\Mat_n(k)$. The \emph{graded Clifford algebra} associated to~$(M_i)_{i=1}^n$ is the quotient of the graded free~$k$\dash algebra~$k\langle x_1,\dotsc,x_n,y_1,\dotsc,y_n\rangle$, where~$|x_i|=1$ and~$|y_i|=2$, by the relations
  \begin{enumerate}
    \item $x_ix_j+x_jx_i=\sum_{m=1}^n(M_m)_{i,j}y_m$, where~$i,j=1,\dots,n$
    \item $[x_i,y_j]=0$ and~$[y_i,y_j]=0$, where~$i,j=1,\dotsc,n$, i.e.~$y_i$ is central.
  \end{enumerate}
\end{definition}
To understand the properties of this graded Clifford algebra, we have to interpret the matrices~$M_1,\dotsc,M_n$ as quadratic forms~$Q_i$ (in the variables~$y_i$). In this way we obtain~$n$ quadric hypersurfaces in~$\mathbb{P}^{n-1}$, which span a linear system. There are many different ways in which the geometry of the linear system of quadrics influences the algebraic and homological properties of these Clifford algebras.

In particular, the graded Clifford algebra is Artin--Schelter regular if and only if the linear system of quadrics is basepoint-free \cite[proposition~7]{MR1356364}. It is also possible to describe the (fat) point modules \cite[proposition~9]{MR1356364}. We can consider the matrix~$M=\sum_{i=1}^nM_iy_i$, and at a point~$p\in\mathbb{P}_{y_1,\dotsc,y_n}^{n-1}$ we can consider the rank of the matrix~$M(p)$. Generically it is of full rank, and the point modules correspond to those points for which the rank is~1 or~2. We are interested in the case where~$n=3$, so the point modules are given by the determinant of~$M$, which describes a cubic curve inside~$\mathbb{P}^2$. We get back to this in \cref{subsection:quaternionic-noncommutative-planes}.

\begin{remark}
  \label{remark:graded-clifford-is-clifford}
  Although \cref{definition:graded-clifford-algebra} is at first sight quite different from \cref{definition:clifford-algebra}, we would like to mention that every graded Clifford algebra is in fact a Clifford algebra. To see this, let~$\clifford(M)$ be the graded Clifford algebra defined by~$n$~symmetric~$n \times n$ matrices~$M_m$. Then there is an isomorphism
  \begin{equation}
    \clifford(M) \cong \clifford_A(E,q)
  \end{equation}
  by setting
  \begin{equation}
    \begin{aligned}
       A &= k[y_1, \ldots, y_n] \\
       F &= kx_1 \oplus \ldots \oplus kx_n \\
       E &= F \otimes_k A \\
    \end{aligned}
  \end{equation}
  and considering the quadratic form
  \begin{equation}
    \begin{aligned}
      q_F\colon \Sym^2_kF &\rightarrow ky_1 \oplus \ldots \oplus ky_n \\
      x_ix_j + x_jx_i &\mapsto \sum_{m=1}^n (M_m)_{i,j} y_m \\
    \end{aligned}
  \end{equation}
  and extending it as
  \begin{equation}
    q\colon \Sym^2_AE = \Sym^2_kF \otimes_k A \xrightarrow{q_F \otimes \id_A} (ky_1 \oplus \ldots \oplus ky_n) \otimes_k A = A_2 \otimes_k A \rightarrow A
  \end{equation}
  In the next section we continue this idea of interpreting Clifford algebras in two ways.
\end{remark}

\subsection{Clifford algebras with values in ample line bundles}
\label{subsection:clifford-algebras-with-ample-values}
Consider a projective morphism~$f\colon Y\to X$, and let~$\mathcal{L}$ be an~$f$\dash relatively ample line bundle on~$Y$. In this case we can consider the graded~$\mathcal{O}_X$\dash algebra
\begin{equation}
  \label{equation:algebra-A}
  \mathcal{A}\coloneqq\bigoplus_{n\geq 0}f_*(\mathcal{L}^{\otimes n})
\end{equation}
and we have that the relative Proj recovers~$Y$,~i.e.~$Y\cong\relProj_X\mathcal{A}$.

Let~$\mathcal{E}$ be a vector bundle on~$X$. Consider a quadratic form
\begin{equation}
  q\colon\Sym_Y^2(f^*\mathcal{E})\to\mathcal{L}
\end{equation}
We can associate two Clifford algebras to this quadratic form:
\begin{enumerate}
  \item Using \cref{subsection:clifford-algebras-with-values} there is the sheaf of even Clifford algebras~$\clifford_0(f^*(\mathcal{E},q,\mathcal{L})$, which is a coherent sheaf of algebras on~$Y$.

  \item On the other hand, we can use the algebra~$\mathcal{A}$ from \eqref{equation:algebra-A} and consider it with a doubled grading,~i.e.
    \begin{equation}
      \mathcal{A}_n=
      \begin{cases}
        f_*(\mathcal{L}^{\otimes m}) & n=2m \\
        0 & n=2m+1.
      \end{cases}
    \end{equation}
    Changing the grading in this way does not change the property that~$Y$ is the relative Proj of~$\mathcal{A}$.

    Because~$f^*$ is monoidal and using the adjunction~$f^*\dashv f_*$ we can consider the quadratic form~$q$ as a morphism~$\Sym_X^2\mathcal{E}\to f_*(\mathcal{L})=\mathcal{A}_2$. We can then extend this morphism to the quadratic form
    \begin{equation}
      Q\colon\Sym_{\mathcal{A}}^2(\mathcal{E}\otimes_{\mathcal{O}_X}\mathcal{A})\to\mathcal{A}.
    \end{equation}

    Using \cref{subsection:clifford-algebras} we can define the Clifford algebra~$\clifford_{\mathcal{A}}(\mathcal{E}\otimes_{\mathcal{O}_X}\mathcal{A},Q)$, which we can consider as a sheaf of graded~$\mathcal{A}$\dash algebras on~$X$.
\end{enumerate}
We can compare these two constructions as follows.
\begin{proposition}
  \label{proposition:ample-comparison}
  With the notation from above, the coherent sheaf of algebras on~$Y$ associated to~$\clifford_{\mathcal{A}}(\mathcal{E}\otimes_{\mathcal{O}_X}\mathcal{A},Q)$ is isomorphic to the sheaf of even Clifford algebras~$\clifford_0(f^*\mathcal{E},q,\mathcal{L})$.

  Moreover, the coherent sheaf on~$Y$ associated to~$\clifford_{\mathcal{A}}(\mathcal{E}\otimes_{\mathcal{O}_X}\mathcal{A},Q)(1)$ is isomorphic to the Clifford bimodule~$\clifford_1(f^*\mathcal{E},q,\mathcal{L})$.

  \begin{proof}
    We have by the relative version of Serre's theorem that~$\Qcoh Y\cong\QGr_X\mathcal{A}$. Then~$f^*\mathcal{E}$ on~$Y$ is given by the graded~$\mathcal{A}$\dash module~$\mathcal{E}\otimes_{\mathcal{O}_X}\mathcal{A}$. It now suffices to observe that by the characterisation of relatively ample line bundles \cite[tag~01VJ]{stacks-project} we can compare the two constructions of Clifford algebras, using \cite[lemmas~3.1 and~3.2]{MR1260714}.

    The statement about the Clifford module follows from the fact that~$\mathcal{A}$ is concentrated in even degree, and \cite[lemma~3.1]{MR1260714}.
  \end{proof}
\end{proposition}

\begin{example}
  As a special case of this construction we recover the graded Clifford algebra from \cref{subsection:graded-clifford-algebras}. To see this, we consider the graded Clifford algebra as an algebra over~$\mathbb{C}[y_1,\dotsc,y_n]$. If we denote~$E$ the vector space spanned by the~$y_i$'s, we have a morphism~$\mathbb{P}(E)\to\Spec k$, and the graded Clifford algebra is nothing but the Clifford algebra associated to~$(E\otimes_k\mathcal{O}_{\mathbb{P}(E)},q,\mathcal{O}_{\mathbb{P}(E)}(1))$, where~$q$ describes the linear system of quadrics as symmetric matrices in the global sections~$E\otimes E^\vee$.
\end{example}

\section{Noncommutative \texorpdfstring{$\mathbb{P}^1$}{P1}-bundles as Clifford algebras}
\label{section:nc-P1-as-clifford}
As mentioned in the introduction the goal of this paper is to compare two constructions of noncommutative surfaces of (numerical) type~$m=2$ of \eqref{equation:type-Bm} which was described above. We will prove that both categories are equivalent to (a Serre quotient of) the module category over a sheaf of graded Clifford algebras on~$\mathbb{P}^1$. In this section we describe how the noncommutative~$\mathbb{P}^1$\dash bundle obtained in~\cite[theorem~5.1]{1503.03992v4} is equivalent to the module category over a sheaf of graded Clifford algebras on~$\mathbb{P}^1$.

In \cref{subsection:construction-as-bundle} we quickly recall this construction as a noncommutative~$\mathbb{P}^1$\dash bundle, or rather the associated module category~$\QGr(\mathbb{S}({}_f (\mathcal{O}_{\mathbb{P}^1})_{\id}))$.

In \cref{subsection:twisting} we recall the technicalities in the construction of noncommutative~$\mathbb{P}^1$\dash bundles as in \cite{1503.03992v4} in more detail. We also recall the notion of \emph{twisting} for sheaf\dash $\mathbb{Z}$-\dash algebras, an operation which induces equivalences of categories at the level of~$\Gr$ and~$\QGr$. We use this twist operation to show that~$\QGr(\mathbb{S}({}_f (\mathcal{O}_{\mathbb{P}^1})_{\id}))$ is equivalent to~$\QGr(\mathrm{H}(\mathbb{P}^1/\mathbb{P}^1))$ for a sheaf of graded algebras~$\mathrm{H}(\mathbb{P}^1/\mathbb{P}^1)$ associated to the finite morphism~$f\colon\mathbb{P}^1 \rightarrow \mathbb{P}^1$.

In \cref{subsection:symmetric-sheaves} we then show that~$\mathrm{H}(\mathbb{P}^1/\mathbb{P}^1)$ is a \emph{symmetric sheaf of graded algebras}. In the particular case that~$f$ has degree~4, this implies using \cref{lemma:H4-clifford} that~$\mathrm{H}(\mathbb{P}^1/\mathbb{P}^1)$ is isomorphic to a sheaf of graded Clifford algebras on~$\mathbb{P}^1$.

The proofs of both \cref{lemma:HYX-symmetric} and \cref{lemma:H4-clifford} are based on local computations: it was shown in \cite[\S 3]{1503.03992v4} that noncommutative~$\mathbb{P}^1$\dash bundles can be studied locally using the theory of generalized preprojective algebras~$\Pi_C(D)$ associated to a relative Frobenius pair~$D/C$ of finite rank as in \cite{MR3565443}. We recall this theory in \cref{subsection:veronese-clifford-comparison}. In this section we also show that~$\QGr(\Pi_C(D))$ is equivalent to~$\QGr(\mathrm{H}(D/C))$ where~$\mathrm{H}(D/C)$ is a graded algebra with symmetric relations. The theory in \cref{subsection:symmetric-relations-clifford-comparison} then shows  that~$\mathrm{H}(D/C)$ maps surjectively onto a graded Clifford algebra~$\clifford_A(E,q)$. In the particular case of a relative Frobenius pair~$D/C$ of rank~4 (these are the only pairs we encounter locally when~$f$ has degree 4) this map is in fact an isomorphism by \cref{proposition:veronese-to-clifford-iso}.

We can summarise the steps in the comparison as follows.
\begin{equation*}
  \begin{array}{cr}
    \qgr \mathbb{S}({}_f(\mathcal{O}_{\mathbb{P}^1})_{\id}) \\
    \big\downarrow & \text{\cref{lemma:twist-2}} \\
    \qgr\Pi(Y/X) \\
    \big\downarrow \\
    \qgr\Pi(Y/X)^{(2)} \\
    \big\downarrow & \eqref{equation:2-veronese-PiYX} \\
    \qgr\mathrm{H}(Y/X) \\
    \big\downarrow & \text{\cref{corollary:H-is-clifford}} \\
    \qgr\left( \clifford_{\Sym(\mathcal{O}_{\mathbb{P}^1} \oplus \mathcal{O}_{\mathbb{P}^1}(1))}\big(\mathcal{O}_{\mathbb{P}^1}^{\oplus 3} \otimes_{\mathcal{O}_{\mathbb{P}^1}} \Sym(\mathcal{O}_{\mathbb{P}^1} \oplus \mathcal{O}_{\mathbb{P}^1}(1)), q\big) \right) \\
    \big\downarrow & \text{\cref{remark:q-global}} \\
    \qgr\left( \clifford_{\Sym(\mathcal{O}_{\mathbb{P}^1} \oplus \mathcal{O}_{\mathbb{P}^1}(1))}\big(E \otimes_k \Sym(\mathcal{O}_{\mathbb{P}^1} \oplus \mathcal{O}_{\mathbb{P}^1}(1)), q\big) \right)
  \end{array}
\end{equation*}

\subsection{Construction of the surface as a noncommutative bundle}
\label{subsection:construction-as-bundle}
In \cite{MR2958936} a notion of noncommutative~$\mathbb{P}^1$\dash bundles was introduced, in order to describe noncommutative Hirzebruch surfaces. This is done by defining a suitable notion of a noncommutative symmetric algebra for a locally free sheaf, where the left and right structures do not necessarily agree but where the rank is~2 on both sides.

In \cite{1503.03992v4} this construction was modified to define a noncommutative symmetric algebra where the rank is~1 on the left and~4 on the right. In particular it provides noncommutative~$\mathbb{P}^1$\dash bundles on~$\mathbb{P}^1$ as categories~$\QGr(\mathbb{S}({}_f \mathcal{L}_{\id}))$ where
\begin{enumerate}
  \item $\mathcal{L}$ is a line bundle on~$\mathbb{P}^1$,
  \item $f\colon\mathbb{P}^1 \rightarrow \mathbb{P}^1$ is a finite morphism of degree~4,
  \item $\mathbb{S}({}_f \mathcal{L}_{\id})$ is the \emph{symmetric sheaf~$\mathbb{Z}$\dash algebra} associated to the bimodule~${}_f \mathcal{L}_{\id}$
\end{enumerate}
One is referred to \cref{subsection:twisting} for the exact definition of this object.

Such a noncommutative~$\mathbb{P}^1$\dash bundle on~$\mathbb{P}^1$ has ``morphisms'' to two copies of~$\mathbb{P}^1$ (see \cite[\S4]{1503.03992v4}). This means that there are functors~$\Pi_{0,*}, \Pi_{1,*}\colon \QGr(\mathbb{S}({}_f \mathcal{L}_{\id})) \rightarrow \Qcoh(\mathbb{P}^1)$ with left adjoints denoted by~$\Pi_{0}^*, \Pi_{1}^*$. It is then shown in \cite[theorem~5.1]{1503.03992v4} that one can use these functors to lift the Beilinson exceptional sequence on~$\mathbb{P}^1$ to a full and strong exceptional sequence
\begin{equation}
  \label{equation:bundle-collection}
  \derived^\bounded(\qgr(\mathbb{S}({}_f \mathcal{L}_{\id})))
  =
  \left\langle
    \Pi_{1}^* \mathcal{O}_{\mathbb{P}^1},
    \Pi_{1}^*\mathcal{O}_{\mathbb{P}^1}(1),
    \Pi_{0}^*\mathcal{O}_{\mathbb{P}^1},
    \Pi_{0}^*\mathcal{O}_{\mathbb{P}^1}(1)
  \right\rangle
\end{equation}
in~$\derived^\bounded(\qgr(\mathbb{S}({}_f \mathcal{L}_{\id}))$. Moreover there is an explicit formula (\cite[theorem~4.1]{1503.03992v4}) for computing~$\Ext$\dash spaces of the form~$\Ext^n(\Pi^*_i \mathcal{F}, \Pi^*_j \mathcal{G})$. In particular for~$\mathcal{L} = \mathcal{O}_{\mathbb{P}^1}$ one finds that the Gram matrix for the above exceptional sequence is given by
\begin{equation}
  \begin{pmatrix}
    1 & 2 & 4 & 5 \\
    0 & 1 & 0 & 1 \\
    0 & 0 & 1 & 2 \\
    0 & 0 & 0 & 1 \\
  \end{pmatrix}.
\end{equation}
It was shown in~\cite[\S6]{1607.04246v1} that~$\mathrm{B}_2$ is mutation equivalent to this matrix.

\subsection{Graded algebras with symmetric relations}
\label{subsection:symmetric-relations-clifford-comparison}
Let~$C$ be a commutative ring in which 2 is invertible. Let~$F$ be a finitely generated, projective~$C$\dash module and let~$R$ be a direct summand of~$\Sym_C^2(F)$. As we assumed~$2 \in C$ to be invertible, we can use the isomorphism in \eqref{equation:comparison-of-sym} to view~$R$ as a submodule of~$\Sym_2(F) \subset \tensor_C^2(F)$. As such we can consider the algebra~$\tensor_CF / (R)$. The goal of this section is to show that this algebra is isomorphic to a Clifford algebra.

We construct the Clifford algebra~$\clifford_A(E,q)$ as in \cref{definition:clifford}. For this let~$Q = \Sym_C^2(F)/ R$,~$A = \Sym_C Q$ and~$E = F \otimes_C A$. The quotient map~$\tilde{q}\colon\Sym_C^2(F)\rightarrow Q$ induces a map:
\begin{equation}
  q\colon \Sym_A^2(E) = \Sym_C^2(F) \otimes_C A \rightarrow Q \otimes_C A \overset{\mu}{\rightarrow} A
\end{equation}
where~$\mu$ is given by multiplication in~$A$. Moreover~$A$ can be considered as a~$\mathbb{Z}$\dash graded algebra by giving~$Q$ degree 2.

\begin{remark}
  \label{remark:exterior-product}
  Giving~$A$ degree~0 and~$E$ degree 1 equips~$\clifford_A(E,q)$ with a filtration such that
  \begin{equation}
     \gr^\bullet \clifford_A(E,q)\cong\bigwedge\nolimits^\bullet E
  \end{equation}
  as graded~$A$\dash modules.
\end{remark}

\begin{lemma}
  \label{lemma:symmetric-relations-to-clifford}
  With the above notation there exists a morphism
  \begin{equation}
    \label{equation:symmetric-relations-to-clifford}
    \varphi\colon\tensor_C(F)/(R)\rightarrow \clifford_A(E,q).
  \end{equation}

  \begin{proof}
    The inclusions~$C \hookrightarrow A$ and~$F \hookrightarrow E$ induce a morphism
    \begin{equation}
      \tensor_C(F) \rightarrow \tensor_A(E) \rightarrow  \clifford_A(E,q).
    \end{equation}
    It hence suffices to show that this morphism factors through~$\tensor_C(F)/(R)$, i.e.~that the above map sends~$R$ to~$0$. For this note that~$R$ is generated by elements of the form~$v \otimes w + w \otimes v$. The image of such an element in~$\clifford_A(E,q)$ is given by~$q(vw)$. And~$q(vw)=0$ for~$v \otimes w + w \otimes v \in R$ because~$q$ was obtained as an~$A$\dash linear extension of the quotient map~$\tilde{q}$.
  \end{proof}
\end{lemma}

\begin{proposition}
  \label{proposition:map-to-clifford-is-iso} 
  The morphism \eqref{equation:symmetric-relations-to-clifford} is always an epimorphism. 

  \begin{proof}
    It suffices to show that~$A \subset \clifford_A(E,q)$ and~$E \subset \clifford_A(E,q)$ lie in the image of~$\varphi$. By construction~$C$ and~$F$ lie in the image of~$\varphi$ and~$E = F \otimes_C A$. In order to prove surjectivity of~$\varphi$, we hence only need to show that~$A$ lies in the image of~$\varphi$. Finally, as~$A$ is~$\mathbb{N}$\dash graded and generated in degree~2, it suffices to show that~$A_0$ and~$A_2$ lie in the image of~$\varphi$. For~$A_0$ this is obvious as~$A_0 = C$. For~$A_2$ this follows by noticing that~$A_2 = q(F \otimes F)$.


  \end{proof}
\end{proposition}

\subsection{Generalised preprojective algebras}
\label{subsection:veronese-clifford-comparison}
In this section we introduce generalized preprojective algebras as in \cite{MR3565443}. First we introduce some short hand notation.
\begin{definition}
  Let~$C$ be a commutative ring and~$D$ a commutative~$C$\dash algebra, then we define the tensor algebra~$\tensor(D/C)$ via
  \begin{equation}
    \tensor(D/C) \coloneqq \tensor_{C \oplus D}( {}_CD_D \oplus {}_DD_C)
  \end{equation}
  where~${}_CD_D$ is a~$C \oplus D$-bimodule by letting~$C$ act on the left,~$D$ act on the right and all other actions being trivial. Similarly~${}_DD_C$ is a~$C \oplus D$-bimodule. In particular~${}_CD_D \otimes_{C \oplus D} {}_CD_D = {}_DD_C \otimes_{C \oplus D} {}_DD_C = 0$.
\end{definition}

\begin{definition}
  Let~$C$ be a commutative ring and~$D$ a commutative~$C$\dash algebra. Then~$D/C$ is called a \emph{relative Frobenius extension of rank~$n$} if the following conditions are satified:
  \begin{itemize}
    \item $D$ is a free~$C$\dash module of rank~$n$,
    \item there exists an isomorphism~$\varphi\colon\Hom_C(D,C) \rightarrow D$ of~$D$\dash (bi)modules.
  \end{itemize}
\end{definition}

\begin{remark}
  As mentioned in \cite[remark~1.2]{1503.03992v4} one can also define relative Frobenius pairs~$D/C$ where~$D$ is projective over~$C$. Most of the results in op.~cit.~can be lifted to the projective setting by suitably localizing~$C$. As we will only need the case where~$C$ is a local ring, this extra generality is not necessary for our purposes.
\end{remark}

\begin{definition}
  \label{definition:generalized-preprojective}
  Let~$D/C$ be a relative Frobenius extension of rank~$n$ and let~$\varphi$ denote the isomorphism~$\Hom_C(D,C) \cong D$. The \emph{generalized preprojective algebra}~$\Pi_C(D)$ is defined as the quotient of~$\tensor(D/C)$ with the relations given by the images of the following morphisms
  \begin{itemize}
    \item the structure morphism~$i\colon C \rightarrow {}_CD_C = {}_CD_D \otimes_DD_C$
    \item the~$D$\dash bimodule morphism~$r\colon D \rightarrow {}_DD_C \otimes_CD_D: 1_D \mapsto r_{\varphi}$, where~$r_{\varphi}$ is the element in~${}_DD_C \otimes_CD_D$ corresponding to~$\varphi \in \Hom_D(\Hom_C(D,C),D)$ under the identification
      \begin{equation}
        \Hom_D(\Hom_C(D,C),D) \cong {}_DD_C \otimes_CD_D.
      \end{equation}
  \end{itemize}
\end{definition}

\begin{remark} \label{rem:dual-bases}
  It was mentioned in \cite[definition~1.3 and remark~1.4]{MR3565443} that the morphism~$r$ as above can also be described in a different way. For this one fixes a morphism~$\Lambda\colon C \rightarrow D$ which generates~$\Hom_C(D,C)$ as a~$D$\dash module, i.e.~$\Hom_C(D,C) = \Lambda D$. Every basis~$\{e_1, \ldots , e_n \}$ for~$D$ as a~$C$\dash module admits a dual basis~$\{f_1, \ldots , f_n \}$ for~$D$ in the sense that
  \begin{equation}
    \Lambda(e_i f_j) = \delta_{i,j}.
  \end{equation}
  The element~$r_{\varphi} \in {}_DD_C \otimes_CD_D$ is then given by~$\sum_{i=1}^n e_i \otimes f_i$. The fact that this defines a~$D$\dash bimodule morphism follows from~\cite[lemma~3.24]{1503.03992v4}.
\end{remark}

\begin{definition}
  Let~$D/C$ be a relative Frobenius extension of rank~$n$ and let~$\Pi_C(D)$ be as above. Then we define
  \begin{equation}
    \label{equation:definition-HDC}
    \mathrm{H}(D/C) \coloneqq \left( (1_C,0) \Pi_C(D) (1_C,0) \right)^{(2)}
  \end{equation}
  It is immediate that~$\mathrm{H}(D/C)$ can be written as a quotient of~$\tensor_C(D/i(C))$.
\end{definition}

\begin{lemma}
  \label{lemma:veronese-to-clifford}
  There exists a Clifford algebra~$\clifford_A(E,q)$ and a surjective morphism
  \begin{equation}
    \label{equation:veronese-to-clifford}
    \varphi\colon\mathrm{H}(D/C)\to\clifford_A(E,q).
  \end{equation}

  \begin{proof}
    By \cref{proposition:map-to-clifford-is-iso} it suffices to prove that the relations in~$\mathrm{H}(D/C)$ as a quotient of~$\tensor_C(D/i(C))$ are symmetric. Note that
    \begin{equation}
      \label{equation:definition-HDC2}
      \mathrm{H}(D/C) = \left( \tensor_C(D/i(C))\right) / (R)
    \end{equation}
    where~$R$ is the image of
    \begin{equation}
      \label{equation:morphism-r-locally}
      D \overset{r}{\rightarrow} D \otimes_C D \rightarrow D/i(C) \otimes_C D/i(C)
    \end{equation}
    We now make the following three observations:
    \begin{itemize}
      \item The image of~$r$ is described by~$r_\varphi$, which can be written as~$\sum_{i=1}^n e_i \otimes f_i$ where~$\{e_1, \ldots , e_n \}$ is a basis for~$D$ as a~$C$\dash module and~$\{f_1, \ldots , f_n \}$ is a basis dual to this one.
      \item As~$r_\varphi$ can be defined without the choice of a basis for~$D$ as~$C$\dash module, it does not depend on the choice of basis~$\{e_1, \ldots , e_n \}$.
      \item If~$\Lambda(e_i f_j) = \delta_{i,j}$ then~$\Lambda(f_i e_j) = \delta_{i,j}$ as well. In particular if~$\{f_1, \ldots , f_n \}$ is the dual basis for~$\{e_1, \ldots , e_n \}$, then~$\{e_1, \ldots , e_n \}$ is the dual basis for~$\{f_1, \ldots , f_n\}$ as well.
    \end{itemize}
    These 3 observations together imply
    \begin{equation} \label{equation:ei-and-fi-commute}
      \sum_{i=1}^n e_i \otimes f_i = r_\varphi = \sum_{i=1}^n f_i \otimes e_i
    \end{equation}
    such that~$R \subset \Sym_{C,2}(D/i(C)) \cong \Sym_C^2(D/i(C))$ as required.
  \end{proof}
\end{lemma}

\begin{proposition}
  \label{proposition:veronese-to-clifford-iso}
  The morphism in \eqref{equation:veronese-to-clifford} is an isomorphism provided that~$n=4$.

  \begin{proof}
    Surjectivity is immediate by the construction as in \cref{proposition:map-to-clifford-is-iso}. Hence it suffices to prove injectivity.

    We now claim that~$\mathrm{H}(D/C)$ and~$\clifford_A(E,q)$ have a structure as graded~$C$\dash algebra such that~$\gamma$ becomes a graded~$C$\dash algebra morphism and such that~$\mathrm{H}(D/C)_n$ and~$(\clifford_A(E,q))_n$ are free~$C$\dash modules of the same rank for each~$n$. The proposition then follows.

    As~$( 1_C \Pi_C(D))^{(2)} = ( 1_C \Pi_C(D)1_C)^{(2)} = \mathrm{H}(D/C)$ we can use \cite[lemma~3.10]{1503.03992v4} to conclude that~$\mathrm{H}(D/C)$ is a graded~$C$\dash algebra such that~$\mathrm{H}(D/C)_n$ is a free~$C$-module of rank~$n+1$.

    Next note that~$A$ has a structure of a graded~$C$\dash algebra. Moreover, as before, we let~$Q \subset A$ have degree 2. Hence letting~$F \subset E$ have degree 1, the relations
    \begin{equation}
      u \otimes v + v \otimes u = q(uv)
    \end{equation}
    become homogeneous of degree 2 for each~$u, v \in F$ and there is a unique induced grading on~$\clifford_A(E,q)$. Using \cref{remark:exterior-product} we find isomorphisms of~$C$\dash modules:
    \begin{equation}
      \begin{aligned}
        \clifford_A(E,q)_{2n} &\cong A_{2n} \oplus (E \wedge E)_{2n-2}, \\
        \clifford_A(E,q)_{2n+1} &\cong E_{2n} \oplus (E \wedge E \wedge E)_{2n-2}.
      \end{aligned}
    \end{equation}
    As~$E$,~$E \wedge E$ and~$E \wedge E \wedge E$ are free~$A$\dash modules of ranks~3, 3 and 1 respectively, we can calculate the rank of~$\clifford_A(E,q)_{2n(+1)}$ using the rank of~$A_n$. Recall that~$A = \Sym_C(Q)$ and that~$Q$ is a free~$C$\dash module of rank~2. As such~$h_A(2n)=n+1$. We hence find that~$\clifford_A(E,q)_{2n}$ is a free~$C$\dash module of~rank~$n+1 + 3n = 4n+1 = 2(2n)+1$ and that~$\clifford_A(E,q)_{2n+1}$ is a free~$C$\dash module of rank~$3(n+1) + n = 4n+3 = 2(2n+1)+1$.

    Finally it is immediate from the construction that~$\gamma\colon\mathrm{H}(D/C) \rightarrow \clifford_A(E,q)$ is a graded~$C$\dash algebra morphism. The claim and the result follow.
  \end{proof}
\end{proposition}

For use below we also introduce the following~$\mathbb{Z}$\dash algebra:
\begin{definition}
  \label{definition:Z-gen-preproj}
  Let~$D/C$ be a relative Frobenius extension of rank~$n$. We define the tensor-$\mathbb{Z}$\dash algebra~$\tensor_{\mathbb{Z}}(D/C)$ via
  \begin{equation}
    \tensor_{\mathbb{Z}}(D/C)_{(j,j+1)}
    =
    \begin{cases}
      {}_CD_D & \text{if $j$ is even} \\
      {}_DD_C & \text{if $j$ is odd}
    \end{cases}
  \end{equation}
  and we denote~$\Pi_{\mathbb{Z}}(D/C)$ for the quotient of~$\tensor_{\mathbb{Z}}(D/C)$ with the relations given by the images of~$i$ and~$r$ as in \cref{definition:generalized-preprojective}.
\end{definition}

\begin{remark}
  \label{remark:Pi-and-PiZ}
  The generalized preprojective algebra~$\Pi_D(C)$ is related to the~$\mathbb{Z}$\dash algebra~$\Pi_{\mathbb{Z}}(D/C)$ as follows:
  \begin{itemize}
    \item $\Pi_D(C) = \overline{\Pi_{\mathbb{Z}}(D/C)}$ using the notation of \cite[\S3.3]{1503.03992v4}.
    \item $\widecheck{\mathrm{H}(D/C)} \coloneqq \widecheck{ \big( (1_C,0) \Pi_C(D) (1_C,0) \big)^{(2)}} \cong \Pi_{\mathbb{Z}}(D/C)^{(2)}$
  \end{itemize}
\end{remark}

\subsection{Bimodule \texorpdfstring{$\mathbb{Z}$}{Z}-algebras and twisting}
\label{subsection:twisting}
We start by recalling the definitions of bimodules and sheaf~$\mathbb{Z}$\dash algebras as in~\cite{1503.03992v4,MR2958936}. Whenever we use the notations~$W$,~$X$,~$Y$,~$X_n$, \ldots, these schemes are in fact smooth~$k$\dash varieties.

\begin{definition}
  \label{definition:bimodule}
  A coherent~$X\mhyphen Y$ bimodule~$\mathcal{E}$ is a coherent~$\mathcal{O}_{X\times Y}$\dash module such that the support of~$\mathcal{E}$ is finite over~$X$ and~$Y$. We denote the corresponding category by~$\bimod(X\mhyphen Y)$.

  More generally an~$X\mhyphen Y$\dash bimodule is a quasicoherent~$\mathcal{O}_{X\times Y}$\dash module which is a filtered direct limit of objects in~$\bimod(X\mhyphen Y)$. The category of~$X\mhyphen Y$\dash bimodules is denoted~$\BiMod(X\mhyphen Y)$.

\end{definition}

Most bimodules we encounter in this paper are of the form~${}_u\mathcal{U}_v$ where~$u,v$ are finite maps and~$\mathcal{U}$ is a quasicoherent sheaf on some variety~$W$. More fomally we have the following.
\begin{definition}
  \label{definition:standard-form-bimodule}
  Consider finite morphisms~$u\colon W \rightarrow X$ and~$v\colon W \rightarrow Y$. If~$\mathcal{U} \in \Qcoh(W)$, then we denote~$(u, v)_*\mathcal{U} \in \BiMod(X\mhyphen Y)$ as~${}_u\mathcal{U}_v$. One easily checks that
  \begin{equation}
    \label{equation:pullback-tensor-product}
     - \otimes {}_u\mathcal{U}_v = v_*(u^*(-) \otimes_W \mathcal{U})
   \end{equation}
\end{definition}

Moreover it was shown in \cite{MR2958936} that (under some technical conditions which will always be fulfilled in this paper) every bimodule~$\mathcal{E} \in \bimod(X\mhyphen Y)$ has a unique \emph{right dual}~$\mathcal{E}^*\in\bimod(Y\mhyphen X)$ which is defined by requiring that the functor
\begin{equation}
  - \otimes_Y \mathcal{E}^*\colon \Qcoh(Y) \rightarrow \Qcoh(X)
\end{equation}
is right adjoint to the functor~$- \otimes_X \mathcal{E}$. The dual notion leads to the \emph{left dual}: an object~$^{*}{\mathcal{E}} \in \bimod(Y\mhyphen X)$. By Yoneda's lemma we have
\begin{equation}
  \mathcal{E}=^{*}{(\mathcal{E}^*)}=(^{*}{\mathcal{E}})^*
\end{equation}
Hence one can use the following notation
\begin{equation}
  \mathcal{E}^{*n}
  =
  \begin{cases}
    \mathcal{E}^{\overbrace{*\ldots *}^{n}} & n \geq 0\\
    ^{\overbrace{*\ldots *}^{-n}}{\mathcal{E}} & n<0
  \end{cases}
\end{equation}
and there are unit and counit morphisms
\begin{equation}
  \label{equation:unit-morphism}
  \begin{gathered}
    i_n\colon{}_{\id}(\mathcal{O}_{X_n})_{\id} \rightarrow \mathcal{E}^{*n} \otimes \mathcal{E}^{*n+1} \\
    j_{n}\colon\mathcal{E}^{*n}\otimes \mathcal{E}^{*n-1} \rightarrow {}_{\id}(\mathcal{O}_{X_n})_{\id}
  \end{gathered}
\end{equation}
where~$X_{2n}=X$ and~$X_{2n+1}=Y$ for all~$n$.

\begin{example}
  Let~$f\colon Y \rightarrow X$ be a finite morphism, where~$Y$ is a smooth variety over~$\Spec k$, and let~$\mathcal{E}$ be the bimodule~${}_f ( \mathcal{O}_Y) _{\id}$. One can use the explicit formula for~$\mathcal{E}^*$ as in the discussion following \cite[proposition~3.1.6]{MR2958936} to obtain~$\mathcal{E}^* = {}_{\id} ( \mathcal{O}_Y) _f$. Using \eqref{equation:pullback-tensor-product} we find
  \begin{equation}
    - \otimes_X \mathcal{E} = f^*(-) \textrm{ and } - \otimes_Y \mathcal{E}^* = f_*(-).
  \end{equation}
  As such the adjunction~$- \otimes_X \mathcal{E}\dashv- \otimes_Y \mathcal{E}^*$ is nothing but the usual adjunction~$f^* \dashv  f_*$. Using Grothendieck duality we know that~$f_*$ has a right adjoint given by~$f^!(-) = f^*(-) \otimes \omega_{X/Y}$ because~$f$ is finite and flat \cite{MR2346195}, where~$\omega_{X/Y} \in \Qcoh(Y)$ is defined by
  \begin{equation}
    f_* \omega_{X/Y} = \sheafHom(f_* \mathcal{O}_Y, \mathcal{O}_X).
  \end{equation}
  In particular we find
  \begin{equation}
    \label{equation:upper-shriek}
    \mathcal{E}^{**} = {}_f ( \omega_{X/Y}) _{\id} = \mathcal{E} \otimes \omega_{X/Y}
  \end{equation}
  with associated unit morphism
  \begin{equation}
    \label{equation:upper-shriek-2}
    {}_{\id}(\mathcal{O}_Y)_{\id} \rightarrow \mathcal{E}^{*} \otimes \mathcal{E} \otimes {}_{\id} (\omega_{X/Y})_{\id}
  \end{equation}
  Moreover by induction we have that
  \begin{equation}
    \label{equation:upper-shriek-3}
    \begin{aligned}
      \mathcal{E}^{(2n)*} &= {}_f ( \omega_{X/Y}^n) _{\id} \\
      \mathcal{E}^{(2n+1)*} &= {}_{\id} ( \omega_{X/Y}^{-n}) _f \\
    \end{aligned}
  \end{equation}
\end{example}

The tensor product of~$\mathcal{O}_{W \times X \times Y}$\dash modules induces a tensor product
\begin{equation}
  \BiMod(W\mhyphen X) \otimes  \BiMod(X\mhyphen Y) \rightarrow \BiMod(W\mhyphen Y):(\mathcal{E},\mathcal{F})\mapsto \mathcal{E}\otimes_X \mathcal{F}
\end{equation}
through the formula
\begin{equation}
  \mathcal{E} \otimes \mathcal{F}\coloneqq{\pi_{W\times Y}}_*\big(\pi^*_{W\times X}(\mathcal{E}) \otimes_{W \times X \times Y} \pi^*_{X\times Y} ( \mathcal{F}) \big).
\end{equation}
For each~$\mathcal{E} \in \BiMod(W\mhyphen X)$ this defines a functor
\begin{equation}
  - \otimes_X \mathcal{E}\colon \Qcoh(W) \rightarrow \Qcoh(X):\mathcal{M}\mapsto  \mathcal{M} \otimes_X \mathcal{E}\coloneqq {\pi_X}_*\big(\pi^*_W(\mathcal{M}) \otimes_{W \times X} \mathcal{E} \big)
\end{equation}
analogous to the notion of a Fourier--Mukai transform, which is right exact in general and exact if~$\mathcal{E}$ is locally free on the left. We mention that \cite[lemma 3.1.1]{MR2958936} shows that this functor determines the bimodule~$\mathcal{E}$ uniquely. As such the category~$\BiMod(W\mhyphen X)$ is embedded in the more abstract categories~$\Bimod(W\mhyphen X)$ and~$\BIMOD(W\mhyphen X)$ as in op.~cit.

The above tensor product turns~$\BiMod(X\mhyphen X)$ into a monoidal category. As such it is possible to define algebra objects in this category. More general one can construct~$\mathbb{Z}$\dash algebra objects with respect to a collection of categories~$\BiMod(X_i\mhyphen X_j)$ as follows.
\begin{definition}
  Let~$(X_i)_{i\in \mathbb{Z}}$ be a sequence of smooth varieties. A \emph{sheaf~$\mathbb{Z}$\dash algebra~$\mathcal{A}$} is a collection of~$X_i\mhyphen X_j$\dash bimodules~$\mathcal{A}_{i,j}$  equipped with multiplication and identity maps
  \begin{equation}
    \begin{gathered}
      \mu_{i,j,k}\colon\mathcal{A}_{i,j}\otimes \mathcal{A}_{j,k} \longrightarrow \mathcal{A}_{i,k} \\
      u_i\colon\mathcal{O}_{X_i}\longrightarrow \mathcal{A}_{i,i}
    \end{gathered}
  \end{equation}
  such that the associativity
  \begin{equation}
    \begin{tikzcd}
      \mathcal{A}_{i,j}\otimes \mathcal{A}_{j,k}\otimes{A}_{k,l} \arrow{r}{\mu_{i,j,k}\otimes 1} \arrow[swap]{d}{1\otimes \mu_{j,k,l}} & \mathcal{A}_{i,k}\otimes \mathcal{A}_{k,l} \arrow{d}{\mu_{i,k,l}} \\
      \mathcal{A}_{i,j}\otimes \mathcal{A}_{j,l} \arrow{r}{\mu_{i,j,l}} & \mathcal{A}_{i,l}
    \end{tikzcd}
  \end{equation}
  and unit diagrams
  \begin{equation}
    \begin{tikzcd}[column sep = small]
      \mathcal{O}_{X_i}\otimes\mathcal{A}_{i,j} \arrow{rr}{u_i\otimes 1} \arrow{dr} & & \mathcal{A}_{i,i}\otimes\mathcal{A}_{i,j} \arrow{dl}{\mu_{i,i,j}} & \mathcal{A}_{i,j}\otimes\mathcal{O}_{X_i} \arrow{rr}{1\otimes u_i} \arrow{dr} & & \mathcal{A}_{i,j}\otimes \mathcal{A}_{j,j} \arrow{dl}{\mu_{i,j,j}} \\
      & \mathcal{A}_{i,j} & & & \mathcal{A}_{i,j}
    \end{tikzcd}
  \end{equation}
  commute.

  A \emph{graded~$\mathcal{A}$\dash module} is a sequence of quasi-coherent~$\mathcal{O}_{X_i}$\dash modules~$\mathcal{M}_i$ together with maps
  \begin{equation}
    \mu_{\mathcal{M},i,j}\colon\mathcal{M}_i \otimes \mathcal{A}_{i,j} \longrightarrow \mathcal{M}_j
  \end{equation}
  compatible with the multiplication and identity maps on~$\mathcal{A}$ in the usual sense.

  A \emph{morphism of graded~$\mathcal{A}$\dash modules}~$f\colon\mathcal{M}\longrightarrow \mathcal{N}$ is a collection of~$X_i$\dash module morphisms~$f_i\colon\mathcal{M}_i\longrightarrow \mathcal{N}_i$ such that the obvious diagrams commute. The associated category is denoted~$\Gr(\mathcal{A})$.

  An~$\mathcal{A}$\dash module is \emph{right bounded} if~$\mathcal{M}_i = 0$ for~$i \gg 0$. An~$\mathcal{A}$\dash module is called \emph{torsion} if it is a filtered colimit of right bounded modules. Let~$\Tors(\mathcal{A})$ be the subcategory of~$\Gr(\mathcal{A})$ consisting of torsion modules. If~$\Gr(\mathcal{A})$ is a locally noetherian category (which is always the case for our applications, see for example \cite[theorem~3.1]{1503.03992v4}), then~$\Tors(\mathcal{A})$ is a localizing subcategory and the corresponding quotient category is denoted by~$\QGr(\mathcal{A})$.
 \end{definition}

\begin{remark}
  \label{remark:graded-sheaf-Z-algebra}
  Let~$\mathcal{R}$ be a sheaf of graded algebras on~$X$, then~$\mathcal{R}$ induces a sheaf\dash $\mathbb{Z}$\dash algebra~$\check{\mathcal{R}}$ on~$(X_i)_{i\in \mathbb{Z}}$ with~$X_i=X$ for all~$i$ via
  \begin{equation}
    \check{\mathcal{R}}_{i,j} = \mathcal{R}_{j-i}.
  \end{equation}
  It is well known that there are induced equivalences of categories
  \begin{equation}
    \Gr(\check{\mathcal{R}}) \cong \Gr(\mathcal{R}) \textrm{ and } \QGr(\check{\mathcal{R}}) \cong \QGr(\mathcal{R}).
  \end{equation}
\end{remark}

Of particular interest are so-called \emph{symmetric} sheaf\dash $\mathbb{Z}$\dash algebras, which are defined in \cite{MR2958936} and \cite{1503.03992v4}.

\begin{definition}
  Let~$X$ and~$Y$ be smooth varieties over~$k$ and let~$\mathcal{E}$ be an~$X\mhyphen Y$\dash bimodules for which all duals~$\mathcal{E}^{n*}$ exist. Then the \emph{tensor sheaf~$\mathbb{Z}$\dash algebra}~$\tensor(\mathcal{E})$ is the sheaf\dash $\mathbb{Z}$\dash algebra generated by the~$\mathcal{E}^{n*}$, more precisely
  \begin{equation}
    \tensor(\mathcal{E})_{m,n}
    =
    \begin{cases}
      0 & n<m \\
      {}_{\id} \big( \mathcal{O}_{X} \big)_{\id} & \textrm{ if $n=m$ is even}\\
      {}_{\id} \big( \mathcal{O}_{Y} \big)_{\id} & \textrm{ if $n=m$ is odd}\\
      \mathcal{E}^{m*} \otimes \ldots \otimes \mathcal{E}^{(n-1)*} & n>m\\
    \end{cases}
  \end{equation}
  The \emph{symmetric sheaf\dash$\mathbb{Z}$\dash algebra}~$\mathbb{S}(\mathcal{E})$ is the quotient of~$\tensor(\mathcal{E})$ by the relations defined by \eqref{equation:unit-morphism} above. More precisely,~$\mathbb{S}(\mathcal{E})_{m,n}$ is defined as
  \begin{equation}
    \begin{cases}
      \tensor(\mathcal{E}_i)_{m,n} & n \leq m+1\\
      \tensor(\mathcal{E})_{m,n} / \Big(\im(i_m)\otimes\ldots) + (\mathcal{E}^{m*}\otimes \im(i_{m+1}) \otimes \ldots)+\ldots + (\ldots \otimes \im(i_{m-2}))\Big) & n \geq m+2.
    \end{cases}
  \end{equation}
\end{definition}

A fundamental operation in the context of sheaf\dash$\mathbb{Z}$\dash algebras is that of twisting by a sequence of invertible bimodules. Invertible bimodules are defined in \cite{1503.03992v4} or \cite{MR2958936}, for our applications it suffices to realize that bimodules of the form~${}_{\id} \mathcal{L} _{\id}$ are invertible. As the next theorem shows, this operation induces equivalences of categories at the level of~$\Gr$ and~$\QGr$.
\begin{theorem} (see for example \cite[theorem~2.14]{1503.03992v4})
  \label{theorem:twisting}
  Let~$(X_i)_i$ and~$(Y_i)_i$ be sequences of smooth varieties over~$k$ and~$\mathcal{A}$ a sheaf\dash$\mathbb{Z}$\dash algebra on~$(X_i)_i$. Given a collection of invertible~$X_i\mhyphen Y_i$\dash bimodules~$(\mathcal{T}_i)_i$ one can construct a sheaf\dash$\mathbb{Z}$\dash algebra~$\mathcal{B}$ by
  \begin{equation}
    \label{equation:bundles-twist-definition}
    \mathcal{B}_{i,j}\coloneqq\mathcal{T}_i^{-1} \otimes \mathcal{A}_{i,j}\otimes \mathcal{T}_j
  \end{equation}
  called the \emph{twist} of~$\mathcal{A}$ by~$(\mathcal{T}_i)_i$.

  There is an equivalence of categories given by the functor
  \begin{equation}
    \mathcal{T}\colon \Gr(\mathcal{A})\cong \Gr(\mathcal{B}): \mathcal{M}_i \longrightarrow \mathcal{M}_i \otimes \mathcal{T}_i.
  \end{equation}
  If moreover~$\mathcal{A}$ and~$\mathcal{B}$ are noetherian, then this also induces an equivalence at the level of~$\QGr$
\end{theorem}

We will apply this theorem to the symmetric sheaf\dash $\mathbb{Z}$\dash algebras as defined in \cite[definition~2.13]{1503.03992v4} or \cite[\S 1]{MR2958936}.

\begin{lemma}
  \label{lemma:twist-2}
  Let~$f\colon Y\rightarrow X$ be a finite morphism between smooth projective curves and let~$\mathbb{S}({}_f (\mathcal{O}_Y)_{\id})$ be the associated symmetric sheaf\dash $\mathbb{Z}$\dash algebra. Moreover denote~$\Pi_\mathbb{Z}(Y/X)$ for the sheaf\dash $\mathbb{Z}$\dash algebra (over the collection~$(X_i)_{i \in \mathbb{Z}}$ with~$X_{2j}=X$ and~$X_{2j+1}=Y$) defined as follows:
  \begin{itemize}
    \item $\Pi_{\mathbb{Z}}(Y/X)_{m,n} = 0$ whenever~$m>n$;
    \item $\Pi_{\mathbb{Z}}(Y/X)_{2j,2j} = {}_{\id}( \mathcal{O}_X)_{\id}$ for all~$j$;
    \item $\Pi_{\mathbb{Z}}(Y/X)_{2j+1,2j+1} = {}_{\id}( \mathcal{O}_Y)_{\id}$ for all~$j$;
    \item $\Pi_{\mathbb{Z}}(Y/X)_{2j,2j+1} = {}_f (\mathcal{O}_Y)_{\id}$ for all~$j$;
    \item $\Pi_{\mathbb{Z}}(Y/X)_{2j+1,2j+2} = {}_{\id}( \mathcal{O}_Y)_{f}$ for all~$j$;
    \item $\Pi_{\mathbb{Z}}(Y/X)$ is freely generated by the~$\Pi_{\mathbb{Z}}(Y/X)_{n,n+1}$ subject to the relations
      \begin{equation}
        \begin{aligned}
          {}_{\id}( \mathcal{O}_X)_{\id}&\subset{}_{\id}( f_* \mathcal{O}_Y)_{\id} \\
          &= {}_f( \mathcal{O}_X)_{f}  \\
          &= \Pi_{\mathbb{Z}}(Y/X)_{2j,2j+1} \otimes_Y \Pi_{\mathbb{Z}}(Y/X)_{2j+1,2j+2}
        \end{aligned}
      \end{equation}
      and
      \begin{equation}
        r({}_{\id}( \omega_{f}^{-1})_{\id}) \subset \Pi_{\mathbb{Z}}(Y/X)_{2j-1,2j} \otimes_Y \Pi_{\mathbb{Z}}(Y/X)_{2j,2j+1}
      \end{equation}
      where
      \begin{equation}
        \label{equation:definition-r-global}
        r\colon {}_{\id}( \omega_{f}^{-1})_{\id} \rightarrow {}_{\id}( \mathcal{O}_Y)_{f} \otimes {}_f (\mathcal{O}_Y)_{\id}
      \end{equation}
      is induced by~$f_* \dashv f^! = f^* \otimes \omega_{X/Y}$ as in \eqref{equation:upper-shriek-2}.
  \end{itemize}
    Then~$\Pi_\mathbb{Z}(Y/X)$ is a twist of~$\mathbb{S}({}_f (\mathcal{O}_Y)_{\id})$.
\end{lemma}

\begin{proof}
  Using \eqref{equation:upper-shriek-3} it follows immediately that
  \begin{equation}
    \left(\mathbb{S}({}_f (\mathcal{O}_Y)_{\id}) \right)_{2j,2j+1} = (\Pi_{\mathbb{Z}}(Y/X))_{2j,2j+1} \otimes {}_{\id}(\omega_{X/Y}^j)_{\id}
  \end{equation}
  and
  \begin{equation}
    \left(\mathbb{S}({}_f (\mathcal{O}_Y)_{\id}) \right)_{2j+1,2j+2} = {}_{\id}(\omega_{X/Y}^{-j})_{\id} \otimes (\Pi_{\mathbb{Z}}(Y/X))_{2j+1,2j+2}.
  \end{equation}

  Now let~$\mathcal{T}_i$ be defined by
  \begin{equation}
    \mathcal{T}_i\coloneqq
    \begin{cases}
      {}_{\id}( \mathcal{O}_X )_{\id} & i=2j \\
      {}_{\id}( \omega_{X/Y}^j )_{\id} & i=2j+1.
    \end{cases}
  \end{equation}
  We then claim that
  \begin{equation}
    \left(\mathbb{S}({}_f (\mathcal{O}_Y)_{\id}) \right)_{m,n} \cong \mathcal{T}_m^{-1} \otimes \Pi_{\mathbb{Z}}(Y/X)_{m,n} \otimes \mathcal{T}_n
  \end{equation}
  holds for all~$m,n$.

  As both algebras are generated in degree 1 and have quadratic relations, it suffices to check the claim for~$n-m = 0,1,2$. By the above the claim holds for~$n-m=0,1$.

  For~$n=m+2=2j+2$ we have
  \begin{equation}
    \left(\mathbb{S}({}_f (\mathcal{O}_Y)_{\id}) \right)_{2j,2j+2} = {}_{\id} \left( f_* \mathcal{O}_Y /  \mathcal{O}_X \right)_{\id} = (\Pi_{\mathbb{Z}}(Y/X))_{2j,2j+2}
  \end{equation}

  To get the claim for~$n=m+2$,~$m=2j+1$ it suffices to check that
  \begin{equation}
    \begin{aligned}
      i_{2j+1} ({}_{\id}( \mathcal{O}_Y)_{\id})
      &={}_{\id}(\omega_{X/Y}^{-j})_{\id} \otimes  r({}_{\id}( \omega_{f}^{-1})_{\id}) \otimes {}_{\id}(\omega_{X/Y}^{j+1})_{\id}  \\
      &={}_{\id}(\omega_{X/Y}^{-j})_{\id} \otimes r({}_{\id}( \omega_{f}^{-1})_{\id}) \otimes {}_{\id}(\omega_{X/Y})_{\id} \otimes {}_{\id}(\omega_{X/Y}^j)_{\id}
    \end{aligned}
  \end{equation}
  This in turn follows by using the fact that $i_{2j+1}$ and $r$ are defined as unit maps, the fact that the bimodules underlying these unit maps are twists of each other and the fact that
  \begin{equation}
     \left({}_{\id}(\omega_{X/Y}^{-j})_{\id} \otimes {}_{\id}(\mathcal{O}_Y)_f \right)^* = \left({}_{\id}(\mathcal{O}_Y)_f \right)^* \otimes {}_{\id}(\omega_{X/Y}^{j})_{\id}.
  \end{equation}
\end{proof}

\begin{remark}
  \label{remark:Pi-ZYX-2-periodic}
  The sheaf\dash$\mathbb{Z}$\dash algebra~$\Pi_\mathbb{Z}(Y/X)$ as introduced above is 2-periodic, i.e.~for each~$i$ and~$j$ there is an isomorphism
  \begin{equation}
    \left(\Pi_\mathbb{Z}(Y/X)\right)_{i,j} \cong \left(\Pi_\mathbb{Z}(Y/X)\right)_{i+2,j+2}
  \end{equation}
  and these isomorphisms are compatible with the multiplication in~$\Pi_\mathbb{Z}(Y/X)$.

  In particular let~$\mathrm{H}(Y/X)$ be the graded sheaf of algebras on~$X$ defined as follows:
  \begin{itemize}
    \item $\mathrm{H}(Y/X)_0 = \mathcal{O}_X$;
    \item $\mathrm{H}(Y/X)_1 = f_* \mathcal{O}_Y / \mathcal{O}_X$;
    \item $\mathrm{H}(Y/X)$ is generated by~$\mathrm{H}(Y/X)_1$ subject to the relation
      \begin{equation}
        f_* \omega_{X/Y}^{-1} \subset \mathrm{H}(Y/X)_1 \otimes \mathrm{H}(Y/X)_1 =  f_* \mathcal{O}_Y / \mathcal{O}_X \otimes f_* \mathcal{O}_Y / \mathcal{O}_X
      \end{equation}
      which is induced by
      \begin{equation}
        \begin{aligned}
          {}_{\id} (f_* \omega_{X/Y}^{-1}) _{\id} & = {}_f (\omega_{X/Y}) _f \\
          &= {}_f (\mathcal{O}_Y)_{\id} \otimes {}_{\id} ( \omega_{X/Y}^{-1}) _{\id} \otimes {}_{\id}(\mathcal{O}_Y) _f \\
          &\overset{r}{\rightarrow} {}_f (\mathcal{O}_Y)_{\id} \otimes {}_{\id}(\mathcal{O}_Y) _f \otimes {}_f (\mathcal{O}_Y)_{\id} \otimes {}_{\id}(\mathcal{O}_Y) _f \\
          &= {}_{\id} (f_* \mathcal{O}_Y \otimes f_* \mathcal{O}_Y) _{\id} \\
          &\rightarrow {}_{\id} (f_* \mathcal{O}_Y/ \mathcal{O}_X \otimes f_* \mathcal{O}_Y / \mathcal{O}_X) _{\id};
        \end{aligned}
      \end{equation}
  \end{itemize}
  then there is an isomorphism of sheaf\dash $\mathbb{Z}$\dash algebras:
  \begin{equation}
    \label{equation:2-veronese-PiYX}
    \Pi_\mathbb{Z}(Y/X)^{(2)} \cong \check{H}(Y/X)
  \end{equation}
\end{remark}

Using \cref{theorem:twisting}, 
\cref{lemma:twist-2}, \cref{remark:Pi-ZYX-2-periodic} and \cref{remark:graded-sheaf-Z-algebra} we have equivalences of categories
\begin{equation}
  \label{equation:twisting-equivalent-categories}
  \begin{aligned}
    \QGr(\mathbb{S}({}_f (\mathcal{O}_{\mathbb{P}^1})_{\id}))
    &\cong \QGr(\Pi_{\mathbb{Z}}(\mathbb{P}^1/\mathbb{P}^1)) \\
    &\cong \QGr(\mathrm{H}(\mathbb{P}^1/\mathbb{P}^1)).
  \end{aligned}
\end{equation}

As a last result of this section we mention that the above category is preserved by the~$\PGL_2 \times \PGL_2$\dash action on the set of degree~4~maps~$f\colon\mathbb{P}^1 \rightarrow \mathbb{P}^1$ (where the action is given by coordinate changes of on the domain and codomain).
\begin{lemma}
  Let~$f\colon\mathbb{P}^1 \rightarrow \mathbb{P}^1$ be a finite map and let~$\varphi, \psi \in \Aut(\mathbb{P}^1)$, then
  \begin{equation}
    \QGr(\mathbb{S}({}_f (\mathcal{O}_{\mathbb{P}^1})_{\id})) \cong \QGr(\mathbb{S}({}_{\varphi \circ f \circ \psi} (\mathcal{O}_{\mathbb{P}^1})_{\id})).
  \end{equation}
\end{lemma}
\begin{proof}
  By \cref{theorem:twisting} it suffices to show that~$\mathbb{S}({}_f (\mathcal{O}_{\mathbb{P}^1})_{\id})$ and~$\mathbb{S}({}_{\varphi \circ f \circ \psi} (\mathcal{O}_{\mathbb{P}^1})_{\id})$ are twists of each other. We claim that
  \begin{equation}
    \left(\mathbb{S}({}_{\varphi \circ f \circ \psi} (\mathcal{O}_{\mathbb{P}^1})_{\id})\right)_{i,j} \cong \mathcal{T}_i^{-1} \otimes \left(\mathbb{S}({}_f (\mathcal{O}_{\mathbb{P}^1})_{\id})\right)_{i,j} \otimes \mathcal{T}_j.
  \end{equation}
  holds for all~$i$ and~$j$ if we choose
  \begin{equation}
    \mathcal{T}_i\coloneqq
    \begin{cases}
      {}_{\id} \mathcal{O}_\varphi & \textrm{ if $i$ is even} \\
      {}_\psi \mathcal{O}_{\id} & \textrm{if $i$ is odd}
    \end{cases}.
  \end{equation}
  For example we have that
  \begin{equation}
    \begin{aligned}
      \left(\mathbb{S}({}_{\varphi \circ f \circ \psi} (\mathcal{O}_{\mathbb{P}^1})_{\id})\right)_{2,3} & = {}_{\varphi \circ f \circ \psi} \left(\omega_{\varphi \circ f \circ \psi}\right)_{\id} \\
      & \cong {}_{\varphi \circ f \circ \psi} \left(\omega_\psi \otimes \psi^* \omega_{X/Y} \otimes \psi^*f^* \omega_\varphi \right)_{\id} \\
      & \cong {}_{\varphi \circ f \circ \psi} \left(\mathcal{O}_{\mathbb{P}^1} \otimes \psi^* \omega_{X/Y} \otimes \psi^*f^* \mathcal{O}_{\mathbb{P}^1} \right)_{\id} \\
      & \cong {}_{\varphi \circ f \circ \psi} \left(\psi^* \omega_{X/Y} \right)_{\id} \\
      & \cong {}_{\varphi} ( \mathcal{O}_{\mathbb{P}^1})_{\id} \otimes {}_f \left(\omega_{X/Y} \right)_{\id} \otimes {}_\psi ( \mathcal{O}_{\mathbb{P}^1})_{\id}\\
      & \cong \left({}_{\id} ( \mathcal{O}_{\mathbb{P}^1})_{\varphi}\right)^{-1} \otimes \left(\mathbb{S}({}_f (\mathcal{O}_{\mathbb{P}^1})_{\id})\right)_{2,3} \otimes {}_\psi ( \mathcal{O}_{\mathbb{P}^1})_{\id}.
    \end{aligned}
  \end{equation}
  We leave it as an exercise to the reader to check that this twisting is compatible with the quadratic relations of~$\mathbb{S}({}_f (\mathcal{O}_{\mathbb{P}^1})_{\id})$ and~$\mathbb{S}({}_{\varphi \circ f \circ \psi} (\mathcal{O}_{\mathbb{P}^1})_{\id})$.
\end{proof}

\subsection{Symmetric sheaves of graded algebras}
\label{subsection:symmetric-sheaves}
This section is devoted to generalizing \cref{lemma:symmetric-relations-to-clifford}, \cref{proposition:map-to-clifford-is-iso} and \cref{lemma:veronese-to-clifford} to the level of sheaves of algebras. Throughout this section, all schemes are assumed to have~2~invertible, and for ease of statement we will also assume that they are irreducible.

\begin{definition}
  \label{definition:symmetric-sheaves-of-algebras}
  Let~$\mathcal{H} = \mathcal{H}_0 \oplus \mathcal{H}_1 \oplus \ldots$ be a graded sheaf of algebras on a scheme~$X$. We say~$\mathcal{H}$ is a \emph{symmetric sheaf of graded algebras} if the following conditions hold:
  \begin{itemize}
    \item $\mathcal{H}_1$ is a locally free sheaf on~$X$;
    \item there is a surjective morphism of sheaves of graded algebras
      \begin{equation}
        \varphi\colon\tensor_{\mathcal{O}_X}(\mathcal{H}_1) \twoheadrightarrow \mathcal{H};
      \end{equation}
    \item $\varphi$ is an isomorphism in degree 0 and 1;
    \item $\ker(\varphi) \cong  \mathcal{R} \otimes_{\mathcal{O}_X} \tensor_{\mathcal{O}_X}(\mathcal{H}_1)$ where~$\mathcal{R}$ is a direct summand of~$\Sym_{\mathcal{O}_X,2}(\mathcal{H}_1) \subset \tensor_{\mathcal{O}_X}(\mathcal{H}_1)_2$.
  \end{itemize}
\end{definition}

The following is immediate from the local nature of the construction of~$\Sym_{\mathcal{O}_X,2}(-)$ and~$\tensor_{\mathcal{O}_X}(-)$.
\begin{lemma}
  \label{lemma:symmetric-sheaf-locally}
  Let~$\mathcal{H}$ be a sheaf of graded algebras on~$X$. Then the following are equivalent:
  \begin{enumerate}
    \item $\mathcal{H}$ is a symmetric sheaf of graded algebras on~$X$.
    \item for each point~$p \in X$ we have that~$\mathcal{H}_p$ is a graded~($\mathcal{O}_{X,p}$-)algebra with symmetric relations as in \cref{subsection:symmetric-relations-clifford-comparison}.
  \end{enumerate}
\end{lemma}

We use this lemma to prove the following.

\begin{lemma}
  \label{lemma:HYX-symmetric}
  Let~$X$ and~$Y$ be smooth varieties over a field~$k$ of characteristic different from 2 and let~$f\colon Y \rightarrow X$ be a finite morphism of degree~$n$. Let~$\mathrm{H}(Y/X)$ be the associated sheaf of graded algebras as in \cref{remark:Pi-ZYX-2-periodic}. Then~$\mathrm{H}(Y/X)$ is a symmetric sheaf of graded algebras on~$X$.
\end{lemma}

\begin{proof}
  Using \cref{lemma:symmetric-sheaf-locally} it suffices to prove that~$\mathrm{H}(Y/X)_p$ is a graded algebra with symmetric relations as in \cref{subsection:symmetric-relations-clifford-comparison}. For this we first use \cite[lemma~3.2.7]{1503.03992v4} which (among other things) shows that~$(f_* \mathcal{O}_Y)_p / \mathcal{O}_{X,p}$ is relative Frobenius of rank~$n$. For example one has an isomorphism of~$(f_* \mathcal{O}_Y)_p$-modules
  \begin{equation}
    \label{equation:omega-f-p-locally}
    (f_* \mathcal{O}_Y)_p \cong (\omega_{X/Y})_p \cong \Hom_{\mathcal{O}_{X,p}}((f_* \mathcal{O}_Y)_p / \mathcal{O}_{X,p}).
  \end{equation}
  We now claim that
  \begin{equation} \label{equation:f-to-frobenius-pair}
    \mathrm{H}(Y/X)_p = \mathrm{H}\big((f_* \mathcal{O}_Y)_p / \mathcal{O}_{X,p}\big)
  \end{equation}
  where the right hand side was defined in \eqref{equation:definition-HDC}. The lemma follows from this claim and the fact that~$\mathrm{H}((f_* \mathcal{O}_Y)_p / \mathcal{O}_{X,p})$ is a graded algebra with symmetric relations (see for example the proof of \cref{lemma:veronese-to-clifford}).

  By comparing \eqref{equation:definition-HDC2} with \cref{remark:Pi-ZYX-2-periodic} we see that the claim holds in degree 0 and 1: both algebras can be written as quotients of the tensor algebra~$\tensor_{\mathcal{O}_{X,p}}((f_* \mathcal{O}_Y)_p / \mathcal{O}_{X,p})$ subject to quadratic relations.

  Recall that the quadratic relations in~$\mathrm{H}((f_* \mathcal{O}_Y)_p / \mathcal{O}_{X,p})$ were obtained by composing the morphisms~$r$ in \eqref{equation:morphism-r-locally} with the obvious morphism
  \begin{equation}
    \label{equation:obvious-pi}
    \pi\colon (f_* \mathcal{O}_Y)_p \otimes_{\mathcal{O}_{X,p} } (f_* \mathcal{O}_Y)_p \rightarrow (f_* \mathcal{O}_Y)_p / \mathcal{O}_{X,p} \otimes_{\mathcal{O}_{X,p} } (f_* \mathcal{O}_Y)_p / \mathcal{O}_{X,p}.
  \end{equation}
  A closer examination of the morphisms~$r$ in \eqref{equation:morphism-r-locally}, see for example the proof of \cite[lemma~3.4.2]{1503.03992v4}, shows that the forgetful functor~$\Mod((f_* \mathcal{O}_Y)_p) \rightarrow \Mod(\mathcal{O}_{X,p})$ is left adjoint to the functor~$- \otimes_{\mathcal{O}_{X,p}}(f_* \mathcal{O}_Y)_p : \Mod(\mathcal{O}_{X,p}) \rightarrow \Mod((f_* \mathcal{O}_Y)_p)$ and that~$r$ coincides with the associated unit morphism
  \begin{equation}
    \label{equation:unit-morphism-locally}
    (f_* \mathcal{O}_Y)_p \rightarrow  (f_* \mathcal{O}_Y)_p \otimes_{\mathcal{O}_{X,p}}(f_* \mathcal{O}_Y)_p.
  \end{equation}

  Similarly the quadratic relations in~$\mathrm{H}(Y/X)_p$ were obtained as the image of~$\pi \circ r'$ where~$\pi$ is as in \eqref{equation:obvious-pi} and~$r'$ is obtained by localizing the unit morphism in \eqref{equation:definition-r-global}. Hence in order to prove the claim it suffices to show that the localization of the unit morphism in \eqref{equation:definition-r-global} coincides with the unit morphism in \eqref{equation:unit-morphism-locally}. For this we use the isomorphism~$(f_* \mathcal{O}_Y)_p \cong (\omega_{X/Y})_p$ as in \eqref{equation:omega-f-p-locally} and the commutativity of the following diagrams
  \begin{equation}
    \begin{tikzcd}[column sep = large]
      \Qcoh(\mathcal{O}_Y) \arrow{r}{f_*=-\otimes\prescript{}{\id}{(\mathcal{O}_Y)}_f} \arrow[d, "f_{*,p}"] & \Qcoh(\mathcal{O}_X) \arrow[d, "(-)_p"] \\
      \Mod((f_*\mathcal{O}_Y)_p) \arrow[r, "\textrm{forget}"] & \Mod(\mathcal{O}_{X,p})
    \end{tikzcd}
  \end{equation}
  and
  \begin{equation}
    \begin{tikzcd}[column sep = large]
      \Qcoh(\mathcal{O}_X) \arrow{r}{f^*=-\otimes\prescript{}{f}{(\mathcal{O}_Y)_{\id}}} \arrow[d, "(-)_p"] & \Qcoh(\mathcal{O}_Y) \arrow[d, "f_{*,p}"] \\
      \Mod(\mathcal{O}_{X,p}) \arrow[r, "- \otimes_{\mathcal{O}_{X,p}} (f_*\mathcal{O}_Y)_p "] & \Mod((f_*\mathcal{O}_Y)_p)
    \end{tikzcd}
  \end{equation}
  where for the second diagram, commutativity follows from the projection formula.
\end{proof}

The local nature of the constructions in \cref{definition:clifford-algebra-sheaf} allow us to generalize \cref{lemma:symmetric-relations-to-clifford} and \cref{proposition:map-to-clifford-is-iso} as follows.
\begin{lemma}
  \label{lemma:symmetric-relations-to-clifford-global}
  Let~$\mathcal{H} = \tensor_{\mathcal{O}_X}(\mathcal{H}_1)/ \mathcal{R}$ be a symmetric sheaf of graded algebras on~$X$. Let~$\mathcal{Q} = \Sym_{\mathcal{O}_X}^2(\mathcal{H}_1))/ \mathcal{R}$, $\mathcal{A} = \Sym_{\mathcal{O}_X} \mathcal{Q}$ and~$\mathcal{E} = \mathcal{H}_1 \otimes_{\mathcal{O}_X} \mathcal{A}$.

  Let~$q\colon\Sym_{\mathcal{A}}^2(\mathcal{E}) \rightarrow \mathcal{A}$ be induced by the quotient map~$\tilde{q}\colon \Sym_{\mathcal{O}_X}^2(\mathcal{H}_1)) \rightarrow \mathcal{Q}$; i.e.
  \begin{equation}
    q\colon \Sym_\mathcal{A}^2(\mathcal{E}) = \Sym_{\mathcal{O}_X}^2(\mathcal{H}_1) \otimes_{\mathcal{O}_X} \mathcal{A} \rightarrow \mathcal{Q} \otimes_{\mathcal{O}_X} \mathcal{A} \overset{\mu}{\rightarrow} \mathcal{A}
  \end{equation}
  where~$\mu$ is given by multiplication in~$\mathcal{A}$.

  Then the inclusions~$\mathcal{O}_X \hookrightarrow \mathcal{A}$ and~$\mathcal{H}_1 \hookrightarrow \mathcal{E}$ induce a epimorphism
  \begin{equation}
    \label{equation:symmetric-relations-to-clifford-sheaf}
    \varphi\colon \mathcal{H} \twoheadrightarrow \clifford_{\mathcal{A}}(\mathcal{E},q).
  \end{equation}
\end{lemma}

We now combine \cref{lemma:symmetric-relations-to-clifford-global} with \cref{lemma:HYX-symmetric} to obtain the following.
\begin{lemma}
  \label{lemma:H4-clifford}
  Let~$X$ and~$Y$ be smooth varieties over a field~$k$ of characteristic different from 2 and let~$f\colon Y \rightarrow X$ be a finite morphism of degree~4. Let~$\mathrm{H}(Y/X)$ and~$\clifford_{\mathcal{A}}(\mathcal{E},q)$ be as above. Then there is an isomorphism
  \begin{equation}
    \label{equation:HYX-iso-clifford}
    \mathrm{H}(Y/X) \cong \clifford_{\mathcal{A}}(\mathcal{E},q).
  \end{equation}
\end{lemma}

\begin{proof}
  Combining \cref{lemma:symmetric-relations-to-clifford-global} and \cref{lemma:HYX-symmetric} we already know that there is an epimorphism
  \begin{equation}
    \varphi\colon \mathrm{H}(Y/X) \twoheadrightarrow \clifford_{\mathcal{A}}(\mathcal{E},q).
  \end{equation}
  It hence suffices to show that~$\varphi$ is in fact an isomorphism, and this can be checked locally. Hence let~$p$ be any point in~$X$. We must show that
  \begin{equation}
    \varphi_p \colon \mathrm{H}(Y/X)_p = \mathrm{H}((f_* \mathcal{O}_Y)_p / \mathcal{O}_{X,p}) \twoheadrightarrow \left(\clifford_{\mathcal{A}}(\mathcal{E},q)\right)_p = \clifford_{\mathcal{A}_p}(\mathcal{E}_p,q_p)
  \end{equation}
  is an isomorphism. This follows from \cref{proposition:veronese-to-clifford-iso} by noticing that~$\varphi_p$ coincides with \eqref{equation:veronese-to-clifford}.
\end{proof}

We are particularly interested in the case where~$Y=X=\mathbb{P}^1$. Hence from now on fix a finite, degree~4 morphism~$f\colon\mathbb{P}^1 \rightarrow \mathbb{P}^1$. We continue this section by describing the sheaf of algebras~$\mathcal{A}$ and the~$\mathcal{A}$\dash module~$\mathcal{E}$ in \eqref{equation:HYX-iso-clifford}. We will use the following easy lemmas.

\begin{lemma}
  \label{lemma:pushforward-line-bundles}
  We have for~$n \in \mathbb{Z}$ and~$i=0,1,2,3$ that
  \begin{equation}
    f_*(\mathcal{O}_{\mathbb{P}^1}(4n+i))\cong\mathcal{O}_{\mathbb{P}^1}(n)^{\oplus i+1}\oplus\mathcal{O}_{\mathbb{P}^1}(n-1)^{3-i}.
  \end{equation}

  \begin{proof}
    Because~$f$ is a finite morphism between regular schemes we have that it is necessarily flat.
    Hence~$f_*(\mathcal{O}_{\mathbb{P}^1}(4n+i))$ must be locally free. By Grothendieck's splitting theorem it must split as a direct sum of line bundles. The exact splitting can be found using that
    \begin{equation}
      \begin{aligned}
        \Hom_{\mathbb{P}^1}(f_* \left(\mathcal{O}_{\mathbb{P}^1}(4n+i) \right), \mathcal{O}_{\mathbb{P}^1}(a))
        &\cong \Hom_{\mathbb{P}^1}(\mathcal{O}_{\mathbb{P}^1}(4n+i) , f^* \left(\mathcal{O}_{\mathbb{P}^1}(a)\right)) \\
        &\cong \Hom_{\mathbb{P}^1}(\mathcal{O}_{\mathbb{P}^1}(4n+i) , \mathcal{O}_{\mathbb{P}^1}(4a))
      \end{aligned}
    \end{equation}
    which is~$\max \{0,4a-4n-i\}$\dash dimensional.
  \end{proof}
\end{lemma}

\begin{lemma}
  \label{lemma:description-F}
  We have that
  \begin{equation}
    \mathcal{F}_{Y/X}\coloneqq\mathrm{H}(\mathbb{P}^1/\mathbb{P}^1)_1 \cong\mathcal{O}_{\mathbb{P}^1}(-1)^{\oplus 3}.
  \end{equation}

  \begin{proof}
    By \cref{lemma:pushforward-line-bundles} we have that~$f_*(\mathcal{O}_{\mathbb{P}^1})\cong\mathcal{O}_{\mathbb{P}^1}\oplus\mathcal{O}_{\mathbb{P}^1}(-1)^{\oplus 3}$. As the structure morphism~$\mathcal{O}_{\mathbb{P}^1} \rightarrow f_*(\mathcal{O}_{\mathbb{P}^1})$ is nonzero we have that it is necessarily an isomorphism between~$\mathcal{O}_{\mathbb{P}^1}$ and the summand~$\mathcal{O}_{\mathbb{P}^1}$ of~$f_*(\mathcal{O}_{\mathbb{P}^1})$.
  \end{proof}
\end{lemma}

\begin{lemma}
  \label{lemma:description-I}
  We have that
  \begin{equation}
    \omega_{X/Y} \cong\mathcal{O}_{\mathbb{P}^1}(6).
  \end{equation}

  \begin{proof}
    By construction
    \begin{equation}
      \begin{aligned}
        f_* \omega_{X/Y}
        &\cong\sheafHom(f_* \mathcal{O}_{\mathbb{P}^1}, \mathcal{O}_{\mathbb{P}^1}) \\
        &\cong\sheafHom(\mathcal{O}_{\mathbb{P}^1} \oplus \mathcal{O}_{\mathbb{P}^1}(-1)^{\oplus 3}, \mathcal{O}_{\mathbb{P}^1}) \\
        &\cong\mathcal{O}_{\mathbb{P}^1} \oplus \mathcal{O}_{\mathbb{P}^1}(1)^{\oplus 3}\\
        &\cong f_* (\mathcal{O}_{\mathbb{P}^1}(6))
      \end{aligned}
    \end{equation}
    where we applied \cref{lemma:pushforward-line-bundles} twice. The result now follows as~$\mathcal{I}$ is an invertible sheaf and~$f_*$ is faithful (because~$f$ is affine).
  \end{proof}

\end{lemma}

We can use these lemmas to obtain the following.
\begin{proposition}
  \label{proposition:description-Q}
  Let~$\mathcal{Q} \coloneqq \Sym_{\mathbb{P}^1}^2\big(\mathcal{F}_{Y/X}/f_*(\omega_{X/Y}^{-1})\big)$ as above.

  We have that
  \begin{equation}
    \mathcal{Q}\cong\mathcal{O}_{\mathbb{P}^1}(-2)\oplus\mathcal{O}_{\mathbb{P}^1}(-1).
  \end{equation}
\end{proposition}

\begin{proof}
  We wish to prove that
  \begin{enumerate}
    \item the morphism~$r\colon f_*(\omega_{X/Y}^{-1}) \rightarrow \Sym_{\mathbb{P}^1}^2(\mathrm{H}(\mathbb{P}^1/\mathbb{P}^1)_1)$ is injective;
    \item the quotient~$\mathcal{Q}$ is a locally free sheaf.
  \end{enumerate}
  Assuming that these hold there is a short exact sequence
  \begin{equation}
    \label{equation:Q-as-cokernel}
    0 \rightarrow \mathcal{O}_{\mathbb{P}^1}(-2)^{\oplus 3} \oplus \mathcal{O}_{\mathbb{P}^1}(-3) \rightarrow \mathcal{O}_{\mathbb{P}^1}(-2)^{\oplus 6} \rightarrow \mathcal{Q} \rightarrow 0
  \end{equation}
  where we used \cref{lemma:description-F} to obtain~$\Sym_{\mathbb{P}^1}^2(\mathcal{F}) \cong \mathcal{O}_{\mathbb{P}^1}(-2)^{\oplus 6}$ and \cref{lemma:description-I} together with \cref{lemma:pushforward-line-bundles} to find~$f_*(\omega_{X/Y}^{-1}) \cong f_*(\mathcal{O}_{\mathbb{P}^1}(-6)) \cong \mathcal{O}_{\mathbb{P}^1}(-2)^{\oplus 3} \oplus \mathcal{O}_{\mathbb{P}^1}(-3)$. If~$\mathcal{Q}$ is locally free it splits as a sum of two line bundles~$\mathcal{O}_{\mathbb{P}^1}(i) \oplus \mathcal{O}_{\mathbb{P}^1}(j)$ where~$i \geq j \geq -2$. By \cite[exercise~II.5.16.d]{MR0463157} we find
  \begin{equation}
    \mathcal{O}_{\mathbb{P}^1}(-12) \cong \mathcal{O}_{\mathbb{P}^1}(-9) \otimes \mathcal{O}_{\mathbb{P}^1}(i+j)
  \end{equation}
  with unique solution~$i=-1, j=-2$. It hence remains to prove that~$r$ is injective and that is cokernel is locally free.

  The injectivity is checked locally. As in the proof of \cref{lemma:HYX-symmetric},~$r$~is locally given by
  \begin{equation}
    D \rightarrow {}_DD_C \otimes_CD_D: 1_D \mapsto \sum_{i=1}^4 e_i \otimes f_i
  \end{equation}
  where~$D = (f_* \mathcal{O}_Y)_p$, $C = \mathcal{O}_{X,p}$ and~$e_1, e_2, e_3, e_4$ and~$f_1, f_2, f_3, f_4$ form dual bases for~$D$ as~$C$\dash module. As~$\coker(r) = (0,1_D)\left((\Pi_C(D))_2\right)(0,1_D)$ is a free~$C$-module of rank~12 by \cite[lemma~3.10]{MR3565443}, $r$~must necessarily be injective.

  That the cokernel of~$r$ is locally free can also be checked locally. Hence let~$Q= \Sym_C^2(D/C) / \im(r)$ for a relative Frobenius pair~$D/C$. A computation similar to the one carried out in \cite[\S3]{MR3565443} reduces to showing that~$\dim_k(Q)$ does not depend on~$D$ in the case where~$C$ is an algebraically closed field~$k$. This in turn follows from the equality
  \begin{equation}
    \dim_k(Q) = \dim_k(\Sym_k^2(D/k)) - \dim_k(D) = 6 - 4 = 2 .
  \end{equation}
\end{proof}

We can now conclude the description in the following corollary.

\begin{corollary}
  \label{corollary:H-is-clifford}
  There exists an isomorphism
  \begin{equation}
    \mathrm{H}(\mathbb{P}^1 / \mathbb{P}^1) \cong \clifford_{\Sym_{\mathbb{P}^1}(\mathcal{O}_{\mathbb{P}^1}(-2)\oplus\mathcal{O}_{\mathbb{P}^1}(-1))} \left(\mathcal{O}_{\mathbb{P}^1}(-1)^{\oplus 3} \otimes \Sym_{\mathbb{P}^1}(\mathcal{O}_{\mathbb{P}^1}(-2)\oplus\mathcal{O}_{\mathbb{P}^1}(-1)), q \right)
  \end{equation}
\end{corollary}
We can moreover twist everything in the construction by the appropriate~$\mathcal{O}_{\mathbb{P}^1}(i)$'s to find and equivalence of categories
\begin{eqnarray*}
& \qgr \left( \clifford_{\Sym_{\mathbb{P}^1}(\mathcal{O}_{\mathbb{P}^1}(-2)\oplus\mathcal{O}_{\mathbb{P}^1}(-1))} \left(\mathcal{O}_{\mathbb{P}^1}(-1)^{\oplus 3} \otimes \Sym_{\mathbb{P}^1}(\mathcal{O}_{\mathbb{P}^1}(-2)\oplus\mathcal{O}_{\mathbb{P}^1}(-1)), q \right) \right) & \\
& \downarrow \ \cong & \\
& \qgr \left( \clifford_{\Sym_{\mathbb{P}^1}(\mathcal{O}_{\mathbb{P}^1}\oplus\mathcal{O}_{\mathbb{P}^1}(1))} \left(\mathcal{O}_{\mathbb{P}^1}^{\oplus 3} \otimes \Sym_{\mathbb{P}^1}(\mathcal{O}_{\mathbb{P}^1}\oplus\mathcal{O}_{\mathbb{P}^1}(1)), q \right) \right) &
\end{eqnarray*}

\begin{remark}
  \label{remark:q-global}
  The morphism~$q\colon\Sym^2(\mathcal{O}_{\mathbb{P}^1}^{\oplus 3}) \rightarrow \mathcal{O}_{\mathbb{P}^1} \oplus \mathcal{O}_{\mathbb{P}^1}(1)$ in \eqref{equation:veronese-to-clifford} is obtained from a surjective map~$q\colon\Sym^2(E) \rightarrow V$ where~$E\coloneqq\HH^0(\mathbb{P}^1,\mathcal{O}_{\mathbb{P}^1}^{\oplus 3})$ and~$V\coloneqq\HH^0(\mathbb{P}^1,\mathcal{O}_{\mathbb{P}^1} \oplus \mathcal{O}_{\mathbb{P}^1}(1))$ because these sheaves are generated by their global sections, and the dimensions of the Hom's between the sheaves is the same as the dimension of the Hom's between their global sections.
\end{remark}

\section{Quaternion orders on \texorpdfstring{$\mathbb{F}_1$}{F1} as Clifford algebras}
\label{section:blowup-as-clifford}
We will now write the construction from \cite{MR3847233} in the case responsible for type~$m=2$ in~\eqref{equation:type-Bm} as a Clifford algebra. This allows us to compare it to the Clifford algebra obtained in \cref{section:nc-P1-as-clifford}, which is done in \cref{section:comparison}.

In \cref{subsection:construction-as-blowup} we quickly recall the construction of \cite{MR3847233}. In \cref{subsection:quaternionic-noncommutative-planes} we describe how we can describe the Artin--Schelter regular algebras in the special case where the order of the automorphism is~2. In \cref{subsection:blowing-up-clifford-algebras} we can then use the description as a graded Clifford algebra before blowing up to give a description as a Clifford algebra with values in a line bundle after blowing up, which allows us to write the associated abelian category in the same way as was obtained in \cref{section:nc-P1-as-clifford}. This can be summarised as follows.

\begin{equation*}
  \begin{array}{cr}
    \coh p^* \mathcal{S} \\
    \big\downarrow & \text{\cref{lemma:sheaf-is-clifford-algebra}} \\
    \coh p^* \big( \clifford_{\mathbb{P}^2}(\Sym^2(E) \otimes_k \mathcal{O}_{\mathbb{P}^2}, q, \mathcal{O}_{\mathbb{P}^2}(1))_0 \big) \\
    \big\downarrow & \text{\cref{subsection:blowing-up-clifford-algebras}} \\
    \coh\clifford_{\mathbb{F}_1}(\Sym^2(E) \otimes_k \mathcal{O}_{\mathbb{F}_1}, q, \mathcal{O}_{\mathbb{F}_1}(1))_0 \\
    \big\downarrow & \S \ref{subsection:clifford-algebras-with-values} \\
    \qgr\clifford_{\Sym(\mathcal{O}_{\mathbb{P}^1} \oplus \mathcal{O}_{\mathbb{P}^1}(1))}\big(E \otimes_k \Sym(\mathcal{O}_{\mathbb{P}^1} \oplus \mathcal{O}_{\mathbb{P}^1}(1)), q\big)
  \end{array}
\end{equation*}

\subsection{Construction of the surface as a blowup}
\label{subsection:construction-as-blowup}
In this section we quickly recall the construction from \cite{MR3847233}. The idea is to use the maximal order on~$\mathbb{P}^2$ induced from a quadratic~3\dash dimensional Artin--Schelter regular algebra finite over its center, and blow up a point~$p$ on~$\mathbb{P}^2$ \emph{outside} the ramification locus of the maximal order. The pullback of the maximal order is again a maximal order, now on~$\Bl_x\mathbb{P}^2$. Using a generalisation of Orlov's blowup formula to the case of maximal orders on smooth projective varieties we can then construct derived categories with a full and strong exceptional collection whose Cartan matrix is of the form described in \cite{1607.04246v1} for all~$m\geq 2$. For the purpose of this paper we are only interested in the case~$m=2$, where we wish to show that the construction from op.~cit.~is equivalent to the construction from~\cite{1503.03992v4}.

The setup for this situation is as follows. Let~$A$ be a~3\dash dimensional quadratic Artin--Schelter regular algebra, which is finite over its center.
Then in \cite{MR1880659} it was shown that the center of the sheaf of algebras induced on the Proj of the center of the graded algebra is always isomorphic to~$\mathbb{P}^2$. Hence we always obtain a maximal order~$\mathcal{S}$ on~$\mathbb{P}^2$. Because~$\coh\mathcal{S}\cong\qgr A$, the derived category~$\derived^\bounded(\coh\mathcal{S})$ has a full and strong exceptional collection, mimicking that of~$\mathbb{P}^2$. This maximal order has a ramification divisor~$C$ and an Azumaya locus~$\mathbb{P}^2\setminus C$. The divisor~$C$ is a cubic curve, which is an isogeny of another cubic curve~$E$ which is the point scheme used in the classification of these algebras.

The final step is to consider a point~$x$ in the Azumaya locus. There is a fat point module of the graded algebra~$A$ associated to this point, which defines a skyscraper sheaf supported at~$x$ whose fibre is a simple module for~$\mathcal{S}_p\otimes k(x)$. In \cite[\S3.3]{MR3847233} the first two authors have shown how it is possible to generalise Orlov's blowup formula to this setting. Let us denote the morphisms involved in the blowup as follows.

\begin{equation}
  \begin{tikzcd}
    E=\mathbb{P}^1 \arrow[d] \arrow[r] & \Bl_x\mathbb{P}^2 \arrow[d, "p"] \\
    x \arrow[r] & \mathbb{P}^2.
  \end{tikzcd}
\end{equation}

Then there exist a semiorthogonal decomposition of~$\derived^\bounded(\coh p^*\mathcal{S})$ in terms of~$\derived^\bounded(\coh\mathcal{S})$ and an exceptional object which is the (noncommutative) structure sheaf of the exceptional fibre. This gives the existence of a full and strong exceptional collection in~$\derived^\bounded(\coh p^*\mathcal{S})$. We can summarise the result of \cite{MR3847233} in the case we are interested in as follows, where the unexplained terminology regarding elliptic triples will be introduced shortly.

\begin{theorem}
  \label{theorem:belmans-presotto-construction}
  Let~$A$ be a quadratic~3\dash dimensional Artin--Schelter regular algebra associated to an elliptic triple~$(E,\sigma,\mathcal{L})$ for which the order of~$\sigma$ is~2. With the sheaf of maximal orders~$\mathcal{S}$ as above, and~$p\colon\Bl_x\mathbb{P}^2\to\mathbb{P}^2$ the blowup in a point~$x\in\mathbb{P}^2\setminus C$ in the Azumaya locus, there exists a full and strong exceptional collection
  \begin{equation}
    \label{equation:blowup-collection}
    \derived^\bounded(p^*\mathcal{S})
    =
    \left\langle
      p^*\mathcal{S}_0,
      p^*\mathcal{S}_1,
      p^*\mathcal{S}_2,
      p^*\mathcal{F}
    \right\rangle
  \end{equation}
  where~$\mathcal{S}_i\coloneqq\widetilde{A(i)}$ and~$\mathcal{F}$ is the skyscraper associated to a fat point module, and whose Gram matrix is mutation equivalent to type~$\mathrm{B}_2$.
\end{theorem}

The Gram matrix for the exceptional colletion in \eqref{equation:blowup-collection} is given by
\begin{equation}
  \begin{pmatrix}
    1 & 3 & 6 & 2 \\
    0 & 1 & 3 & 2 \\
    0 & 0 & 1 & 2 \\
    0 & 0 & 0 & 1
  \end{pmatrix}.
\end{equation}
It was shown in \cite[proposition~14]{MR3847233} that~$\mathrm{B}_2$ is mutation equivalent to this matrix.

\subsection{Quaternionic noncommutative planes}
\label{subsection:quaternionic-noncommutative-planes}
We wish to understand all quadratic AS-regular algebras that can appear in \cref{theorem:belmans-presotto-construction}, and we want to write everything in terms of Clifford algebras. There exists a complete classification of 3\dash dimensional quadratic Artin--Schelter regular algebras, due to Artin--Tate--Van den Bergh \cite{MR1086882} and Bondal--Polishchuk \cite{MR1230966}. The Artin--Tate--Van den Bergh approach to the classification is in terms of triples of geometric data~\cite[definition~4.5]{MR1086882}.
\begin{definition}
  An \emph{elliptic triple} is a triple~$(C,\sigma,\mathcal{L})$ where
  \begin{enumerate}
    \item $C$ is a divisor of degree~3\ in~$\mathbb{P}^2$;
    \item $\sigma\in\Aut(C)$;
    \item $\mathcal{L}$ is a very ample line bundle of degree~3 on~$C$.
  \end{enumerate}
  We say that it is a \emph{regular triple} if moreover
  \begin{equation} \label{equation:regular-triple}
    \mathcal{L}\otimes(\sigma^*\circ\sigma^*(\mathcal{L}))\cong\sigma^*(\mathcal{L})\otimes\sigma^*(\mathcal{L}).
  \end{equation}
\end{definition}
We are only interested in algebras which are finite over their center, and by \cite[theorem~7.1]{MR1128218} we have that this is the case if and only if the order of~$\sigma$ is finite.

\begin{remark}
  \label{remark:finite-order-curves}
  The condition that~$\sigma$ is of finite order gives a restriction on the curves which can appear in a regular triple: out of~9~possible point schemes, only~4~remain. These are
  \begin{enumerate}
    \item the elliptic curves,
    \item the nodal cubic,
    \item a conic and line in general position,
    \item a triangle of lines.
  \end{enumerate}
  This can be done explicitly by describing the automorphism groups of cubic curves, and then take into account the required compatibility between~$\sigma$ and~$\mathcal{L}$ to define a regular triple.

  Another way of obtaining this restriction is by using the description of the ramification data for maximal orders on~$\mathbb{P}^2$. In this case the Artin--Mumford sequence can be used to show that these~4~curves are the only possible ramification divisors.

  A third way in which it is possible to see this restriction is in terms of the \'etale local structure of a maximal order, which shows that the singularities are necessarily at most nodal \cite{MR1880659}.

  Finally, after proving that the specific algebras we are interested in (those for which~$\sigma$ is not just finite, but of order~2) can be described using graded Clifford algebras, we will see how for those algebras there is yet another way of seeing these~4~curves, namely as discriminants of nets of conics without basepoints.
\end{remark}

As the construction in \cite{1503.03992v4} only concerns the case~$m=2$ of type~\eqref{equation:type-Bm} in the classification of \cite{1607.04246v1}, we can restrict the general construction from \cite{MR3847233} to the case where the order of~$\sigma$ is~2. Then the fat points all have multiplicity~2. Let us introduce some terminology for these.
\begin{definition}
  Let~$A$ be a quadratic 3\dash dimensional Artin--Schelter regular algebra. We say that~$A$ (resp.~$\qgr A$) is \emph{quaternionic} if the automorphism in the associated elliptic triple is of order~2.
\end{definition}
Observe that Artin--Schelter regular algebras whose automorphism~$\sigma$ has order~6 \emph{also} have fat point modules of multiplicity~2 \cite[theorem~7.3]{MR1128218}. But using the proof of \cref{proposition:quaternionic-is-graded-clifford} we will show in \cref{lemma:order-6} that these algebras are not new when considered up to Zhang twist.

We will now show that all quaternionic noncommutative planes can be written as a graded Clifford algebra, up to Zhang twist. This result does not appear as such in the literature, but it follows almost immediately by combining the classification of elliptic triples and the description of ``Clifford quantum~$\mathbb{P}^2$'s'' in \cite[corollary~4.8]{MR2320448}. These are precisely the~3\dash dimensional graded Clifford algebras which are Artin--Schelter regular, i.e.~which are associated to a basepoint-free net of conics. The following proposition can be seen as a converse to this result. Here the condition that the characteristic of~$k$ is not~3 is important.

\begin{proposition}
  \label{proposition:quaternionic-is-graded-clifford}
  Let~$A = A(C,\sigma,\mathcal{L})$ be a quaternionic Artin--Schelter regular algebra. Then~$A$ is the Zhang twist of a graded Clifford algebra.

  \begin{proof}
    Recall from \cite[theorem~1.2]{MR2776790} that~$A$ is a Zhang twist of~$A'$ if and only if their associated~$\mathbb{Z}$\dash algebras~$\check{A}$ and~$\check{A'}$ are isomorphic. Now~$\check{A} = A(C,\mathcal{L},\sigma^*\mathcal{L})$ is a quadratic Artin--Schelter regular~$\mathbb{Z}$\dash algebra in the sense of \cite{MR2836401}. It is shown in \cite[theorem~3.5]{MR2836401} (based upon the results in \cite{MR1128218}) that there is a morphism of algebraic groups~$\eta\colon\Pic_0C\to\Aut C$ where~$\Pic_0C$ consists of all line bundles having degree~0 on each component of~$C$. Moreover we have by \cite[theorem~4.2.2]{MR2836401} that every quadratic Artin--Schelter regular~$\mathbb{Z}$\dash algebra~$A(C,\mathcal{L}_0,\mathcal{L}_1)$ is of the form~$\check{A'}$ where~$A'= A(C,\mathcal{L}_0,\sigma')$ and~$\sigma' \in \im(\eta)$. In particular~$A$ is a Zhang twist of a quadratic Artin--Schelter regular algebra\footnote{It is customary to refer to a quadratic Artin--Schelter regular algebra~$A=A(C,\mathcal{L},\sigma)$ as a \emph{translation algebra} whenever~$\sigma \in \im(\eta)$.}~$A'$ whose associated~$C$\dash automorphism lies in~$\im(\eta)$. We claim that~$A'$ is a graded Clifford algebra.

    A closer investigation of \cite[\S 5]{MR1128218} shows that for every~$\mathcal{G} \in \Pic_0C$ we have that~$\eta(\mathcal{G})$ is uniquely defined by the following property:
    \begin{equation}
      \label{equation:eta-for-nonsingular-points}
      \mathcal{O}_C(\eta(\mathcal{G})(p)) \cong \mathcal{G}(p)
    \end{equation}
    if~$p$ is a nonsingular point of~$C$ and
    \begin{equation}
      \eta(\mathcal{G})(p) = p
    \end{equation}
    if~$p$ is a singular point of~$C$. In particular if all points of~$C$ are singular (which only happens when~$C$ is a triple line), the only finite order automorphism in~$\im(\eta)$ is~$\id_C$. However in this case~$A'$ does not allow fat point modules. This is a contradiction, because the existence (and multiplicity) of fat point modules is an invariant of the category~$\QGr(A)$ and by construction there are equivalences of categories
    \begin{equation}
      \QGr(A) \cong \QGr(\check{A}) \cong \QGr(\check{A'}) \cong \QGr(A').
    \end{equation}
    Hence we can assume that~$C$ contains a nonsingular point~$p$. Now assume that~$\eta(\mathcal{G})^n = \id_C$, then by \eqref{equation:eta-for-nonsingular-points} we find
    \begin{equation}
      \mathcal{O}_C(p) \cong \mathcal{G}^{\otimes n}(p)
    \end{equation}
    and hence
    \begin{equation}
      \mathcal{O}_C \cong \mathcal{G}^{\otimes n}.
    \end{equation}
    This implies that~$\eta$ preserves the order of an element in~$\Pic_0C$. By \eqref{equation:regular-triple} we know that~$\mathcal{L}^{-1} \otimes \sigma^* \mathcal{L}$ has order~2\ in~$\Pic_0C$, and hence~$\sigma' = \eta(\mathcal{L}^{-1} \otimes \sigma^* \mathcal{L})$ has order~2 as well. By the description of~$\Pic_0C$ as e.g.~in \cite[table~1]{MR1230966}, we see that~$\Pic_0C$ only admits order~2 elements when~$C$ is an elliptic curve, a nodal cubic, the union of a conic and a line in general position or a triangle of lines, as explained in \cref{remark:finite-order-curves}.

    In the case of an elliptic curve,~$A'$ is a quaternionic Sklyanin algebra and it is shown in \cref{example:quaternionic-sklyanin} that such an~$A'$ is a graded Clifford algebra. For the other~3~options for the point scheme~$C$ there is a unique order~2~element in~$\Pic_0C$ and the associated algebra~$A'$ is the quotient of~$k\langle x,y,z \rangle$ by the relations
    \begin{equation}
        \left\{
          \begin{aligned}
            xy+yx&=0 \\
            yz+zy&=c_1x^2 \\
            zx+xz&=c_2y^2,
          \end{aligned}
        \right.
      \end{equation}
    where~$(c_1,c_2) = (1,1), (0,1), (0,0)$ for the 3~cases respectively. These algebras are graded Clifford algebra by \cref{example:quaternionic-special}.
  \end{proof}
\end{proposition}

The classification of \cite[corollary~4.8]{MR2320448} is a converse to \cref{proposition:quaternionic-is-graded-clifford}: it turns out that quaternionic translation algebras are precisely the graded Clifford algebras, and that they describe all quaternionic noncommutative planes.

\begin{remark}
  The algebra~$A'$ constructed in the above proposition is a graded~$k$\dash algebra with symmetric relations. As such \cref{lemma:symmetric-relations-to-clifford} and \cref{proposition:map-to-clifford-is-iso} show the existence of an epimorphism~$\varphi\colon A' \rightarrow \clifford_k(E,q)$. Both~$A'$ and the Clifford algebra~$\clifford_k(E,q)$ are graded~$k$\dash algebras with Hilbert series~$1/(1-t)^3$, hence~$\varphi$ is an isomorphism. This gives an alternative proof for showing that~$A'$ is a Clifford algebra, but it does not highlight the fact that~$\clifford_k(E,q)$ is a graded Clifford algebra.
\end{remark}

\begin{example}
  Not every quaternionic Artin--Schelter regular algebra is a graded Clifford algebra on the nose. By the proof of \cref{proposition:quaternionic-is-graded-clifford} we can also say that not every Artin--Schelter algebra is a translation algebra. As an example we can consider
  \begin{equation}
    k\langle x,y,z\rangle/(xz-zx,yz-zy,xy+yx).
  \end{equation}
  The point scheme~$C$ is the triangle of lines defined by~$xyz$, and the automorphism is given by rescaling by~$-1$ in one component, and exchanging the two others. In a translation algebra the automorphism needs to preserve the components and rescale them in the same way. The prescribed Zhang twist in this case is induced from the automorphism of the degree~1~part which has~$z\mapsto -z$, and we obtain the Clifford algebra associated to the net~$N^{\mathrm{E}}$ in \eqref{equation:nets-of-conics-singular-discriminant}.
\end{example}

We now show why the algebras with an automorphism of order~6 (and hence fat points of multiplicity~2) can be ignored for the purposes of this paper.
\begin{lemma}
  \label{lemma:order-6}
  Let~$A=A(C,\sigma,\mathcal{L})$ be a an Artin--Schelter regular algebra such that the order of~$\sigma$ is~6. Then~$A$ is the Zhang twist of a quaternionic Artin--Schelter regular algebra.

  \begin{proof}
    We have that~$\sigma=\sigma^3\circ\sigma^{-2}$, where~$\sigma^{-2}$ is of order~3. Moreover~$\sigma^{-2}$ commutes with~$\sigma$, and it extends to an automorphism of the ambient~$\mathbb{P}^2$ in which~$C$ is embedded using~$\mathcal{L}$. Hence up to a Zhang twist as in the proof of \cref{proposition:quaternionic-is-graded-clifford} we can also work with the elliptic triple~$(C,\sigma^{3},\mathcal{L})$, which is quaternionic.
  \end{proof}
\end{lemma}

\paragraph{Nets of conics}
Because we are only interested in the associated category~$\qgr A$, we can restrict ourselves to considering graded Clifford algebras if we are only interested in those algebras which have fat point modules of multiplicity~2. We will now recall the algebras from \cite[corollary~4.8]{MR2320448}.

In the case of~$n=3$ the linear system of quadrics is known as a net of conics, and there exists a complete classification of these \cite{MR0432666}. Because we are only interested in basepoint-free nets of conics we can summarise the part of this classification with the notation of op.~cit.~as follows. Recall that the \emph{discriminant} is the locus of the singular conics in the linear system. Such a singular conic is either two lines intersecting in a point or a double line. The double lines in the net of conics correspond precisely to the singularities of the discriminant for the basepoint-free nets. In \cref{table:basepoint-free-nets-of-conics} the situation is summarised for the types which are relevant to us, for a complete classification one is referred to~\cite[table~2]{MR0432666}.

\begin{table}
  \centering
  \begin{tabular}{ccc}
    \toprule
    type     & discriminant                       & number of double lines \\
    \midrule
    A        & elliptic curve                     & 0 \\
    B        & nodal cubic                        & 1 \\
    D        & conic and line in general position & 2 \\
    E        & triangle of lines                  & 3 \\
    \bottomrule
  \end{tabular}
  \caption{Base-point free nets of conics}
  \label{table:basepoint-free-nets-of-conics}
\end{table}

\begin{example}[Quaternionic Sklyanin algebras]
  \label{example:quaternionic-sklyanin}
  3\dash dimensional Sklyanin algebras associated at points of order~2 are an interesting class of graded Clifford algebras. It can be shown that such an algebra has a presentation of the form
  \begin{equation}
    \left\{
      \begin{aligned}
        xy+yx&=cz^2 \\
        yz+zy&=cx^2 \\
        zx+xz&=cy^2
      \end{aligned}
    \right.,
  \end{equation}
  where~$c\neq 0$,~$c^3\neq 8$ and~$c^3\neq-1$ \cite[example~3.5]{1412.7001v2}.

  If we identify~$x=x_1$, $y=x_2$ and~$z=x_3$ these algebras are Clifford algebras associated to the net of conics given by
  \begin{equation}
    M_1=
    \begin{pmatrix}
      2 & 0 & 0 \\
      0 & 0 & c \\
      0 & c & 0
    \end{pmatrix},
    M_2=
    \begin{pmatrix}
      0 & 0 & c \\
      0 & 2 & 0 \\
      c & 0 & 0
    \end{pmatrix},
    M_3=
    \begin{pmatrix}
      0 & c & 0 \\
      c & 0 & 0 \\
      0 & 0 & 2
    \end{pmatrix},
  \end{equation}
  or equivalently the symmetric matrix
  \begin{equation}
    M=
    \begin{pmatrix}
      2x^2 & cz^2 & cy^2 \\
      cz^2 & 2y^2 & cx^2 \\
      cy^2 & cx^2 & 2z^2
    \end{pmatrix}
  \end{equation}
  In this case the central elements~$y_i$ are~$x^2$, $y^2$ and~$z^2$ which are in turn the norms of the elements~$x,y,z$ in degree~1. Computing the determinant of this matrix to describe the point modules we get the discriminant curve
  \begin{equation}
    (2c^3 + 8)x^2y^2z^2 - 2c^2(x^6 + y^6 + z^6).
  \end{equation}
\end{example}

\begin{example}[Special quaternionic algebras]
  \label{example:quaternionic-special}
  We have seen in \cref{proposition:quaternionic-is-graded-clifford} that there are three other isomorphism classes of graded Clifford algebras which are relevant to us, besides the generic case of a quaternionic Sklyanin algebra. Their nets of conics are given by
  \begin{equation}
    \begin{aligned}
      M^{\mathrm{B}}
      &=
      \begin{pmatrix}
        2x^2 & 0 & y^2 \\
        0 & 2y & x^2 \\
        y^2 & x^2 & 2z^2
      \end{pmatrix},
      \\
      M^{\mathrm{D}}
      &=
      \begin{pmatrix}
        2x^2 & 0 & 0 \\
        0 & 2y^2 & x^2 \\
        0 & x^2 & 2z^2
      \end{pmatrix},
      \\
      M^{\mathrm{E}}
      &=
      \begin{pmatrix}
        2x^2 & 0 & 0 \\
        0 & 2y^2 & 0 \\
        0 & 0 & 2z^2
      \end{pmatrix}
    \end{aligned}
  \end{equation}
  where~B, D and~E refer to the classification in \cref{table:basepoint-free-nets-of-conics}. The particular choice of basis is made to be compatible with \cite[corollary~4.8]{MR2320448}. The nets of conics are in turn given by
  \begin{equation}
    \label{equation:nets-of-conics-singular-discriminant}
    \begin{aligned}
      N^{\mathrm{B}}
      &=
      \langle x^2+yz, y^2+xz, z^2\rangle,
      \\
      N^{\mathrm{D}}
      &=
      \langle x^2+yz,y^2,z^2\rangle,
      \\
      N^{\mathrm{E}}
      &=
      \langle x^2,y^2,z^2\rangle.
    \end{aligned}
  \end{equation}
\end{example}

\begin{remark}
  It is also possible to study the derived category of the Clifford algebra on~$\mathbb{P}^2$ using the derived category of the associated standard conic bundle, as in \cite[\S5.3]{MR3133293}.
\end{remark}

\subsection{Blowing up Clifford algebras}
\label{subsection:blowing-up-clifford-algebras}
Using \cref{proposition:quaternionic-is-graded-clifford} we get that the Artin--Schelter regular algebras we need to consider in the construction can be written as graded Clifford algebras, as our construction is insensitive to Zhang twist. We now show using a chain of easy lemmas how it is possible to write the maximal order~$p^*\mathcal{S}$ on~$\Bl_x\mathbb{P}^2$ as a Clifford algebra in the sense of \cref{subsection:clifford-algebras-with-ample-values}.

\begin{lemma}
  \label{lemma:sheaf-is-clifford-algebra}
  There exists an isomorphism
  \begin{equation}
    \mathcal{S}
    \cong
    \clifford_{\mathbb{P}^2}(E\otimes_k\mathcal{O}_{\mathbb{P}^2},q,\mathcal{O}_{\mathbb{P}^2}(1))_0.
  \end{equation}
\end{lemma}
In particular we also have that
\begin{equation}
  p^*\mathcal{S}
  \cong
  p^*\clifford_{\mathbb{P}^2}(E\otimes_k\mathcal{O}_{\mathbb{P}^2},q,\mathcal{O}_{\mathbb{P}^2}(1))_0.
\end{equation}

By functoriality of the Clifford algebra construction we then get the following description.
\begin{lemma}
  There exists an isomorphism
  \begin{equation}
    p^*\mathcal{S}
    \cong
    \clifford_{\Bl_x\mathbb{P}^2}(E\otimes_k\mathcal{O}_{\Bl_x\mathbb{P}^2},q,\mathcal{O}_{\Bl_x\mathbb{P}^2}(1))_0.
  \end{equation}

  \begin{proof}
    This follows from the isomorphism
    \begin{equation}
      p^*\clifford_{\mathbb{P}^2}(E\otimes_k\mathcal{O}_{\mathbb{P}^2},q,\mathcal{O}_{\mathbb{P}^2}(1))_0
      \cong
      \clifford_{\Bl_x\mathbb{P}^2}(E\otimes_k\mathcal{O}_{\Bl_x\mathbb{P}^2},q,\mathcal{O}_{\Bl_x\mathbb{P}^2}(H))_0.
    \end{equation}
  \end{proof}
\end{lemma}

Then by \cref{proposition:ample-comparison} we get the following description of the Clifford algebra, which puts the category of coherent sheaves over~$p^*\mathcal{S}$ on the same footing as that of the noncommutative~$\mathbb{P}^1$\dash bundle as in \cref{corollary:H-is-clifford}. This will allow us to compare the two constructions in \cref{section:comparison}.

\begin{corollary}
  \label{corollary:pS-is-clifford}
  There exists exists an equivalence of categories
  \begin{equation}
    \coh p^*\mathcal{S}
    \cong
    \qgr_{\mathbb{P}^1}\clifford_{\Sym_{\mathbb{P}^1}(\mathcal{O}_{\mathbb{P}^1}\oplus\mathcal{O}_{\mathbb{P}^1}(1))}(E\otimes_k\Sym_{\mathbb{P}^1}(\mathcal{O}_{\mathbb{P}^1}\oplus\mathcal{O}_{\mathbb{P}^1}(1)),q).
  \end{equation}

  \begin{proof}
    In \eqref{equation:F1-BlxP2} we have seen the classical isomorphism
    \begin{equation}
      \Bl_x\mathbb{P}^2\cong\relProj\Sym_{\mathbb{P}^1}(\mathcal{O}_{\mathbb{P}^1}\oplus\mathcal{O}_{\mathbb{P}^1}(1)).
    \end{equation}
    We apply \cref{proposition:ample-comparison} to the morphism~$\pi\colon\Bl_x\mathbb{P}^2\cong\mathbb{F}_1\to\mathbb{P}^1$, where
    \begin{equation}
      \begin{aligned}
        \mathcal{E}&\coloneqq E\otimes_k\mathcal{O}_{\mathbb{P}^2}, \\
        \mathcal{L}&\coloneqq p^*(\mathcal{O}_{\mathbb{P}^2}(1)),
      \end{aligned}
    \end{equation}
    and therefore
    \begin{equation}
      \mathcal{A}\cong\Sym_{\mathbb{P}^1}(\mathcal{O}_{\mathbb{P}^1}\oplus\mathcal{O}_{\mathbb{P}^1}(1))
    \end{equation}
    using the identification
    \begin{equation}
      \HH^0(\mathbb{P}^1,\mathcal{O}_{\mathbb{P}^1}\oplus\mathcal{O}_{\mathbb{P}^1}(1))=\HH^0(\mathbb{P}^2,\mathcal{O}_{\mathbb{P}^2}(1))
    \end{equation}
    In this situation we have that~$p^*(\mathcal{O}_{\mathbb{P}^2}(1)$ corresponds to the shift by~1\ in the sheaf of graded algebras~$\Sym_{\mathbb{P}^1}(\mathcal{O}_{\mathbb{P}^1}\oplus\mathcal{O}_{\mathbb{P}^1}(1))$ on~$\mathbb{P}^1$. The result follows.
  \end{proof}
\end{corollary}

We state and prove the following easy observation, and come back to this in \cref{remark:gldim-2}.
\begin{proposition}
  \label{proposition:gldim-2}
  The global dimension of~$p^*\mathcal{S}$ is~2.

  \begin{proof}
    Because we have blown up a point in the Azumaya locus, we have that~$\mathcal{S}|_{\mathbb{P}^2\setminus\{x\}}\cong (p^*\mathcal{S})|_{\Bl_x\mathbb{P}^2\setminus E}$. Because we can check global dimension in the stalks, we conclude that the global dimension is~$\leq 2$ in the points of~$\Bl_x\mathbb{P}^2\setminus E$. For points on~$E$ we use that the global dimension of an Azumaya algebra is equal to the global dimension of the ring.
  \end{proof}
\end{proposition}

\section{Comparing the two constructions}
\label{section:comparison}
By \cref{section:nc-P1-as-clifford,section:blowup-as-clifford} the categories~$\qgr(\mathbb{S}({}_f(\mathcal{O}_{\mathbb{P}^1}(1))_{\id}))$ and~$\coh(p^*\mathcal{S})$ are equivalent to a category of the form
\begin{equation}
  \label{equation:final-description}
  \qgr\Big( \clifford_{\Sym(\mathcal{O}_{\mathbb{P}^1} \oplus \mathcal{O}_{\mathbb{P}^1}(1))}\big(E \otimes_k \Sym(\mathcal{O}_{\mathbb{P}^1} \oplus \mathcal{O}_{\mathbb{P}^1}(1)), q\big) \Big)
\end{equation}
where~$q$ is obtained from a surjective map~$\Sym^2(E) \rightarrow V$. Using the isomorphism~$\Bl_x\mathbb{P}^2\cong\mathbb{F}_1$ and the blowup~$p\colon\Bl_x\mathbb{P}^2\to\mathbb{P}^2$ we also write the associated category using
\begin{equation}
  \label{equation:final-description-as-blowup}
  p^*\clifford_{\mathbb{P}^2}(E\otimes_k\mathcal{O}_{\mathbb{P}^2},q,\mathcal{O}_{\mathbb{P}^2}(1))_0.
\end{equation}
by inverting the construction in \cref{subsection:blowing-up-clifford-algebras}.

Using this description the comparison of the abelian categories boils down to understanding how the geometric data defining both categories can be compared. We do this by explaining how we can go back from the data as in \eqref{equation:final-description} to the morphism~$f\colon\mathbb{P}^1\to\mathbb{P}^1$, and how this gives an inverse to the results of \cref{section:nc-P1-as-clifford}.

In \cref{subsection:geometry-linear-systems} we discuss the geometry of the linear systems we are interested in: we will need to understand the classification of pencils of binary quartics \cite{MR1636607} and pencils of conics \cite{MR0432666}, in order to set up a correspondence between pencils of binary quartics on one hand, and nets of conics together with the choice of a smooth conic in the net on the other. To set up the correspondence we need to discuss a branched covering of the dual of the net of conics in some detail, this is done in \cref{subsection:branched-coverings}. This allows us to obtain the comparison between the geometric data in \cref{subsection:geometric-comparison}. Finally we discuss the comparison of the categories in \cref{subsection:categorical-comparison}.

In \cref{subsection:final-remarks} we discuss how the constructions which are compared in this paper relate to other notions of noncommutative blowing up and noncommutative~$\mathbb{P}^1$\dash bundles. In particular, there is a conjectural relationship between the constructions in this paper and these other notions, obtained by suitably degenerating the construction. And we discuss how the Hochschild cohomology of these categories is expected to behave, based on the automorphisms of the geometric data.

\subsection{Nets of conics and pencils of binary quartics}
\label{subsection:geometry-linear-systems}
In this section we will recall some results on linear systems which we will need to relate the two constructions. The two types which we want to compare are nets of conics (which are used to define graded Clifford algebras) and pencils of binary quartics (which are used to define noncommutative~$\mathbb{P}^1$\dash bundles). In \cref{subsection:branched-coverings} we will construct a natural~$4:1$\dash cover of~$\mathbb{P}^2$ by~$\mathbb{P}^2$ associated to a net of conics, which will allow us to compare pencils of binary quartics to nets of conics together with the choice of a smooth conic in the net, as in \cref{subsection:geometric-comparison}.

\paragraph{Pencils of conics} {\ } \\
As a tool in setting up the correspondence between nets of conics together with a point and pencils of binary quartics, we need to say a few words on pencils of conics. These are~2\dash dimensional subspaces of~$\HH^0(\mathbb{P}^2,\mathcal{O}_{\mathbb{P}^2}(2))$. Their classification is completely classical. It is also the first (non-trivial) case of Segre's classification of pencils of quadrics using Segre symbols. It is given in \cref{table:pencils-of-conics}, and the main observation is that there is a correspondence between the base locus and the types of singular conics which appear in the pencil.

\begin{table}[ht]
  \centering
  \begin{tabular}{cccc}
    \toprule
    Segre symbol & base locus  & \# singular fibres & \# double lines \\
    \midrule
    $[1,1,1]$    & $(1,1,1,1)$ & 3                  & 0 \\
    $[2,1]$      & $(2,1,1)$   & 2                  & 0 \\
    $[3]$        & $(3,1)$     & 2                  & 0 \\
    $[(1,1),1]$  & $(2,2)$     & 2                  & 1 \\
    $[(2,1)]$    & $(4)$       & 1                  & 1 \\
    \bottomrule
  \end{tabular}
  \caption{Pencils of conics}
  \label{table:pencils-of-conics}
\end{table}
The number of singular fibres and the number of double lines will allow us to describe the ramification of~$f\colon\mathbb{P}^1\to\mathbb{P}^1$ as in \cref{table:case-by-case}.

\paragraph{Nets of conics} {\ } \\
In \cref{subsection:graded-clifford-algebras}~$n-1$\dash dimensional linear systems of~$n-1$\dash dimensional quadrics were used to define graded Clifford algebras. We are interested in the special case of nets of conics, i.e.~of~3\dash dimensional subspaces of~$\HH^0(\mathbb{P}^2,\mathcal{O}_{\mathbb{P}^2}(2))$.

We will consider a net of conics as a surjective morphism~$\phi\colon\Sym^2E\to V$, where~$E$ and~$V$ are~3\dash dimensional vector spaces. Then the net itself is~$\mathbb{P}(V^\vee)$, whilst the conics live in~$\mathbb{P}(E)$. Associated to this there is the \emph{discriminant} of the net, which is the locus in~$\mathbb{P}(V^\vee)$ of the singular conics in the net. It is also known as the \emph{Hessian curve}.

The \emph{Jacobian} is the union of the singular points of each conic in~$\mathbb{P}(E)$. Both the discriminant and Jacobian are cubics, but they live in different spaces and they are not necessarily of the same type.

In \cref{subsection:branched-coverings} we will construct the morphism~$\Theta\colon\mathbb{P}(E)\to\mathbb{P}(V)$, sending a point to the subpencil which has this point in its baselocus. This is a~$4:1$\dash cover. The \emph{branch curve} inside~$\mathbb{P}(V)$ is the locus of points whose fibres do not have~4~distinct points, Generically the branch curve is the dual of the discriminant, hence it is of degree~6, but if the discriminant is zero the degree is lower.

The \emph{ramification curve} inside~$\mathbb{P}(E)$ are those points which are multiple points of a fiber. Using \cref{table:pencils-of-conics} this also means that the associated subpencil of conics has at most two singular points, hence the ramification curve is equal to the Jacobian. In \cref{table:curves} we give an overview of all the curves associated to a net of conics. Observe that the point scheme for an Artin--Schelter regular Clifford algebra is a double cover of the ramification curve, and hence coincides with the double cover of the discriminant constructed by considering the two lines in~$\mathbb{P}(E)$ associated to the singular conic parametrised by a point on the discriminant. Hence the curves living in the same ambient~$\mathbb{P}^2$ coincide.

\begin{table}
  \centering
  \begin{tabular}{clcc}
    \toprule
                         & curve                        & ambient space        & degree \\
    \midrule
    $\Delta$             & discriminant                 & $\mathbb{P}(V^\vee)$ & 3 \\
    $J$                  & Jacobian                     & $\mathbb{P}(E)$      & 3 \\
    $R$                  & ramification curve           & $\mathbb{P}(E)$      & 3\\
    $B$                  & branch curve                 & $\mathbb{P}(V)$      & $\leq 6$ \\
    $C$                  & point scheme                 & $\mathbb{P}(E^\vee)$ & 3 \\
    $\widetilde{\Delta}$ & double cover of discriminant & $\mathbb{P}(E^\vee)$ & 3 \\\addlinespace
                         & conic in the net             & $\mathbb{P}(E)$      & 2 \\
    \bottomrule
  \end{tabular}
  \caption{Curves associated to a net of conics}
  \label{table:curves}
\end{table}

The classification of nets of conics is obtained in \cite{MR0432666}. We are only interested in basepoint-free nets of conics, otherwise the associated graded Clifford algebra is not Artin--Schelter regular. The classification in this case was summarized in \cref{table:basepoint-free-nets-of-conics}, with the labeling from \cite{MR0432666}. The double lines in the net of conics necessarily correspond to the singularities in the discriminant for a basepoint-free net of conics.


We now give an example of a net of conics. It will be important in understanding a special element in the classification of noncommutative~$\mathbb{P}^1$\dash bundles, see also \cref{example:type-A-pencil-of-quartics,subsection:final-remarks}.

\begin{example}
  \label{example:type-D-net-of-conics}
  Consider the net of conics given by~$x^2,y^2,z^2+2xy$. It is (up to base change) the only net of conics of type~D. Its discriminant is a conic and line in general position, and the two singularities of the discriminant correspond to the double lines~$x^2$ and~$y^2$. The Jacobian turns out to be a triangle of lines, but we will not use this.
\end{example}

\paragraph{Pencils of binary quartics} {\ } \\
The only input needed for the construction of a noncommutative~$\mathbb{P}^1$\dash bundle in the sense of \cite{1503.03992v4} is a finite morphism~$f\colon\mathbb{P}^1\to\mathbb{P}^1$ of degree~4. This data is equivalent to the data of a basepoint-free pencil of binary quartics, i.e.~a~2\dash dimensional subspace of~$\HH^0(\mathbb{P}^1,\mathcal{O}_{\mathbb{P}^1}(4))$. Indeed, if~$\langle f_1,f_2\rangle$ is such a pencil, then~$[x:y]\mapsto[f_1(x,y):f_2(x,y)]$ is a well-defined finite morphism~$f\colon\mathbb{P}^1\to\mathbb{P}^1$. These pencils have been studied in detail in \cite{MR1636607}, and we will quickly recall the relevant results here.

As for nets of conics, we can associate the \emph{Jacobian} and the \emph{discriminant} to a pencil of binary quartics. In this case the Jacobian represents the branch points of the associated morphism, and the discriminant are the images of these. We can moreover describe the branching behaviour of the associated morphism using its \emph{symbol}: algebraically it is obtained by considering the roots of the discriminant and listing the multiplicities of the remaining factors. Geometrically it describes the cycle type of the monodromy around each branch point. Its notation is not to be confused with that of Segre symbols as in \cref{table:pencils-of-conics}.

The classification is given in \cref{table:case-by-case}, where we have already included the comparison to the net of conics, from \cref{theorem:quadruples-are-isomorphic}.

\begin{example}
  \label{example:type-A-pencil-of-quartics}
  The ``most degenerate'' morphism~$f\colon\mathbb{P}^1\to\mathbb{P}^1$ is given by~$[x^4:y^4]$. It is ramified in~$[0:1]$ and~$[1:0]$ only, where~4 branches come together, hence its symbol is~$[(4)(4)]$. It corresponds to type~A in the classification of \cref{table:case-by-case}.
\end{example}

\begin{remark}
  In \cite[\S2]{MR1636607} it is remarked that there is no pencil of binary quartics whose symbol is~$[(2,2)(2,2)(3,1)]$, which is excluded by considering the monodromy of the would-be corresponding morphism~$f\colon\mathbb{P}^1\to\mathbb{P}^1$ of degree~4. Using the correspondence we have set up between nets of conics and pencils of binary quartics we can also argue that this case cannot occur using the geometry of cubics: there is no cubic curve with~2~singularities and~1~inflection point.
\end{remark}

%

\subsection{Branched coverings and nets of conics}
\label{subsection:branched-coverings}
Let~$V$ and~$E$ be~3\dash dimensional vector spaces, and let
\begin{equation}
  \phi\colon\Sym^2E\to V
\end{equation}
be a surjection. This corresponds to a net of conics inside~$\mathbb{P}(E)$ parametrised by~$\mathbb{P}(V^\vee)$. We will assume that the net is basepoint-free, such that it defines a graded Clifford algebra which is Artin--Schelter regular, as in \cref{subsection:graded-clifford-algebras}.

Define
\begin{equation}
  \begin{aligned}
    S&\coloneqq\Sym(E^\vee), \\
    R&\coloneqq\Sym(V^\vee)
  \end{aligned}
\end{equation}
and consider~$S$ as a graded~$R$\dash algebra via~$\phi^\vee\colon V^\vee \rightarrow \Sym^2 E^\vee$. Moreover we define
\begin{equation}
  A\coloneqq S/SR_{\geq 1}.
\end{equation}

For the notion of a (graded) Frobenius algebra we refer to \cite[\S3]{MR1388568}.
\begin{lemma}
  \label{lemma:graded-Frobenius}
  The algebra~$A$ is a graded Frobenius algebra with Hilbert series~$1,3,3,1$.
\end{lemma}

\begin{proof}
  Because the net of conics is basepoint-free, any choice of basis for the net gives a regular sequence in~$S$. The Hilbert series of the associated complete intersection is~$(1+t)^3$, e.g.~using \cite{MR0485835}. Likewise by basepoint-freeness we have that~$A$ is Frobenius because it defines a zero-dimensional complete intersection (hence it is Gorenstein).
\end{proof}


If~$x\in S$, then we will denote~$\overline{x}$ for the induced element in~$A$. If~$\overline{x}\in A_0$ or~$A_1$, then we will also use the notation~$\overline{x}$ to indicate the corresponding element in~$S_0$ or~$S_1$ depending on the context.

There exists a morphism
\begin{equation}
  \trace\colon S\to S/(RS_0\oplus RS_1\oplus RS_2)
\end{equation}
which defines a non-degenerate pairing
\begin{equation}
  \label{equation:trace-pairing}
  S\otimes_RS\to R(-3)\otimes A_3:p\otimes q\mapsto\trace(pq),
\end{equation}
where we have used the isomorphism
\begin{equation}
  S/(RS_0\oplus RS_1\oplus RS_2)\cong R(-3)\otimes A_3.
\end{equation}

We will use dual~$R$\dash bases~$\{e_0,\dotsc,e_7\}$ and~$\{f_0,\dotsc,f_7\}$ for~$S$, where the degrees of the elements~$e_i, f_j$ are
\begin{equation}
 |e_i|
  = |f_{7-i}| =
  \begin{cases}
    0 & i=0 \\
    1 & i=1,2,3 \\
    2 & i=4,5,6 \\
    3 & i=7.
  \end{cases}
\end{equation}
Moreover we can and will assume that~$e_0=f_7=1$ (and therefore~$\overline{e}_7=\overline{f}_0$), $f_4=e_1$, $f_5=e_2$, $f_6=e_3$, and~$\overline{f}_1=\overline{e}_4$, $\overline{f}_2=\overline{e}_5$, $\overline{f}_3=\overline{e}_6$. If we introduce the structure constants
\begin{equation}
  \overline{e}_i\overline{e}_j=\sum_{l=1}^3c_{i,j,l}\overline{e}_{l+3}
\end{equation}
for the multiplication in~$A$, where~$i,j=1,2,3$, then we have the following lemma.
\begin{lemma}
  \label{lemma:structure-constants-permutation-invariant}
  The structure constants~$c_{i,j,l}$ are invariant under permutation of~$i,j$ and~$l$.
\end{lemma}

\begin{proof}
  We evaluate~$\trace(\overline{e_ae_be_d})$ for~$a,b,d=1,2,3$:
  \begin{equation}
    \begin{aligned}
      \trace(\overline{e_ae_be_d})
      &=\sum_{i=1}^3 c_{a,b,i} \trace ( \overline{e_{i+3}e_d} ) \\
      &=\sum_{i=1}^3 c_{a,b,i} \trace ( \overline{f_ie_d}  )\\
      &=\sum_d c_{a,b,i} \delta_{i,d} \\
      &=c_{a,b,d}.
    \end{aligned}
  \end{equation}
\end{proof}

By construction we have the identifications
\begin{equation}
\label{equation:comparison:AiE1}
  \begin{aligned}
    A_1&=E^\vee, \\
    A_2&=\Sym^2E^\vee/V^\vee, \\
    A_3&=\Sym^3E^\vee/E^\vee V^\vee,
  \end{aligned}
\end{equation}
and the multiplication~$\mult\colon A_1\otimes_kA_1\to A_2$ is nothing but the quotient map~$\Sym^2E^\vee\to\Sym^2E^\vee/V^\vee$, with the inclusion of~$V^\vee$ into~$\Sym^2E^\vee$ given by~$\phi^\vee$. Because~$A$ is graded Frobenius by \cref{lemma:graded-Frobenius} the multiplication in the algebra provides us with a duality in the form of a perfect pairing
\begin{equation}
  A_1\otimes A_2\to A_3
\end{equation}
which yields the identifications
\begin{equation}
\label{equation:comparison:AiE2}
  \begin{aligned}
    A_1^\vee&=A_2\otimes A_3^{-1}, \\
    A_2^\vee&=A_1\otimes A_3^{-1}.
  \end{aligned}
\end{equation}
If we define
\begin{equation}
  \alpha\colon
  A_1\otimes A_3^{-1}\to\Sym^2A_2\otimes A_3^{-2}:
  e_i\otimes\overline{e}_7^{-1}\mapsto\sum_{j=1}^3\overline{e}_i\overline{e}_j\cdot\overline{f}_j\otimes\overline{e}_7^{-2}
\end{equation}
and
\begin{equation}
  \label{equation:beta}
  \beta\colon
  \Sym^2A_2\otimes A_3^{-2}\to V:
  \overline{e}_{i+3}\cdot\overline{e}_{j+3}\otimes\overline{e}_7^{-2}\mapsto\phi(e_i^\vee\cdot e_j^\vee)
\end{equation}
then these morphisms are compatible with the algebra structure on~$A$ in the following way using \cref{lemma:structure-constants-permutation-invariant}.
\begin{lemma}
  \label{lemma:commutative-diagram-algebra-structure}
  The diagram
  \begin{equation}
    \begin{tikzcd}
      0 \arrow[r] & A_1\otimes A_3^{\otimes-1} \arrow[r, "\alpha"] \arrow[d, equals] & \Sym^2A_2\otimes A_3^{\otimes-2} \arrow[r, "\beta"] \arrow[d, equals] & V \arrow[d, equals] \arrow[r] & 0 \\
      0 \arrow[r] & A_2^\vee \arrow[r, "\mult^\vee"] & \Sym^2A_1^\vee \arrow[r, "\phi"] & V \arrow[r] & 0
    \end{tikzcd}
  \end{equation}
  where the vertical arrows are the identifications obtained from the graded Frobenius structure, is commutative.
\end{lemma}

\begin{proof}
  We first discuss the square involving~$\alpha$. Note that
  \begin{equation}
    \begin{aligned}
      \alpha(e_i\otimes \overline{e}_7^{-1})
      &=\sum_{j=1}^3 \overline{e_ie_j}\cdot \overline{e}_{j+3}\otimes \overline{e}_7^{-2} \\
      &=\sum_{l,j=1}^3 c_{i,j,l} \overline{e}_{l+3} \cdot \overline{e}_{j+3}\otimes \overline{e}_7^{-2}
    \end{aligned}
  \end{equation}
  whereas the lower composition in this square is given by
  \begin{equation}
    \begin{aligned}
      e_i\otimes \overline{ e}^{-1}_7
      &\rightarrow \overline{e}_{i+3}^\vee \\
      &\rightarrow \sum_{p,q=1}^3 c_{p,q,i} \overline{e}_p^\vee\cdot \overline{e}_q^\vee \\
      &\rightarrow \sum_{p,q=1}^3 c_{p,q,i} \overline{e}_{p+3}\cdot \overline{e}_{q+3}\otimes \overline{e}^{-2}_7.
    \end{aligned}
  \end{equation}
  It now suffices to use the symmetry of the structure constants as in \cref{lemma:structure-constants-permutation-invariant}.

  The fact that the square involving~$\beta$ commutes follows immediately from the definition of~$\beta$ and the fact that the middle vertical identification is given by
  \begin{equation}
    \overline{e_{i+3}} \cdot \overline{e_{j+3}} \otimes \overline{e_7}^{-2} \mapsto \overline{e_i}^\vee \overline{e_j}^\vee.
  \end{equation}
\end{proof}

The trace pairing from \eqref{equation:trace-pairing} gives an identification
\begin{equation}
  \label{equation:S-dual}
  S^\vee\coloneqq\Hom_R(S,R) \cong S(3)\otimes A_3^{\otimes-1}: e_i^\vee \mapsto f_i\otimes e_7^{-1}
\end{equation}
of graded~$S$\dash modules. Using this identification, together with the~$R$\dash dualised multiplication~$S\otimes_RS\to S$ we get a copairing
\begin{equation}
  \label{equation:delta}
  \delta\colon S(-3)\otimes A_3\to S\otimes_RS
\end{equation}
such that
\begin{equation}
  \delta(u\otimes\overline{e}_7)=\sum_{i=0}^7ue_i\otimes f_i.
\end{equation}
Similar to \eqref{equation:ei-and-fi-commute} and \cite[lemma~3.24]{1503.03992v4} we see that~$\sum_{i=0}^7e_i\otimes f_i=\sum_{i=0}^7f_i\otimes e_i$ is a central element in~$S\otimes_RS$.

We will now do the preliminary algebraic constructions for the results of \cref{subsection:geometric-comparison}, by suitably describing everything in terms of graded modules. Consider the decomposition~$S=S_{\mathrm{even}}\oplus S_{\mathrm{odd}}$ of graded~$R$\dash modules into the part of even and odd grading. Then there exists a canonical identification
\begin{equation}
  S_{\mathrm{even}}/R=(R\otimes A_2)(-2).
\end{equation}
We moreover define a graded~$R$\dash module~$\Omega$ as the cokernel in the leftmost morphism of the Koszul sequence, i.e.
\begin{equation}
  \label{equation:truncated-koszul}
  0\rightarrow R(-6) \overset{\kappa}{\rightarrow} V\otimes R(-4)\rightarrow\Omega\rightarrow 0.
\end{equation}
We have that~$\Omega$ induces~$\mathrm{T}_{\mathbb{P}(V)}(-3)$ after sheafification, which in turn is isomorphic to~$\Omega_{\mathbb{P}(V)}^1$.

Finally consider the morphism
\begin{equation}
  \Phi\colon S_{\mathrm{odd}}(-3)\otimes A_3\to\Sym^2A_2\otimes_k R(-4)
\end{equation}
which is the composition of the inclusion into~$S(-3)\otimes A_3$, the morphism~$\delta$ from \eqref{equation:delta}, the projection~$S \twoheadrightarrow S_{\mathrm{even}}$ and the quotient~$S_{\mathrm{even}} \twoheadrightarrow S_{\mathrm{even}}/R \cong A_2 \otimes_k R(-2)$ in both factors of the tensor product, and the quotient~$A_2\otimes_kA_2\to\Sym^2A_2$.

\begin{proposition} \label{proposition:coker-Phi}
  There exists an isomorphism
  \begin{equation}
    \coker\Phi\cong\Omega\otimes A_3^{\otimes 2}.
  \end{equation}
  More precisely, the induced morphism
  \begin{equation}
    \Sym^2A_2\otimes_k R(-4)\to\coker\Phi\cong\Omega\otimes A_3^{\otimes2}
  \end{equation}
  is the composition of~$\beta\otimes\id_{R(-4)}$ and the quotient map in \eqref{equation:truncated-koszul}.
\end{proposition}

\begin{proof}
  The module~$S_{\mathrm{odd}}$ contains canonically a graded~$R$\dash submodule~$A_1\otimes R(-1)$.  We first describe how~$\Phi$ acts on~$A_1\otimes A_3\otimes R(-4)$. Using the definition this map is given by
  \begin{equation}
    \overline{e}_i\otimes \overline{e}_7 \overset{\delta}{\mapsto} \sum_{j=0}^7 e_ie_j\cdot f_j \mapsto \sum_{j=0}^7 \overline{e_ie_j}\cdot \overline{f}_j  = \sum_{j=1}^3 \overline{e_ie_j}\cdot \overline{f}_j
  \end{equation}
  for~$i=1,2,3$ (where we used the fact that~$e_i e_j \not \in S_{\mathrm{even}}$ for~$j=0,4,5,6$ and~$e_ie_7 \in R$).

  In other words it is, up to tensoring with~$A_3^{-2}$, precisely the map~$\alpha$ in \cref{lemma:commutative-diagram-algebra-structure}. More precisely, it is given by the unique map~$\alpha'$ making the following diagram commute
\begin{equation}
    \begin{tikzcd}
      A_1\otimes A_3^{\otimes-1} \otimes R(-4) \arrow[r, "\alpha \otimes \id_{R(-4)}"] \arrow[d, "\cong"] & \Sym^2A_2\otimes A_3^{\otimes-2} \otimes R(-4)  \arrow[d, "\cong"] \\
      A_1 \otimes A_3\otimes R(-4) \arrow[r, "\alpha'"] & \Sym^2A_2 \otimes R(-4)
    \end{tikzcd}
  \end{equation}

  Next consider the following commutative diagram
  \begin{equation}
    \begin{tikzcd}
      & 0 \\
      & R(-6)\otimes A_3\otimes A_3 \arrow[u] \\
      & S_{\mathrm{odd}}(-3)\otimes A_3 \arrow[u] \arrow[r, "\Phi"] & \Sym^2A_2\otimes R(-4) \arrow[r] & \coker\Phi \arrow[r] & 0 \\
      0 \arrow[r] & A_1\otimes A_3\otimes R(-4) \arrow[u] \arrow[r, "\alpha'"] & \Sym^2A_2\otimes R(-4) \arrow[u, equals] \arrow[r, "\beta"] & V\otimes R(-4)\otimes A_3^{\otimes 2} \arrow[u] \arrow[r] & 0 \\
      & 0 \arrow[u]
    \end{tikzcd}
  \end{equation}

  In particular~$\beta \circ \Phi$ induces a map~$R(-6) \otimes A_3^2 \rightarrow V \otimes R(-4) \otimes A_3^2$.

  We claim that this map is trivial on~$A_3^{\otimes 2}$:
  \begin{equation}
    \begin{aligned}
      \overline{e}_7\otimes \overline{e}_7
      &\overset{\Phi}{\mapsto} \sum_{i=4}^6 \overline{e_7f_i}\otimes \overline{e_i} \\
      &=\sum_{i,j=4}^6 \trace(e_7f_if_j) \overline{e_j}\otimes \overline{e_i} \\
      &\overset{\beta}{\mapsto} \sum_{i,j=4}^6\trace (e_7f_if_j)\phi(\overline{e}_{j-3}^\vee \cdot e_{i-3}^\vee)\otimes e^2_7 \\
      &=\sum_{i,j=4}^6\trace(e_7f_if_j)\phi(f_{j}^\vee \cdot f_{i}^\vee)\otimes \overline{e}^2_7.
    \end{aligned}
  \end{equation}
  The right hand side should be considered as an element of~$R_2\otimes V\otimes A_3^{\otimes 2}$.

  Now with the notation of \cref{lemma:linear-algebra} (which is an easy result from linear algebra) applied to~$\alpha=\trace(e_7-):S^2 E^\vee\rightarrow R_2=V^\vee$, $\beta=\phi^\vee:V^\vee \rightarrow S^2E^\vee$ and~$(a_l)_l=(f_i\cdot f_j)_{i,j=4,5,6}$ we find
  \begin{equation}
    \sum_{i,j=4}^6\trace(e_7f_if_j)\otimes \phi(f_{j}^\vee \cdot f_{i}^\vee)\otimes \overline{e}^2_7=\id_V\otimes \overline{e}_7^2
  \end{equation}
  with~$\id_V$ considered as an element of~$V^\vee\otimes V$. As such, by the construction of the Koszul sequence we have that~$\beta \circ \Phi\colon R(-6) \otimes A_3^2 \rightarrow V \otimes R(-4) \otimes A_3^2$ can be decomposed as~$\kappa \otimes \id_{A_3^{\otimes 2}}$ with~$\kappa$ as in \eqref{equation:truncated-koszul}. In particular~$\coker (\beta \circ \Phi) \cong \coker(\kappa) = \Omega$. Conversely the image of~$\beta \circ \Phi$ is given by~$\ker \gamma$, such that~$\coker(\beta \circ \Phi) \cong \im(\gamma) = \coker(\Phi)$, proving the lemma.
\end{proof}

\begin{lemma}
  \label{lemma:linear-algebra}
  Let~$V_1,V_2,V_3$ be finite-dimensional vector spaces. Let~$\alpha\colon V_1\to V_2$ and~$\beta\colon V_3\to V_1$ be linear maps. Let~$v_1,\dotsc,v_n$ be a basis for~$V_1$. Then
  \begin{equation}
    \sum_{i=1}^n\alpha(v_i)\otimes\beta^\vee(v_i^\vee)
  \end{equation}
  is equal to
  \begin{equation}
    \alpha\circ\beta\colon V_3\to V_2
  \end{equation}
  considered as elements of~$V_3^\vee\otimes_kV_2$.
\end{lemma}

We will now use these algebraic results, to describe explicitly a morphism
\begin{equation}
  \label{equation:morphism-g}
  g\colon\mathbb{P}(E)\to\mathbb{P}(V)
\end{equation}
which is a~$4:1$\dash cover. This is a morphism which exists more generally for nets of curves on smooth projective surfaces, as explained in \cite[\S11.4.4]{MR3617981}, and it will be the crucial ingredient in comparing the two constructions. As it turns out there is a purely geometric result comparing morphisms~$f\colon\mathbb{P}^1\to\mathbb{P}^1$ of degree~4 to nets of conics and the choice of a smooth conic in the net, which is proven in \cref{subsection:geometric-comparison}, whilst the categorical incarnation comparing the two constructions is given in \cref{subsection:categorical-comparison}.

Consider~$Z\coloneqq\Proj S=\mathbb{P}(E)$ and~$T\coloneqq\Proj R=\mathbb{P}(V)$, where we equip~$R$ with the doubled grading. From~$\phi\colon\Sym^2E \twoheadrightarrow V$ we get a corresponding morphism~$g\colon Z\to T$.
\begin{lemma}
  \label{lemma:g-gives-pencil-of-conics}
  The morphism~$g\colon Z\to T$ sends a point to the pencil of conics for which it is a basepoint.
\end{lemma}

\begin{proof}
  Note that every element in~$\Sym^2E^\vee$ (or rather~$(\Sym^2E^\vee)/k^\times$) defines a conic in~$\mathbb{P}(E)$. In particular~$\phi^\vee \colon V^\vee \hookrightarrow \Sym^2E^\vee$ identifies~$V^\vee$ (or rather~$V^\vee/k^\times$) as a net of conics. More concretely, fix bases~$x_1, x_2, x_3$ and~$v_1, v_2, v_3$ for~$E$ and~$V$ respectively and write~$\phi(x_i x_j) = \alpha_{1,i,j} v_1 + \alpha_{2,i,j} v_2 + \alpha_{3,i,j} v_3$. Then~$V^\vee$ is the net of conics spanned by
  \begin{equation}
    \begin{aligned}
      v_1^\vee &= \sum_{1=i \leq j}^3 \alpha_{1,i,j} x_i^\vee x_j^\vee, \\
      v_2^\vee &= \sum_{1=i \leq j}^3 \alpha_{2,i,j} x_i^\vee x_j^\vee, \\
      v_3^\vee &= \sum_{1=i \leq j}^3 \alpha_{3,i,j} x_i^\vee x_j^\vee.
    \end{aligned}
  \end{equation}
  An element~$[t_1:t_2:t_3] \in T = \mathbb{P}(V)$ corresponds to the following one dimensional subspace of~$V^\vee/k^\times$ (i.e.~the pencil of conics)
  \begin{equation}
    \label{equation:pencil-of-conics-dual}
    \{ a_1 v_1^\vee + a_2 v_2^\vee + a_3 v_3^\vee \mid a_1t_1 + a_2t_2 + a_3t_3 = 0 \} / k^\times
  \end{equation}
  Now~$\phi^\vee$ induces a map~$g\colon Z\to T$ as follows:
  \begin{equation}
    \label{equation:g-factor-trough-veronese}
    Z =\mathbb{P}(E) \xrightarrow{\textrm{Veronese embedding}} \mathbb{P}(\Sym^2 E) \xrightarrow{\Proj(\Sym(\phi^\vee))} \mathbb{P}(V)
  \end{equation}
  In particular~$g$ acts as follows
  \begin{equation}
    \label{equation:action-of-g}
    \begin{aligned}
      [z_1:z_2:z_3]
      &\mapsto [z_1^2:z_2^2:z^2_3:z_1z_2:z_1z_3:z_2z_3] \\
      &\mapsto [ v_1^\vee(z_1,z_2,z_3): v_2^\vee(z_1,z_2,z_3): v_3^\vee(z_1,z_2,z_3) ].
    \end{aligned}
  \end{equation}
  The lemma follows by combining \eqref{equation:pencil-of-conics-dual} and \eqref{equation:action-of-g}.
\end{proof}

\begin{remark}
  Note that~$g([z_1:z_2:z_3])$ is not well defined when~$[z_1:z_2:z_3]$ is a basepoint of the net of conics induced by~$\varphi$. We will show in \cref{proposition:net-YX-is-basepoint-free} below that this situation cannot occur.
\end{remark}

From \cref{lemma:g-gives-pencil-of-conics} we get the following corollary, which is also contained in \cite[\S11.4.4]{MR3617981}.
\begin{corollary}
  \label{corollary:Yw-is-conic}
  With the notation as above as above, let~$\xi \in V^\vee$, let~$t_\xi \in \mathbb{P}(V^\vee)$ be the associated point and~$X_\xi \in \mathbb{P}(V)$ the associated line. Conversely let~$Y_\xi \in \mathbb{P}(E)$ be the conic defined by~$\xi$. Then~$g^{-1}(X_\xi)=Y_\xi$.
\end{corollary}

\begin{proof}
  Write~$\xi = a_1 v_1^\vee + a_2 v_2^\vee + a_3 v_3^\vee$, then~$X_w$ is the line in the coordinates~$[v_1:v_2:v_3]$ defined by~$a_1 v_1  + a_2 v_2  + a_3 v_3 = 0$. In particular
  \begin{equation}
    g^{-1}(X_\xi) = \{ [x_1:x_2:x_3] \mid a_1 v_1^\vee([x_1:x_2:x_3]) + a_2 v_2^\vee([x_1:x_2:x_3]) + a_3 v_3^\vee([x_1:x_2:x_3]) = 0 .
  \end{equation}
  This is equal to~$Y_\xi$.
\end{proof}
Now if~$g\colon Z\to T$ is an arbitrary finite morphism of smooth projective varieties, then by dualising the multiplication
\begin{equation}
  g_*\mathcal{O}_Z\otimes_{\mathcal{O}_T}g_*\mathcal{O}_Z\to g_*\mathcal{O}_Z
\end{equation}
we obtain a copairing
\begin{equation}
  g_*\omega_{Z/T}\to g_*\omega_{Z/T}\otimes_{\mathcal{O}_T}g_*\omega_{Z/T}.
\end{equation}
By tensoring it over~$g_*\mathcal{O}_Z$ on the left and right with~$g_*\omega_{Z/T}^{-1}$ we get an equivalent copairing
\begin{equation}
  \label{equation:copairing-sheaves}
  \delta_{Z/T}\colon g_*\omega_{Z/T}^{-1}\to g_*\mathcal{O}_Z\otimes_{\mathcal{O}_Z}g_*\mathcal{O}_Z.
\end{equation}
This will be the geometric incarnation of \eqref{equation:delta} in the special case we are interested in. This, and other identifications between sheaves and the corresponding graded modules are given in the following lemma.
\begin{lemma}
  \label{lemma:gOz-Seven-gOmegainv-Sodd}
  There exist isomorphisms
  \begin{equation}
    \label{equation:gOz-Seven}
    g_*\mathcal{O}_Z\cong\widetilde{S_{\mathrm{even}}},
  \end{equation}
  and
  \begin{equation}
    \label{equation:gOmegainv-Sodd}
    g_*\omega_{Z/T}^{-1}\cong\widetilde{S_{\mathrm{odd}}(-3)}\otimes_k A_3.
  \end{equation}
  Using these identifications the morphism \eqref{equation:copairing-sheaves} is the sheafification of the composition
  \begin{equation}
    \label{equation:sheafification}
    S_{\mathrm{odd}}(-3)\otimes A_3\hookrightarrow S(-3)\otimes A_3\overset{\delta}{\rightarrow}S\otimes_RS\twoheadrightarrow S_{\mathrm{even}}\otimes_RS_{\mathrm{even}}.
  \end{equation}
\end{lemma}

\begin{proof}
  The identification in \eqref{equation:gOz-Seven} follows from the fact that~$R$ is concentrated in even degrees.

  The identification in \eqref{equation:gOmegainv-Sodd} is obtained by restricting \eqref{equation:S-dual} to the even part, which gives
  \begin{equation}
    g_*\omega_{Z/T}=\widetilde{S_{\mathrm{odd}}(3)}\otimes A_3^{\otimes-1},
  \end{equation}
  and using
  \begin{equation}
    (S_{\mathrm{odd}}(-3)\otimes_k A_3)\otimes_{S_{\mathrm{even}}}(S_{\mathrm{odd}}(3)\otimes_kA_3^{-1})\cong S_{\mathrm{even}}
  \end{equation}
  we get the desired identification.

  Finally the identification in \eqref{equation:sheafification} follows from the explicit form of duality for a finite flat morphism.
\end{proof}

We will now define sheaves on~$T$, which will induce the sheaves~$\mathcal{F}$ and~$\mathcal{Q}$ on~$X$, as in \cref{subsection:symmetric-sheaves}. These are
\begin{equation}
  \begin{aligned}
    \overline{\mathcal{F}}&\coloneqq g_*\mathcal{O}_Z/\mathcal{O}_T, \\
    \overline{\mathcal{Q}}&\coloneqq\coker(\delta'\colon\omega_{Z/T}^{-1}\to\Sym^2\overline{\mathcal{F}})
  \end{aligned}
\end{equation}
where~$\delta'$ is the composition of~$\delta_{Z/T}$ with the projection to~$\Sym^2\overline{\mathcal{F}}$. Then we have the following result.
\begin{proposition}
  \label{proposition:Fbar-Qbar}
  There exist isomorphisms
  \begin{equation}
    \begin{aligned}
      \overline{\mathcal{F}}&\cong A_2\otimes_k\mathcal{O}_T(-1), \\
      \overline{\mathcal{Q}}&\cong\Omega_{T/k}^1\otimes_k A_3^{\otimes 2}.
    \end{aligned}
  \end{equation}
  Moreover we have that the quotient map
  \begin{equation}
    \Sym^2\overline{\mathcal{F}}\to\overline{\mathcal{Q}}
  \end{equation}
  is the composition of the sheafification of the morphism~$\beta$ from \eqref{equation:beta}
  \begin{equation}
    \Sym^2A_2\otimes_k\mathcal{O}_T(-2)\to V\otimes_k\mathcal{O}_T(-2)\otimes A_3^{\otimes 2}
  \end{equation}
  with the quotient map
  \begin{equation}
    V\otimes_k\mathcal{O}_T(-2)\otimes A_3^{\otimes2}\to\Omega_{T/k}^1\otimes A_3^{\otimes2}.
  \end{equation}
\end{proposition}

\begin{proof}
  This follows immediately from the construction together with \cref{proposition:coker-Phi} and \cref{lemma:gOz-Seven-gOmegainv-Sodd}.
\end{proof}

\subsection{The geometric comparison}
\label{subsection:geometric-comparison}
We now consider the situation of \cref{section:nc-P1-as-clifford}, i.e.~we start with a finite morphism~$f\colon Y\to X$ of degree~4, where~$X,Y\cong\mathbb{P}^1$. As in \cref{subsection:symmetric-sheaves} we define
\begin{equation}
  \begin{aligned}
    \mathcal{F}_{Y/X}&\coloneqq f_*\mathcal{O}_Y/\mathcal{O}_X \\
    \mathcal{Q}_{Y/X}&\coloneqq\coker(\omega_{Y/X}^{-1}\to\Sym^2\mathcal{F}_{Y/X}).
  \end{aligned}
\end{equation}
Then we define
\begin{equation}
  \begin{aligned}
    E_{Y/X}&\coloneqq\HH^0(X,\mathcal{F}_{Y/X}(1)) \\
    V_{Y/X}&\coloneqq\HH^0(X,\mathcal{Q}_{Y/X}(2)).
  \end{aligned}
\end{equation}
Using \cref{lemma:description-F,proposition:description-Q} we see that~$\dim_kE_{Y/X}=\dim_kV_{Y/X}=3$, and the quotient morphism~$\Sym^2\mathcal{F}_{Y/X}\to\mathcal{Q}_{Y/X}$ defines a net of conics
\begin{equation}
  \phi_{Y/X}\colon\Sym^2E_{Y/X}\twoheadrightarrow V_{Y/X}
\end{equation}
because the twist in the definition of~$E_{Y/X}$ and~$V_{Y/X}$ makes the higher sheaf cohomology vanish in the definition of~$\mathcal{Q}_{Y/X}$.

We now prove the following lemma, which will allow us to prove that the net of conics~$\phi_{Y/X}$ is basepoint-free. Recall from \cref{subsection:symmetric-sheaves} that~$\mathcal{F}_{Y/X}(1) = \mathcal{O}_{\mathbb{P}^1}^{\oplus 3}$ and~$f_* \omega^{-1}_{Y/X}(2) = \mathcal{O}_{\mathbb{P}^1}^{\oplus 3} \oplus \mathcal{O}_{\mathbb{P}^1}(-1)$. In particular there is a morphism
\begin{equation}
  \theta\colon f_* \omega^{-1}_{Y/X}(2) \rightarrow \mathcal{O}_{\mathbb{P}^1}(-1)
\end{equation}
which is unique up to scalar multiple.

If~$x$ is not a branch point of~$f$, then~$f^{-1}(x) = \{ y_1, y_2, y_3, y_4 \}$ and
\begin{equation}
  \label{equation:decomposition-yi}
  (f_* \omega^{-1}_{Y/X}(2))_x \otimes k(x) \cong (f_* \mathcal{O}_Y)_x \otimes k(x) \cong k(y_1) \oplus k(y_2) \oplus k(y_3) \oplus k(y_4)
\end{equation}

\begin{lemma} \label{lemma:comparison:theta-x}
  For a generic~$x\in X$ the map~$\theta_x\otimes k(x)$ does not vanish on any of the~$k(y_i)$ in \eqref{equation:decomposition-yi}.
\end{lemma}

\begin{proof}
  As in \cite[\S3.1]{1503.03992v4} we can restrict to an affine open~$\Spec C \subset X$ such that~$f^{-1}(\Spec C)=\Spec D$, $D/C$ is relative Frobenius of rank 4 and~$f_* \omega^{-1}_{Y/X}(2)|_{\Spec C} \cong f_* (\mathcal{O}_D)$. In particular~$\theta$ is given by a morphism~$D \rightarrow C$. Moreover as~$f$ has only finitely many branch points, we can choose~$C$ such that~$f^{-1}(x)$ consists of 4 points~$y_1,y_2,y_3,y_4$ for all~$x \in \Spec C$. We will now prove that the lemma holds for all but a finite number of points of~$\Spec C$.

  Note that if~$x$ is any point in~$\Spec C$, we can use \cite[lemma~3.1]{MR3565443} to conclude that
  \begin{equation}
    (f_* \mathcal{O}_Y)_x \otimes k(x) \cong D \otimes_C k(x) \cong k(y_1) \oplus k(y_2) \oplus k(y_3) \oplus k(y_4)
  \end{equation}
  is relative Frobenius of rank 4 over~$k(x)$. This implies that~$k(y_i) \cong k(x)$ for~$i=1,2,3,4$. Let~$e_{x,1}, e_{x,2}, e_{x,3}, e_{x,4}$ denote the idempotents in
  \begin{equation}
    k(y_1) \oplus k(y_2) \oplus k(y_3) \oplus k(y_4) \cong k(x)^{\oplus 4}
  \end{equation}
  We now need to prove (for generic~$x$) that~$(\theta_x \otimes k(x))(e_{x,i}) \neq 0$. For this consider the following commutative diagram
  \begin{equation}
    \label{equation:evaluation-diagram}
    \begin{tikzcd}
      \Hom_C(D,C) \arrow[r] \arrow[d, "\cong"] & \Hom_{k(x)}(D\otimes_Ck(x),k(x)) \arrow[d, "\cong"] \arrow[rr, "\mathrm{eval}(e_{x,i})"] & & k(x) \arrow[d, "\cong"] \\
      D \arrow[r, "-\otimes_Ck(x)"] & D\otimes_Ck(x) \arrow[r, "\cong"] & k(x)^{\oplus4} \arrow[r, "\pi_i"] & k(x)\cong D_{y_i}\otimes k(y_i).
    \end{tikzcd}
  \end{equation}
  Commutativity of the left square follows from the fact that the construction of relative Frobenius pairs is compatible with base change. Commutativity of the right square follows from the fact that the middle vertical isomorphism (or rather its inverse) is given by
  \begin{equation}
    k(x)^{\oplus 4} \xrightarrow{\cong} \Hom_{k(x)}(k(x)^{\oplus 4},k(x)): (a_1,a_2,a_3,a_4) \mapsto \left( (b_1,b_2,b_3,b_4) \mapsto \sum_{i=1}^4 a_ib_i \right)
  \end{equation}
  Now recall that every element in~$d \in D$ defines a function on~$\Spec D$ by sending~$y_i \in \Spec D$ to the image of~$d$ under the bottom horizontal composition in \eqref{equation:evaluation-diagram}. In particular~$\theta \in \Hom_C(D,C)$ defines an element in~$D$ through the left vertical isomorphism in \eqref{equation:evaluation-diagram} and hence a function~$\tilde{\theta}$ on~$\Spec D$. By construction we have
  \begin{equation}
    \tilde{\theta}(y_i) \neq 0 \ \Leftrightarrow \ (\theta_x \otimes k(x))(e_{x,i}) \neq 0
  \end{equation}
  The lemma now follows by noticing that~$\tilde{\theta}$ can only have finitely many zeroes on~$\Spec D$.
\end{proof}

\begin{proposition}
  \label{proposition:net-YX-is-basepoint-free}
  The net of conics~$\phi_{Y/X}$ is basepoint-free.
\end{proposition}

\begin{proof}
  Recall that~$\phi_{Y/X}$ was constructed via the following short exact sequence
  \begin{equation}
    0 \rightarrow \HH^0(X,f_*\omega^{-1}_{Y/X}(2)) \xrightarrow{i_{Y/X}} \Sym^2 (\HH^0(X,\mathcal{F}_{Y/X}(1))) \xrightarrow{\phi_{Y/X}} \HH^0(X,\mathcal{Q}_{Y/X}(2)) \rightarrow 0
  \end{equation}
  In particular it suffices to prove that
  \begin{equation}
    \label{equation:look-at-e-cdot-e}
    \im(i_{Y/X}) \cap \{ e \cdot e \mid e \in \HH^0(X,\mathcal{F}_{Y/X}(1)) \} = \{0 \}.
  \end{equation}
  To see this, note that~$\im(i_{Y/X}) = \ker(\phi_{Y/X})$ and
  \begin{equation}
    \begin{aligned}
      \ker(\phi_{Y/X})
      & = \left\{ \sum_{1=i \leq j}^3 \beta_{i,j} x_ix_j \mid \phi \left(\sum_{1=i \leq j}^3 \beta_{i,j} x_ix_j \right) = 0 \right\} \\
      & = \left\{ \sum_{1=i \leq j}^3 \beta_{i,j} x_ix_j \mid \sum_{1=i \leq j}^3 \alpha_{n,i,j} \beta_{i,j} = 0 \textrm{ for $n=1,2,3$} \right\} \\
      & = \left\{ \sum_{1=i \leq j}^3 \beta_{i,j} x_ix_j \mid v_n^\vee((\beta_{i,j})_{i,j}) = 0 \textrm{ for $n=1,2,3$} \right\}.
    \end{aligned}
  \end{equation}
  where the~$x_i$ and~$\alpha_{i,j}$ are as in \cref{lemma:g-gives-pencil-of-conics}. Moreover as in \eqref{equation:g-factor-trough-veronese}, we interpret~$v_1^\vee, v_2^\vee, v_3^\vee$ as functions on~$\mathbb{P}(\Sym^2 E) \cong \mathbb{P}^5$. The condition that the net spanned by~$v_1^\vee, v_2^\vee, v_3^\vee$ is basepoint-free is equivalent to requiring that the points~$[\beta_{1,1}: \beta_{2,2}: \beta_{3,3}: \beta_{1,2}: \beta_{1,3}: \beta_{2,3}]$ for which~$\sum_{1=i \leq j}^3 \beta_{i,j} x_ix_j \in \ker(\phi_{Y/X}))$ do not lie in the image of the Veronese embedding~$\mathbb{P}(E) \cong \mathbb{P}^2 \hookrightarrow \mathbb{P}(\Sym^2 E)$. And finally the latter is equivalent to~$\sum_{1=i \leq j}^3 \beta_{i,j} x_ix_j$ not lying in the image of
  \begin{equation}
    E \rightarrow \Sym^2 E: \sum_{i=1}^3 \gamma_i x_i \mapsto \left(\gamma_i x_i\right)^2 = \sum_{1=i \leq j}^3 \gamma_i \gamma_j x_ix_j .
  \end{equation}
  In particular in order to prove the result, it suffices to prove \eqref{equation:look-at-e-cdot-e}.

  Let~$j_{Y/X}\colon \ker(\theta) \rightarrow f_*\omega^{-1}_{Y/X}(2) \rightarrow \Sym^2 (\mathcal{F}_{Y/X}(1))$ be the induced composition. Then, using the fact that~$\mathcal{O}_{\mathbb{P}^1}(-1)$ has no global sections, we see that~$\HH^0(X,j_{Y/X}) = i_{Y/X}$. Moreover as both~$\ker(\theta)$ and~$\Sym^2 (\mathcal{F}_{Y/X}(1))$ are given by direct sums of copies of the structure sheaf we have for each point~$x \in X$ that
  \begin{equation}
    j_{Y/X,x} \otimes_{\mathcal{O}_{X,x}} k(x) = \HH^0(X,j_{Y/X}) \otimes_k k(x)= i_{Y/X} \otimes_k k(x)
  \end{equation}
  and
  \begin{equation}
    \{ (e \cdot e) \otimes_k 1_{k(x)} \mid e \in \HH^0(X,\mathcal{F}_{Y/X}(1)) \} = \{ e_x \cdot e_x \mid e_x \in \mathcal{F}_x\otimes k(x) \}.
  \end{equation}
  As such it suffices to prove that
  \begin{equation}
    \im \big(j_{Y/X,x} \otimes_{\mathcal{O}_{X,x}} k(x)\big) \cap \{ (e_x \cdot e_x) \otimes_{\mathcal{O}_{X,p}} 1_{k(x)} \mid e_x \in \mathcal{F}_x \} = \{0 \}
  \end{equation}
  holds for some point~$x \in X$. For this let~$x$ be a generic point as in \cref{lemma:comparison:theta-x}

  Then as in the proof of \cref{lemma:HYX-symmetric} (using the fact that the~$e_i$ form a self-dual basis (\cref{rem:dual-bases})) we see that~$j_{Y/X,x} \otimes_{\mathcal{O}_{X,x}} k(x)$ is given by composing the inclusion~$\ker(\theta_x) \otimes_{\mathcal{O}_{X,x}} k(x) \subset k(x)^{\oplus 4}$ with
  \[ k(x)^{\oplus 4} \rightarrow \Sym^2(k(x)^{\oplus 4}) \rightarrow \Sym^2((k(x)^{\oplus 4})/k(x)): e_i \mapsto e_i \cdot e_i \]
  Note that~$\sum_{i,j=1, i \leq j}^4 a_{i,j} \frac{e_i \cdot e_j + e_j \cdot e_i}{2}$ is of the form~$e \cdot e$ if and only if~$\alpha_{i,i} \alpha_{j,j} = 4 \alpha_{i,j}^2$ holds for all~$i < j$. Such an element lies in the image of~$j_{Y/X,x} \otimes_{\mathcal{O}_{X,x}} k(x)$ if and only if there is 1~$i$ such that~$\alpha_{j,k} = 0$ whenever~$(j,k) \neq (i,i)$. Hence it suffices to show that~$\ker(\theta_x) \otimes_{\mathcal{O}_{X,x}} k(x)$ does not contain an element of the form~$\lambda e_i$. This is however guaranteed by the fact that~$\theta_x \otimes_{\mathcal{O}_{X,x}} k(x)$ does not vanish on any of the~$k(y_i)$ in \eqref{equation:decomposition-yi}.
\end{proof}

One final piece of information obtained from~$f\colon Y\to X$ in this setting is a character~$\xi_{Y/X}\colon V_{Y/X}\to k$, up to scalars, by taking the global sections of the unique non-zero projection
\begin{equation}
  \mathcal{Q}_{Y/X}(2)\cong\mathcal{O}_X\oplus\mathcal{O}_X(1)\to\mathcal{O}_X.
\end{equation}

\begin{definition}
  An \emph{algebraic quadruple} is a tuple~$(E,V,\phi,\xi)$, where~$E$ and~$V$ are vector spaces of dimension~3, $\phi\colon\Sym^2E\to V$ is a surjective morphism defining a basepoint-free net of conics, and~$\xi\colon V\to k$ is a non-zero morphism.
\end{definition}
The procedure described above is a way of getting an algebraic quadruple from a morphism~$f\colon\mathbb{P}^1\to\mathbb{P}^1$ of degree~4, by taking~$(E_{Y/X},V_{Y/X},\phi_{Y/X},\xi_{Y/X})$. The main result of this section is that there exists a correspondence between such morphisms~$f$ and these algebraic quadruples.

From an algebraic quadruple we can define a morphism~$f\colon\mathbb{P}^1\to\mathbb{P}^1$ of degree~4. For this we consider the morphism~$g\colon Z\to T$ defined in \eqref{equation:morphism-g}. From~$\xi\in\mathbb{P}(V^\vee)$ we obtain a line~$X$ inside~$\mathbb{P}(V)$. Then define~$f\colon Y\to X$ using the fibre product
\begin{equation}
  \begin{tikzcd}
    Y \arrow[d, "f"] \arrow[r, hook] & Z \arrow[d, "g"] \\
    X \arrow[r, hook, "i"] & T.
  \end{tikzcd}
\end{equation}
As explained in \cref{corollary:Yw-is-conic}, we can identify~$Y$ with the conic parametrised by the point~$\xi\in\mathbb{P}(V^\vee)$ in the net of conics. Hence if~$\xi$ is taken outside the discriminant locus in~$\mathbb{P}(V^\vee)$, then~$Y$ is a smooth conic, hence we get a morphism~$f\colon\mathbb{P}^1\to\mathbb{P}^1$ of degree~4. The following theorem shows that the algebraic quadruple associated to this morphism is isomorphic to the original algebraic quadruple.

\begin{theorem}
  \label{theorem:quadruples-are-isomorphic}
  Let~$(E,V,\phi,\xi)$ be an algebraic quadruple such that~$\xi$ defines a point outside the discriminant locus of the net of conics~$\phi$. Then
  \begin{equation}
    (E,V,\phi,\xi)\cong(E_{Y/X},V_{Y/X},\phi_{Y/X},\xi_{Y/X}).
  \end{equation}
\end{theorem}

\begin{proof}
  Using \cref{lemma:description-F,proposition:description-Q,proposition:Fbar-Qbar} and the fact that~$X$ is a curve of degree~1 inside~$T$ we have a commutative diagram
  \begin{equation}
    \label{equation:commutative-diagram-phi-YX}
    \begin{tikzcd}
      \Sym^2\mathcal{F}_{Y/X} \arrow[r, two heads] \arrow[d, equals] & \mathcal{Q}_{Y/X} \arrow[r, "\cong"] \arrow[d, equals] & \mathcal{O}_X(-2)\oplus\mathcal{O}_X(-1) \\
      \Sym^2(i^*\overline{\mathcal{F}}) \arrow[r, two heads] & i^*\overline{\mathcal{Q}} \arrow[r, "\cong"] & i^*\Omega_{T/k}^1 \arrow[u]
    \end{tikzcd}
  \end{equation}
  Because~$\Omega_{T/k}^1(2)$ sits inside the twist of the Euler exact sequence
  \begin{equation}
    0\to\mathcal{O}_T(-1)\overset{\gamma}{\to}V\otimes_k\mathcal{O}_T\to\Omega_{T/k}^1(2)\to 0
  \end{equation}
  we get that~$\HH^0(\Omega_{T/k}^1(2))=V$. From this we get isomorphisms
  \begin{equation}
      V \cong V_{Y/X}
  \end{equation}
  and
  \begin{equation}
    \begin{aligned}
      E_{Y/X}
      &=\HH^0(\mathcal{F}_{Y/X}(1)) \\
      &\cong\HH^0(i^*(\overline{\mathcal{F}}(1))) & \text{by \eqref{equation:commutative-diagram-phi-YX}} \\
      &\cong\HH^0(i^*(A_2 \otimes_k \mathcal{O}_T)) & \textrm{\cref{proposition:Fbar-Qbar}} \\
      &\cong\HH^0(A_2 \otimes_k \mathcal{O}_X) \\
      &\cong A_2 \\
      &\cong E & \text{by \eqref{equation:comparison:AiE1} \textrm{ and  }\eqref{equation:comparison:AiE2}}
     \end{aligned}
  \end{equation}
  We also get that~$\phi=\phi_{Y/X}$ by \eqref{equation:commutative-diagram-phi-YX}.

  Finally to show that~$\xi_{Y/X}=\xi$ we need to show that~$\xi$ equals the composition
  \begin{equation}
    \begin{tikzcd}
      V=\HH^0(T,\Omega_{T/k}^1(2)) \arrow[r] & \HH^0(X,i^*\Omega_{T/k}^1(2)) \arrow[d, equals] \\
      & \HH^0(X,\mathcal{O}_X\oplus\mathcal{O}_X(1)) \arrow[r] & \HH^0(X,\mathcal{O}_X)=k.
    \end{tikzcd}
  \end{equation}
  The fiber in~$p\in\mathbb{P}(V)$ of the inclusion is (up to scalar) the morphism~$p\colon k\to V$. If we pull back the short exact sequence along the closed immersion~$i$ we see that the image of the fibers of all pullbacks are contained in~$\ker\xi$, hence we have a complex
  \begin{equation}
    \mathcal{O}_X(-1)\xrightarrow{i^*\gamma} V\otimes_k\mathcal{O}_X\xrightarrow{\xi\otimes\id_{\mathcal{O}_X}}\mathcal{O}_X.
  \end{equation}
  Up to scalars this must be the unique non-zero morphism
  \begin{equation}
    i^*\Omega_{T/k}^1\to\mathcal{O}_X.
  \end{equation}
  By taking global sections we see that~$\xi_{Y/X}=\xi$.
\end{proof}

Unfortunately, we do not have a complete correspondence between the two pieces of geometric data. It is not clear that we can start from~$f\colon\mathbb{P}^1\to\mathbb{P}^1$, produce an algebraic quadruple using geometric arguments, such that the associated morphism is the original~$f$ (up to automorphisms of~$\mathbb{P}^1$). We phrase this as a conjecture.
\begin{conjecture}
  \label{conjecture:inverse}
  Let~$f=(f_1,f_2)$ be a basepoint-free pencil of binary quartics. Then there exists an algebraic quadruple~$(E,V,\phi,\xi)$ such that~$f$ is the associated morphism~$\mathbb{P}^1\to\mathbb{P}^1$.
\end{conjecture}
Luckily, this does not prevent us from completing the comparison between the associated categories, which we perform in \cref{subsection:categorical-comparison}. The geometric correspondence would also follow from a good understanding of the Morita theory of the noncommutative~$\mathbb{P}^1$\dash bundles, but this does not seem within reach.

Moreover, in the classification of the pencils of binary quartics there are~13~types, and for each of these types we can construct an algebraic quadruple giving rise to this pencil of binary quartics. And whenever there are degrees of freedom within a certain class, the degrees of freedom agree with the degrees of freedom in algebraic quadruples of the corresponding type.

\paragraph{Explicit correspondence}
We will now relate the classification of pencils of binary quartics from \cite{MR1636607} to algebraic quadruples, i.e.~to the classification of nets of conics from \cite{MR0432666} and the choice of a smooth conic inside the net. We do this by giving for every element in the classification of pencils of binary quartics the discriminant of the net of conics and the position of the point which is blown up, corresponding to the choice of the smooth conic inside the net.

To understand which type of~$f$ we obtain in from an algebraic quadruple, we only need to describe what the inverse image of each point looks like, i.e.~what the ramification behaviour is. For this, it suffices to understand what the possible intersections of lines from the pencil of lines with the discriminant cubic are, because there is a correspondence between base loci of pencils of conics and the structure of their singular fibres.

\begin{enumerate}
  \item The Segre symbol~$[1,1,1]$ is the generic case, and is not relevant in the construction. It corresponds to a generic line intersecting the discriminant cubic in~3~points, i.e.~a pencil of conics with~3 singular fibres.
  \item The Segre symbol~$[2,1]$ corresponds to ramification of type~$(2,1,1)$, and to realise it we must have a tangent line to the discriminant cubic intersecting the discriminant cubic in another (smooth) point: we have two singular fibres of rank~2, one of which is counted twice.
  \item The Segre symbol~$[3]$ corresponds to ramification of type~$(3,1)$, and to realise it we must have a tangent at an inflection point: we have one singular fibre of rank~2 which is counted thrice.
  \item The Segre symbol~$[(1,1),1]$ corresponds to ramification of type~$(2,2)$, and to realise it we must have a line through a node intersecting the discriminant cubic in another point.
  \item The Segre symbol~$[(2,1)]$ corresponds to ramification of type~$(4)$, and to realise it we must have a line through a node intersecting the discriminant only in this node.
\end{enumerate}

Using this it is straightforward to set up the correspondence given in \cref{table:case-by-case}. It only remains to observe that on a nodal cubic there are~3~inflections points whose tangent lines are not concurrent, and that on an elliptic curve in~$\mathbb{P}^2$ with~$j=0$ there exist sets of~3~inflection points whose tangent lines \emph{are} concurrent. If~$j\neq 0$ then this is not possible (and recall that there are~9~inflection points in total). This follows immediately from the proof of \cite[lemma~1]{MR2583779}, where the inflection tangents are concurrent if~$c=0$, which gives the elliptic curve defined by~$x^3+y^3+z^3$, which indeed has~$j$\dash invariant~0.

\begin{table}[p]
  \centering
  \small
  \begin{tabular}{cllc}
    \toprule
    label         & normal form                                                                                                      & discriminant cubic    & moduli \\
                  & location of the point \\
    \midrule
    A             & $[(4)(4)]$                                                                                                       & conic and line        & 0 \\
                  & \multicolumn{3}{p{10cm}}{intersection of tangent lines at conic in singularities} \\
    \addlinespace
    B             & $[(4)(2,2)(2,1,1)]$                                                                                              & conic and line        & 0 \\
                  & \multicolumn{3}{p{10cm}}{intersection of tangent line at conic in singularity and generic line} \\
    \addlinespace
    C             & $[(4)(3,1)(2,1,1)]$                                                                                              & nodal cubic           & 0 \\
                  & \multicolumn{3}{p{10cm}}{intersection of tangent line at singularity and tangent line of inflection point} \\
    \addlinespace
    D             & $[(4)(2,1,1)(2,1,1)(2,1,1)]$                                                                                     & nodal cubic           & 1 \\
                  & \multicolumn{3}{p{10cm}}{intersection of tangent line at singularity and tangent line at non-inflection point} \\
    \addlinespace
    E             & $[(2,2)(2,2)(2,2)]$                                                                                              & triangle of lines     & 0 \\
                  & \multicolumn{3}{p{10cm}}{any point} \\
    \addlinespace
    F             & $[(2,2)(2,2)(2,1,1)(2,1,1)$                                                                                      & conic and line        & 1 \\
                  & \multicolumn{3}{p{10cm}}{intersection of tangent lines at two smooth points on conic} \\
    \addlinespace
    G             & $[(2,2)(3,1)(3,1)]$                                                                                              & nodal cubic           & 0 \\
                  & \multicolumn{3}{p{10cm}}{intersection of tangent lines at inflection points} \\
    \addlinespace
    H             & $[(2,2)(3,1)(2,1,1)(2,1,1)]$                                                                                     & nodal cubic           & 1 \\
                  & \multicolumn{3}{p{10cm}}{intersection of non-tangent line through singularity and tangent line at inflection point} \\
    \addlinespace
    I             & $[(3,1)(3,1)(3,1)]$                                                                                              & elliptic curve, $j=0$ & 0 \\
                  & \multicolumn{3}{p{10cm}}{intersection of three concurrent tangent lines at inflection points} \\
    \addlinespace
    J             & $[(3,1)(3,1)(2,1,1)(2,1,1)]$                                                                                     & elliptic curve        & 1 \\
                  & \multicolumn{3}{p{10cm}}{generic intersection of two tangent lines at inflection points} \\
    \addlinespace
    K             & $[(2,2)(2,1,1)(2,1,1)(2,1,1)(2,1,1)]$                                                                            & nodal cubic           & 2 \\
                  & \multicolumn{3}{p{10cm}}{intersection of two generic tangent lines} \\
    \addlinespace
    L             & $[(3,1)(2,1,1)(2,1,1)(2,1,1)(2,1,1)]$                                                                            & elliptic curve        & 2 \\
                  & \multicolumn{3}{p{10cm}}{generic intersection of tangent line at inflection point and another tangent line} \\
    \addlinespace
    M             & $[(2,1,1)(2,1,1)(2,1,1)(2,1,1)(2,1,1)(2,1,1)]$                                                                   & elliptic curve        & 3 \\
                  & \multicolumn{3}{p{10cm}}{generic point} \\
    \bottomrule
  \end{tabular}
  \caption{Correspondence between pencils of binary quartics and algebraic quadruples}
  \label{table:case-by-case}
\end{table}

\subsection{The categorical comparison}
\label{subsection:categorical-comparison}
The categorical comparison would follow from \cref{theorem:quadruples-are-isomorphic} and a positive answer to \cref{conjecture:inverse}. Lacking this, we give an algebraic proof for the fact that a noncommutative~$\mathbb{P}^1$\dash bundle gives rise to an algebraic quadruple. The only thing left to show for this is that the center of the blowup is outside the discriminant. From the description in \cref{section:nc-P1-as-clifford} and the computation in \cite[remark~34]{MR3847233} it suffices to show that the global dimension of the associated abelian category is finite.

\begin{proposition}
  \label{proposition:gldim-finite}
  The global dimension of~$\qgr(\mathbb{S}({}_f(\mathcal{O}_{\mathbb{P}^1})_{\id}))$ is finite.

  \begin{proof}
    By \cite[lemma~4.6]{1503.03992v4} we have that~$\Pi_m^*$ is an exact functor, whilst~\cite[lemma~5.3]{1503.03992v4} shows that~$\Pi_{m,*}$ has cohomological dimension~1. The Grothendieck spectral sequence for~$\Hom$ in~$\qgr(\mathbb{S}({}_f(\mathcal{O}_{\mathbb{P}^1})_{\id}))$ and~$\Pi_{m,*}$ then shows that the cohomological dimension of~$\Pi_m^*(\mathcal{O}_{\mathbb{P}^1}(i))$ is~2, i.e.~that~$\Ext^n(\Pi_m^*(\mathcal{O}_{\mathbb{P}^1}(i)),\mathcal{M})=0$ for~$n\geq 3$, using \cite[lemma~4.7]{1503.03992v4}.

    From \cite[theorem~5.1, theorem~5.5]{1503.03992v4} we have that there exists a full and strong exceptional collection in~$\derived^\bounded(\qgr(\mathbb{S}({}_f(\mathcal{O}_{\mathbb{P}^1})_{\id}))$ consisting of objects of cohomological dimension~2. Because the global dimension of this endomorphism algebra is finite (as it is the path algebra of an acyclic quiver modulo an ideal of relations), and the cohomological dimension of the functor realising the equivalence is~2, we can compute~$\Ext^n$ in~$\qgr$ using uniformly bounded complexes in~$\derived^\bounded(kQ/I)$, and we can conclude that the global dimension is indeed finite.
  \end{proof}
\end{proposition}

\begin{remark}
  \label{remark:gldim-2}
  It would be interesting to give a direct proof that the global dimension of this surface-like category is precisely~2. We expect that this can be deduced from \cite[theorem~6.1]{MR3565443} using an appropriate analogue of the local-to-global spectral sequence for Ext.

  Observe that from \cref{theorem:categorical-comparison} we can a posteriori conclude that the global dimension is indeed~2, as the global dimension of the equivalent category constructed as a blowup is~2 by \cref{proposition:gldim-2}.
\end{remark}

\begin{corollary}
  \label{corollary:center-in-azumaya}
  The center of the blowup in \eqref{equation:final-description-as-blowup} is contained in the Azumaya locus of the maximal order on~$\mathbb{P}^2$.

  \begin{proof}
    Assume on the contrary that~$x$ is contained in the discriminant. Then blowing up at~$x$ yields a sheaf of algebras on~$\Bl_x\mathbb{P}^2$ which is not of finite global dimension (and which is not a maximal order), using the local computation from \cite[remark~34]{MR3847233}. But this contradicts \cref{proposition:gldim-finite}, as these two categories are equivalent.
  \end{proof}
\end{corollary}
This does \emph{not} imply \cref{conjecture:inverse}: there could be non-equivalent~$f$ giving rise to an equivalent noncommutative~$\mathbb{P}^1$\dash bundle, and only one of these corresponds to an algebraic quadruple. In other words: it is not clear that there exists a surjection from pencils of binary quartics to algebraic quadruples, but if we only care about these objects up to equivalence of the associated abelian categories, the comparison is complete.

We are now in the position to state and prove the final comparison theorem, giving a correspondence between noncommutative~$\mathbb{P}^1$\dash bundles of rank~$(4,1)$ and blowups of quaternionic noncommutative planes outside the discriminant. In this way, we obtain a purely noncommutative incarnation of the isomorphism~$\Bl_x\mathbb{P}^2\cong\mathbb{F}_1$.

\begin{theorem}
  \label{theorem:categorical-comparison}
  Every noncommutative~$\mathbb{P}^1$\dash bundle~$\qgr(\mathbb{S}({}_f (\mathcal{O}_{\mathbb{P}^1})_{\id}))$ is equivalent to~$\coh p^*\mathcal{S}$ for a quaternionic order~$p^*\mathcal{S}$ on~$\Bl_x\mathbb{P}^2$ and vice versa.

  \begin{proof}
    First let~$\qgr(\mathbb{S}({}_f (\mathcal{O}_{\mathbb{P}^1})_{\id}))$ be a noncommutative~$\mathbb{P}^1$\dash bundle. Using the techniques in \cref{section:nc-P1-as-clifford} we know that~$\qgr(\mathbb{S}({}_f (\mathcal{O}_{\mathbb{P}^1})_{\id}))$ is equivalent to a category of the form
    \begin{equation}
     \qgr\clifford_{\Sym(\mathcal{O}_{\mathbb{P}^1} \oplus \mathcal{O}_{\mathbb{P}^1}(1))}\big(E \otimes_k \Sym(\mathcal{O}_{\mathbb{P}^1} \oplus \mathcal{O}_{\mathbb{P}^1}(1)), q\big)
    \end{equation}
    This category is completely determined by a quadruple~$(E,V,\phi,\xi)$. Using \cref{proposition:net-YX-is-basepoint-free}, we know that this quadruple is algebraic. In particular the category is equivalent to one of the form~$\coh p^* \mathcal{S}$. By \cref{corollary:center-in-azumaya} the point (determined by~$\xi$) in~$\mathbb{P}(V^\vee) \cong \mathbb{P}^2$ blown up by~$p\colon \mathbb{F}_1 \rightarrow \mathbb{P}^2$ lies in the Azumaya locus of~$\mathcal{S}$.

    Conversely, let~$\coh p^* \mathcal{S}$ be the category associated to a quaternionic order on~$\mathbb{F_1}$. The techniques in \cref{section:blowup-as-clifford} tell us that this category is equivalent to a category of the form
    \begin{equation}
     \qgr\clifford_{\Sym(\mathcal{O}_{\mathbb{P}^1} \oplus \mathcal{O}_{\mathbb{P}^1}(1))}\big(E \otimes_k \Sym(\mathcal{O}_{\mathbb{P}^1} \oplus \mathcal{O}_{\mathbb{P}^1}(1)), q\big),
    \end{equation}
    which again is determined by a quadruple~$(E,V,\phi,\xi)$. Using \cite[proposition~7]{MR1356364}, we know that this quadruple is algebraic and by construction~$\xi$ determines a point in the Azumaya locus of~$\mathcal{S}$. Using \cite[proposition~8]{MR1356364}, we get that this point determines a nonsingular conic in the net~$V^\vee$. We can then use \cref{theorem:quadruples-are-isomorphic}, to find a noncommutative~$\mathbb{P}^1$\dash bundle~$\qgr(\mathbb{S}({}_f (\mathcal{O}_{\mathbb{P}^1})_{\id}))$ equivalent~to $\coh p^* \mathcal{S}$.

  \end{proof}
\end{theorem}

\subsection{Final remarks}
\label{subsection:final-remarks}
\paragraph{Blowing up a point on the discriminant} {\ } \\
There is an important difference between the construction of \cite{MR3847233} and the construction of a blowup in \cite{MR1846352}. In the latter a point on the point scheme on an arbitrary noncommutative surface is constructed resulting in a noncommutative~$\Bl_x\mathbb{P}^2$, whereas the former only works for sheaves of orders where a point outside (an isogeny of) the point scheme is blown up. Hence these situations have no overlap. In other terms, the latter construction gives a deformation of the commutative~$\Bl_x\mathbb{P}^2$, whereas there is no commutative counterpart for the former.

Likewise there is an important difference between the notion of a noncommutative~$\mathbb{P}^1$\dash bundle of type~$(1,4)$ \cite{1503.03992v4} and that of a noncommutative~$\mathbb{P}^1$\dash bundle of type~$(1,4)$ \cite{MR2958936}. The latter can be used to construct the noncommutative Hirzebruch surface~$\mathbb{F}_1$. It is expected that this construction gives the same surfaces as the blowup from \cite{MR1846352}, but no proof has been found so far.

We can sum up the situation in the following table, where the rows give the same noncommutative surfaces (albeit conjecturally for the first row), whilst the columns are constructions of an analogous nature. This paper shows that the bottom row gives isomorphic noncommutative surfaces.

\begin{table}[H]
  \centering
  \begin{tabular}{cc}
    noncommutative~$\mathbb{P}^1$\dash bundle of type~$(2,2)$ & abstract blowup on point scheme \\
    \midrule
    noncommutative~$\mathbb{P}^1$\dash bundle of type~$(1,4)$ & pullback along blowup outside ramification
  \end{tabular}
  \caption{Four constructions of a noncommutative $\mathbb{F}_1=\Bl_x\mathbb{P}^2$}
  \label{table:4-constructions}
\end{table}

It is possible to degenerate the situation in the bottom row. If one were to blow up a point on the ramification of the maximal order on~$\mathbb{P}^2$ and pull back the sheaf of algebras along this morphism in the setting of \cref{section:blowup-as-clifford}, the resulting sheaf of orders is no longer a maximal order, nor does it have finite global dimension. But there exist two maximal orders in which the pullback can be embedded\footnote{At least when the point which is blown up is a smooth point on the ramification. If one blows up a singular point, there is only~1~preimage under the isogeny and the maximal order is unique.}, and these sheaves of algebras describe the construction of \cite{MR1846352} in the special case where everything is finite over the center. These two maximal orders correspond to the two inverse images under the isogeny~$C\to C/\tau$, where~$C$ is the point scheme and~$C/\tau$ the ramification.

We suggest that there exists a construction of a noncommutative~$\mathbb{P}^1$\dash bundle which does not use a finite morphism~$f\colon\mathbb{P}^1\to\mathbb{P}^1$ of degree~4, but rather a finite morphism~$f\colon C\to\mathbb{P}^1$ of degree~4 where~$C$ is the \emph{singular} conic in the net of conics associated to the point which is being blown up. Then the formalism from \cite{MR3565443} yields an abelian category of infinite global dimension, which should be the same as the category associated to the non-maximal order of infinite global dimension by methods similar to the ones used in this paper.

Now the choice of a maximal order (of finite global dimension) containing the pullback gives the construction from the top row in \cref{table:4-constructions}. In the construction of the noncommutative~$\mathbb{P}^1$\dash bundle this seems to correspond to a resolution of singularities of the singular conic\footnote{If~$C$ is a double line there will be only one morphism~$\mathbb{P}^1\to\mathbb{P}^1$ of degree~2, which should correspond to there only being one choice of maximal order.}~$C$ degenerating the degree~4 morphism~$f\colon C\to\mathbb{P}^1$ to two degree~2 morphisms~$f_i\colon\mathbb{P}^1\to\mathbb{P}^1$. These then degenerate the noncommutative~$\mathbb{P}^1$\dash bundle of type~$(1,4)$ to two different noncommutative~$\mathbb{P}^1$\dash bundles of type~$(2,2)$. We leave the details for this comparison for future work.

\paragraph{Hochschild cohomology} {\ } \\
In \cite{MR2183254,MR2238922} a deformation theory of abelian ategories was introduced, which is controlled by their Hochschild cohomology~$\HHHH_{\mathrm{ab}}^\bullet$. This turns out to be a derived invariant, so we can use the finite-dimensional algebra obtained through the full and strong exceptional collection to compute it.

In \cite{1705.06098v1} the first author computed the Hochschild cohomology of noncommutative planes and quadrics.  Computer algebra calculations suggest that the result for the noncommutative surfaces considered in this paper is less varied than it is in the situation of op.~cit., where a wide variety of situations appears. Recall that the method to compute Hochschild cohomology for noncommutative planes and quadrics consists of having an explicit computation for~$\HHHH_{\mathrm{ab}}^1$, a vanishing result for~$\HHHH_{\mathrm{ab}}^i$ with~$i\geq 3$ and the Lefschetz trace formula for Hochschild cohomology which then determines the dimension of~$\HHHH_{\mathrm{ab}}^2$ in terms of the number of exceptional objects in the derived category.

For the noncommutative del Pezzo surfaces studied in this paper, only the noncommutative~$\mathbb{P}^1$\dash bundle associated to the pencil of type~A (see also \cref{example:type-A-pencil-of-quartics}) has~$\dim_k\HHHH_{\mathrm{ab}}^1(\qgr\mathbb{S}({}_f\mathcal{L}_{\id}))=1$, all other cases have~$\HHHH^1=0$. Because there is a full and strong exceptional objects of length~4\ in the derived category this means that the noncommutative~$\mathbb{P}^1$\dash bundle associated to the pencil of type~A has~$\dim_k\HHHH_{\mathrm{ab}}^2=4$, whilst all other cases have~$\dim_k\HHHH_{\mathrm{ab}}^2=3$.

As it turns out, type~A is also the only pencil of binary quartics which has a~1\dash dimensional symmetry group, all other types have a finite (and often trivial) automorphism group \cite[\S4]{MR1636607}. One way of seeing this is that the morphism~$f$ of type~A is only ramified over two points, and the subgroup of~$\PGL_2$ fixing these two points is isomorphic to~$\Gm$. The only ingredient missing to mimick the method of \cite{1705.06098v1} is a good understanding of the relationship between automorphisms of the pencil and outer automorphisms of the finite-dimensional algebra obtained from the full and strong exceptional collection to confirm these numerical results.

We can also (heuristically) understand the~1\dash dimensional first Hochschild cohomology space for type~A pencils of binary quartics from the blowup model. Using \cite{1705.06098v1} we have a description of the automorphisms of the elliptic triples. The only elliptic triples we need to consider are those whose automorphism group is positive-dimensional, hence the point scheme is either a conic and a line, or a triangle of lines. Now we are interested in those automorphisms of~$C$ (or rather\footnote{Recall that the discriminant is an isogeny of the point scheme, and is a cubic curve of the same type.} the isogeneous~$C/\sigma$) which preserve the point we wish to blow up. In type~A of \cref{table:case-by-case} we see that the automorphisms of~$C$ preserve the point, hence we get a~1\dash dimensional automorphism group. For the other relevant types, namely those for which the first Hochschild cohomology of the associated noncommutative~$\mathbb{P}^2$ is nonzero (i.e.~types~B, E and~F) we get that at most finitely many automorphisms of the curve also preserve the point we wish to blow up. This heuristic works more generally for all cases in the construction of \cite{MR3847233}, not just when~$n=2$.

\newpage

\bibliographystyle{plain}
\bibliography{bibliography-clean,other}

\end{document}